\documentclass[12pt,a4,reqno]{amsart}
\usepackage{latexsym,amssymb,amsmath,amsfonts,mathtools} 
\usepackage[T1]{fontenc}
\usepackage{dirtytalk}
\usepackage[english]{babel}
\usepackage{xcolor}
\usepackage{subcaption}

\usepackage{dsfont} 
\usepackage[breaklinks=true,colorlinks,pdfpagelabels]{hyperref}

\def\sloppy{\tolerance=9999 \hfuzz=.5pt \vfuzz=.5pt}
\usepackage{todonotes}

\usepackage{tikz}
\usepackage{caption}
\usepackage{xypic}

\newcommand{\demo}{\noindent\textbf{Proof. }} 
\newcommand{\fimdemo}{\hspace*{\fill}\rule{0.3cm}{0.3cm}} 

\newcommand{\R}{\mathbb{R}}

\newcommand{\ep}{\varepsilon}

\newcommand{\al}{\alpha}

\def\noi{\noindent}
\def\ds{\displaystyle}

\newtheorem{theorem}{Theorem}[section]

\newtheorem{definicao}[theorem]{Definition}
\newtheorem{proposition}[theorem]{Proposition}

\newtheorem{lemma}[theorem]{Lemma}

\newtheorem{corolario}[theorem]{Corollary}
\newtheorem{corollary}[theorem]{Corollary}
\newtheorem{remark}[theorem]{Remark}

\def\noi{\noindent}

\def \ds {\displaystyle}

\def\noi{\noindent}
\def\ds{\displaystyle}

\def\sloppy{\tolerance=9999 \hfuzz=.5pt \vfuzz=.5pt}
\usepackage{overpic}
\usepackage{hyperref}
\usepackage{cleveref}  

\newcounter{hypo}
\renewcommand{\thehypo}{H\arabic{hypo}}
\newenvironment{hypothesis}[1][]{
    \refstepcounter{hypo}
    \begin{itemize}
        \item[(\thehypo)] \label{#1}
}{\end{itemize}}

\newcounter{hypoA}
\renewcommand{\thehypoA}{A\arabic{hypoA}}
\newenvironment{assumptionA}[1][]{
  \refstepcounter{hypoA}%
  \begin{itemize}
    \item[(\thehypoA)]\ifx\relax#1\relax\else\label{#1}\fi
}{
  \end{itemize}
}

\numberwithin{equation}{section}
\title
{Attractor Continuity for Semilinear Parabolic Equations on Thin Domains with Degenerating Outward Peaks}
\author[E. A. Tavares--Lima]{E. A. Tavares--Lima}
\author[B. Lorenzi]{B. Lorenzi}
\author[M. C. Pereira]{M. C. Pereira}

\thanks{Departamento de Matem\'atica, Instituto de Matem\'atica e Estat\'istica da Universidade de S\~ao Paulo, Rua do Mat\~ao, 1010 -- CEP 05508-090 -- S\~ao Paulo -- SP -- Brazil.
\newline
E-mail addresses:
\texttt{elaine.tavares@ime.usp.br};
\texttt{bianca.lorenzi@usp.br};
\texttt{marcone@ime.usp.br}.}



\begin{document}

\begin{abstract}
    In this work, we analyze the asymptotic behavior of the attractors associated with a semilinear parabolic equation subject to homogeneous Neumann boundary conditions and defined on a thin domain $R^\varepsilon \subset \mathbb{R}^{1+n}$. We assume that the thin domain exhibits a cusp, known as an outward peak, whose geometry is characterized by a nonnegative function that vanishes at a point on the boundary. Our objective is to rigorously establish the continuity of the attractors as $\varepsilon \to 0$ and to determine their rate of convergence.

\end{abstract}

\maketitle

\pagestyle{plain}

\section{Introduction}

The continuity of attractors under small perturbations of semilinear parabolic equations plays a fundamental role in understanding the qualitative behavior of physical models governed by reaction-diffusion equations. To ensure the robustness of such models, bounded solutions must exhibit permanence and continuity over time, even when initial data, parameters, or the underlying domain undergo changes. The asymptotic behavior of solutions, particularly at infinity, is also of central importance. Classical results have established conditions guaranteeing the existence and continuity of global attractors for semigroups generated by parabolic equations (see for instance \cite{Babin1992, A.N.Carvalho2010, Hale1988,  haleetall}). Moreover, works such as \cite{flank, esperanza, Carvalho2010, Carvalho,leonardo, PPires2024} have refined this theory by providing quantitative estimates for the continuity of the dynamics in one-parameter families of equations converging to a limiting problem. In this context, aspects such as the persistence of equilibrium, local invariant manifolds, and global attractors have been rigorously analyzed, leading to parameter-dependent convergence rates that quantify the stability of well-behaved perturbations.

Furthermore, the nonlinear dynamics of reaction-diffusion equations under domain perturbations has been extensively investigated. From pioneering contributions to more recent studies (see, for instance, \cite{AAB, arrieta, Daners, raugel, Henry_domain, Pereira2007, PPires2024, elaine, NV}), several theories have been developed to address a broad class of perturbed parabolic and elliptic problems. An interesting example in this direction is the thin-domain boundary perturbation problem, which has been carefully analyzed. In this setting, we may mention \cite{esperanza}, where convergence rates for global attractors were established and improved as a positive parameter $\varepsilon \to 0$. See also \cite{ACPS, am2020, mcp}, where the continuity of the dynamics generated by parabolic problems in oscillating thin domains was studied.

A general convergence-rate theory for attractors has emerged (see, for instance, \cite{flank, Carvalho2010, A.N.Carvalho2010, Carvalho}), extending and refining pioneering works and enabling the estimation of convergence for resolvent operators, linear and nonlinear semigroups, hyperbolic equilibria, unstable manifolds, and global attractors through a positive function $\tau(\varepsilon) \to 0$ as $\varepsilon \to 0$. This framework provides a unified approach for analyzing and quantifying the continuity of attractors arising from parameter perturbations in a broad class of nonlinear boundary value problems.

In this work, we build upon previous results and the compact convergence framework developed, for instance, in \cite{german1, Carvalho-Piskarev:06}, to rigorously analyze the convergence of the dynamics generated by a reaction-diffusion equation with homogeneous Neumann boundary conditions, posed on a bounded thin domain $R^\ep \subset \mathbb{R}^{1+n}$ whose geometry is described by a nonnegative function that vanishes at a boundary point. Such open sets will be carefully introduced in the next section (see Section \ref{Dconst}) and are known in the literature as degenerating outward peaks (see, for instance, \cite{KOVARIK20191600, baddomains}). Our goal is to rigorously justify the convergence of the dynamics as $\ep \rightarrow 0$. To the best of our knowledge, this is the first work to address the continuity of attractors for this class of thin domains. 

As will be seen in Section~\ref{PDE}, we consider a semilinear reaction-diffusion equation with a variable diffusion coefficient posed on an outward peak thin domain, subject to homogeneous Neumann boundary conditions. We also assume that the nonlinearities satisfy appropriate dissipative and growth conditions (see hypotheses \eqref{eq:growth} and \eqref{eq:diss}). It is worth noting that our approach advances and refines existing methods concerning the continuity of attractors for parabolic problems posed on thin domains. Moreover, we establish precise estimates describing the convergence from the perturbed problem to the limiting one as the domain of definition of the solutions varies. In this context, our results introduce new models that, in some sense, improve previous contributions, such as \cite{patricia, ACPS,  esperanza,  am2020, raugel, mcp}.

The paper is organized as follows. In Section~\ref{sec:2}, we describe the geometric structure of thin domains with degenerating outward peaks, introduce the rescaled parabolic problem on a fixed reference domain, and present the main assumptions together with the functional analytic framework. Section~\ref{sec:3} is devoted to the well-posedness of the rescaled problem, including the existence, regularity, and uniform boundedness of global solutions, as well as the existence of the associated global attractors. 

In Section~\ref{sec:4}, we investigate the corresponding linear elliptic problems and establish resolvent convergence results, providing sharp estimates and convergence rates as the perturbation parameter tends to zero. In the subsequent sections, namely Sections~\ref{sec:5} and~\ref{sec:6}, we analyze the convergence of the nonlinear semigroups, derive quantitative estimates for the convergence of trajectories, equilibria, and invariant manifolds, and prove the continuity of global attractors, together with explicit rates of convergence for the limiting dynamics. The main result of this work is Theorem~\ref{teoprincipal}, which establishes the continuity of the dynamics generated by the family of problems~\eqref{eq1} (defined by the reaction-diffusion problems \eqref{eq:Pe'} and \eqref{eq:Po}) and provides estimates for the convergence of their global attractors.

\section{Preliminary settings 
}\label{sec:2}

We begin by describing the class of thin domains that define the geometric setting, where the transversal thickness degenerates at one endpoint. 
We also state the main assumptions imposed on the coefficients and nonlinearities involved in the model, which are essential for the well-posedness and asymptotic analysis of the problem.
In addition, we establish some notation and present a number of auxiliary results that will be used throughout the paper. 

\subsection{Domain Construction} \label{Dconst}

Let $a \in C^1([0,1])$ be a non-negative function that determines the domain width transversely to the direction $x \in (0,1)$, satisfying:
\begin{hypothesis}[H1]
$a(0) = 0$ and $a(x) > 0 $ for all $ x \in (0,1]$;
\end{hypothesis}

\begin{hypothesis}[H2]
The function $$W(x) = \int_x^{\frac{1}{2}} \frac{1}{A_{01}(t)a^n(t)}\,dt \quad \text{satisfies} \quad\int_0^1 a^n(x)W^2(x)\,dx < \infty,$$ \noindent where \( A_{01}(\cdot) \in \mathbb{R} \) is a strictly positive smooth function;
\end{hypothesis}
\begin{hypothesis}[H3]
    There exist $0 < x_0 < 1$, $\alpha_1, \alpha_2 \in [1,\infty)$ with $0 \leq \alpha_1 - \alpha_2 < \frac{2}{n}$ and positive constants $K_1$ and $K_2$ such that $K_1 x^{\alpha_1} \leq a(x) \leq K_2 x^{\alpha_2}, \ \text{for } x \in (0,x_0).$
\end{hypothesis}

Now we introduce the family of thin domains that will be the focus of our analysis. For $\ep\in[0,\ep_0]$, let $n \in \mathbb{N}$ and $a : [0,1] \to \mathbb{R}$ be a function that satisfies \eqref{H1}, we define the thin domain depending on $\ep$ by $$R^\ep := \left\{ (x,y) \in \mathbb{R}^{1+n} : x \in (0,1),\ y \in \ep a(x) B_1 \right\},$$ where $B_1 \subset \mathbb{R}^n$ is the unit ball centered at the origin. This domain is a tubular neighborhood around the interval $(0,1) \subset \mathbb{R}$, whose transversal thickness tends to zero as $\ep$ goes to zero, degenerating completely at the end point $x = 0$.

In particular, we denote by $\Omega$ the fixed thin domain defined as
$$\Omega := \left\{ (x,y) \in \mathbb{R}^{1+n} : x \in (0,1),\ y \in a(x) B_1 \right\} = R^1.$$

From this fixed domain $\Omega$, we construct the $\varepsilon$-dependent domain $R^\ep$ by rescaling
$$R^\varepsilon := \left\{ (x, \varepsilon y) : (x, y) \in \Omega \right\}.$$
This change of variables enables us to formulate the problem on a fixed geometry, while the $\varepsilon$-dependence is encoded in the coefficients of the equations. 
\begin{figure}[h!]
\centering
\begin{subfigure}[b]{0.38\textwidth}
  \centering
  \begin{tikzpicture}
    \node[inner sep=0] (img) {\includegraphics[width=\linewidth]{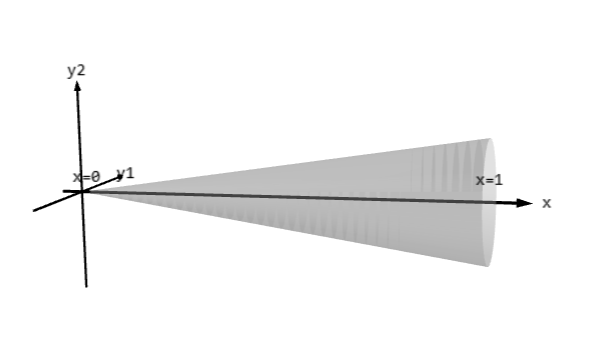}};
    \begin{scope}[x={(img.south east)},y={(img.north west)}]
      \node at (0.35,0.5) {$R^{\varepsilon}$};
    \end{scope}
  \end{tikzpicture}
  \caption{$R^\ep$ with $a(x)=x$ and $n=2$.}
  \label{fig:thin-a-x}
\end{subfigure}
\hfill
\begin{subfigure}[b]{0.38\textwidth}
  \centering
  \begin{tikzpicture}
    \node[inner sep=0] (img) {\includegraphics[width=\linewidth]{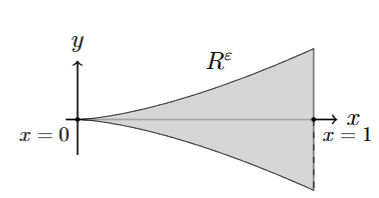}};
  \end{tikzpicture}
  \caption{$R^\ep$ with $a(x)=x^{\frac{3}{2}}$ and $n=1$.}
  \label{fig:thin-a-x32}
\end{subfigure}

\caption{Comparison of thin domains for different profiles $a(x)$.}
\label{fig:two-thin-domains}
\end{figure}

\subsection{Parabolic Problem}  \label{PDE} 

For $\ep \in (0, \ep_0]$, with $\varepsilon_0 < 1$, we consider the following nonlinear parabolic problem in the fixed domain $\Omega$
\begin{equation}\label{eq:Pe'}\tag{$\bf{P}_{\varepsilon}$}
\begin{cases}
    u^\ep_t - \nabla^\varepsilon \cdot \left( A^\varepsilon(x, y) \nabla^\varepsilon u^\varepsilon \right) + u^\varepsilon = f(u^\ep), & \text{in } \Omega, \\[0.3em]
    \big( A^\ep(x,y)\, \nabla^\varepsilon u^\varepsilon \big)\cdot \nu = 0,  & \text{on } \partial \Omega,
\end{cases}
\end{equation}
where the scaled gradient is defined by $$\nabla^\varepsilon u^\varepsilon(x, y) :=
\begin{pmatrix}
\partial_x u^\varepsilon(x, y) \\
\tfrac{1}{\varepsilon} \nabla_y u^\varepsilon(x, y)
\end{pmatrix},$$ and $\nu$ denotes the unit outward normal to $\partial \Omega$.

\medskip

The nonlinear function $f:\mathbb{R} \to \mathbb{R}$ is assumed to belong to $C^2(\mathbb{R})$ and satisfies the following growth condition 
\begin{align}
    |f'(s)| &\leq C(1 + |s|^{\gamma - 1}), \quad \text{for all } s \in \R, \label{eq:growth}
\end{align}
for some $\gamma\geq1$, and the dissipative condition 
\begin{align}
    \exists \;m_f>0 \text{ such that } f(s)\cdot s &\leq 0, \quad \text{for all } |s| \geq m_f. \label{eq:diss}
\end{align}

The family of matrices $A^\ep(x, y) \in \mathbb{R}^{(1+n)\times(1+n)}$ are symmetric, uniformly positive definite matrices, with the asymptotic expansion
$$A^\ep \left(x, y \right) = \sum_{k=0}^{\infty} \ep^k A_k \left(x, y\right),$$
with each term $A_k \in L^\infty\left((0,1) \times \mathbb{R}^n; \mathbb{R}^{(1+n)\times(1+n)}\right).$

We assume that
\begin{assumptionA}[A1]
    The leading order coefficient matrix \( A_0 \left(x, y \right) \in \R \) is symmetric, smooth, and uniformly positive definite. It admits the following block structure
\[
A_0 \left(x, y \right) =
\begin{bmatrix}
A_{01}(x) & 0 \\
0 & A_{03}(x, y)
\end{bmatrix},
\]
where \( A_{03}(x, y) \in \mathbb{R}^{n \times n} \) is a smooth, symmetric, and uniformly positive definite matrix for all \( (x,y) \in (0,1) \times \mathbb{R}^n \). In particular, there exists a constant \( \alpha_0 > 0 \) such that
\[
A_0 \xi \cdot \xi \geq \alpha_0 |\xi|^2, \quad \text{for all } \xi \in \mathbb{R}^{1+n}.
\]
\end{assumptionA}

\begin{assumptionA}[A2]
  Each matrix \( A_k\left(x, y \right) \in \mathbb{R}^{(1+n)\times(1+n)} \) represents the \( \varepsilon^k \)-order correction in the asymptotic expansion of \( A^\varepsilon \). It admits the block structure
\[
A_k\left(x,y \right) =
\begin{bmatrix}
A_{k1}(x) & A_{k2}^T\left(x, y \right) \\
A_{k2}\left(x, y\right) & A_{k3}\left(x, y\right)
\end{bmatrix},
\]
where $A_{k1}$, $A_{k2}$, and $A_{k3}$ contribute respectively to the longitudinal, mixed, and transversal parts of the diffusion operator with effective orders $\varepsilon^k$, $\varepsilon^{k-1}$, and $\varepsilon^{k-2}$ in the weak formulation.
  
  
\end{assumptionA}

\begin{assumptionA}[A3] 
There exists a constant $C_0 > 0$ such that, for $k \in \mathbb{N}$, and $(x,y) \in \Omega$, we have $$\| A_k(x, y) \|_{L^\infty(\Omega)} \le C_0.$$
Consequently, the family of matrices \(A^\varepsilon(x, y)\) satisfies the \emph{uniform ellipticity condition}
\begin{equation}\label{coer}
    A^\varepsilon(x, y)\, \xi \cdot \xi \ge \alpha\, |\xi|^2,
    \quad \forall\, \xi \in \mathbb{R}^{1+n}, \ (x,y)\in\Omega,
\end{equation}
where $\alpha := \alpha_0  - \frac{\varepsilon_0 C_0}{1-\varepsilon_0} > 0$ which remains uniformly positive definite for all $\varepsilon \in (0,\varepsilon_0]$ for some $\varepsilon_0$ sufficiently small.

\end{assumptionA}


We associate to \eqref{eq:Pe'} the $\varepsilon$-dependent Sobolev norm $$\| u \|_{H^1_\varepsilon(\Omega)} :=
\left( \int_\Omega A^\varepsilon(x,y)\nabla^\ep u \cdot \nabla^\ep u + |u|^2 \, dy dx \right)^{1/2}.$$

It is worth observing that the problem \eqref{eq:Pe'} arises naturally from the parabolic problem on the thin domain
\begin{equation}\label{eq:Pe}\tag{$\bf{Q}_{\varepsilon}$}
\begin{cases}
w_t^\varepsilon -\nabla \cdot \big( A^\ep(x, \tfrac{y}{\ep}) \nabla w^\ep \big) + w^\ep = f(w^\ep), & \text{in } R^\varepsilon, \\[0.3em]
\displaystyle \frac{\partial w^\ep}{\partial \nu^\ep} = 0, & \text{on } \partial R^\varepsilon,
\end{cases}
\end{equation}
after performing the change of variables $(x, y) \mapsto (x, \tfrac{y}{\varepsilon})$, which maps the thin domain $R^\varepsilon$ onto the fixed reference domain $\Omega$. In this way, we follow the same approach adopted, for instance, in previous works such as \cite{barros, raugel, mcp}, where parabolic problems on thin domains were considered.

Under the growth condition \eqref{eq:growth}, the nonlinear term in problem \eqref{eq:Pe'} satisfies the standard assumptions ensuring local existence and uniqueness of solutions in $H_\varepsilon^1(\Omega)$. Moreover, by combining this with the dissipative condition \eqref{eq:diss}, we obtain global in time solutions of \eqref{eq:Pe'} which remain uniformly bounded in $H_\varepsilon^1(\Omega)$. Consequently, the semiflow associated with \eqref{eq:Pe'} is well defined and dissipative in $H_\varepsilon^1(\Omega)$ (see, for instance, \cite{cho, Hale1988, henry}).

These properties allow us to expect the existence of a global attractor
$\mathcal{A}_\varepsilon \subset H_\varepsilon^1(\Omega)$ for problem
\eqref{eq:Pe'}
. In particular, there should exist a constant $C>0$, independent of time and
of the parameter $\varepsilon$, such that
\begin{equation}\label{limPe}
    \sup_{u^{\varepsilon}\in \mathcal{A}_{\varepsilon}}
    \|u^{\varepsilon}\|_{H_\varepsilon^1(\Omega)} \le C.
\end{equation}
Moreover, exploiting the geometric condition \eqref{H3} on the degeneracy
of the domain, we will prove that the global attractor $\mathcal{A}_\varepsilon$
is in fact uniformly bounded in $L^\infty(\Omega)$.

\medskip

In order to introduce the limit equation corresponding to the asymptotic regime $\varepsilon \to 0$,
we begin by defining the weighted function spaces that provide the natural functional framework
for the limiting problem.

The weighted Lebesgue space is defined by
\[
L^2_a(0,1) := \left\{ u:(0,1)\to\mathbb{R}\ \text{measurable}:\ \int_0^1 a^n(x)\,|u(x)|^2\,dx<\infty \right\},
\]
endowed with the norm $\|u\|_{L^2_a(0,1)} := \left(\int_0^1 a^n(x)\,|u(x)|^2\,dx\right)^{1/2}.$

Similarly, we set
\[
H^1_a(0,1) := \left\{ u\in L^2_a(0,1):\ u_x\in L^2_a(0,1)\right\}.
\]

As $\varepsilon\to0$, the thin domain collapses onto the interval $(0,1)$ and the corresponding
limit reaction--diffusion equation reads
\begin{equation}\tag{$\bf{P}_{0}$}\label{eq:Po}
\begin{cases}
u^0_t-\dfrac{1}{a^n}\left(A_{01} a^n u^0_x\right)_x + u^0 =  f(u^0), & x \in (0,1), \\
\lim\limits_{x \to 0^+} a^n(x) u^0_x(x) = 0, & \\
u^0_x(1) = 0.&
\end{cases}
\end{equation}

The limit problem \eqref{eq:Po} generates a dynamical system possessing a global
attractor $\mathcal{A}_0 \subset H^1_a(0,1)$.
Moreover, by one-dimensional Sobolev embeddings and standard parabolic
regularity, solutions on the attractor are uniformly bounded in $L^\infty(0,1)$,
so that
\begin{equation}\label{limPo}
\sup_{u^{0}\in \mathcal{A}_{0}}\|u^{0}\|_{L^{\infty}(0,1)} < \infty.
\end{equation}
In addition, the system admits a Lyapunov functional given by the associated
energy, and therefore the dynamical system generated by \eqref{eq:Po} has a
gradient structure (see \cite{Hale1988}), implying that its attractor consists
of equilibria and the heteroclinic connections between them. Naturally, $\mathcal{A}_0$ can be embedded into $H^1_\varepsilon(\Omega)$ by extending each function
$u^0(x)$ to $\Omega$ through the trivial extension $\tilde{u}^0(x,y)=u^0(x)$ for all $(x,y)\in\Omega$.

\medskip

We define the extension operator associated with the thin domain \( R^\varepsilon \), which maps functions defined on \( (0,1) \) into functions defined on \( R^\varepsilon \), by
\begin{align*}
E_\varepsilon : L^2_a(0,1) &\longrightarrow L^2(R^\varepsilon), \\
u &\longmapsto E_\varepsilon u(x,y) := u(x).
\end{align*}
This operator extends the function \( u(x) \) constantly along the transversal section \( y \in \varepsilon a(x) B_1 \).


We also define the averaging operator associated with the thin domain \( R^\varepsilon \), which maps functions defined on \( R^\varepsilon \) into functions of \( x \in (0,1) \), by
\begin{align*}
M_\varepsilon : L^2(R^\varepsilon) &\longrightarrow L^2_a(0,1), \\
u &\longmapsto M_\varepsilon u(x) := \frac{1}{\varepsilon^n a^n(x) |B_1|} \int_{\varepsilon a(x) B_1} u(x,y)\, dy.
\end{align*}
This operator computes the average of \( u \) in the transverse section \( \varepsilon a(x) B_1 \).

Similarly, we define the corresponding operators in the fixed domain $\Omega$. The extension operator is given by
\begin{align*}
E : L^2_a(0,1) &\longrightarrow L^2(\Omega), \\
u &\longmapsto E u(x,y) := u(x),
\end{align*}
and the averaging operator is defined as
\begin{align*}
M : L^2(\Omega) &\longrightarrow L^2_a(0,1), \\
u &\longmapsto M u(x) := \frac{1}{a^n(x)|B_1|} \int_{a(x) B_1} u(x,y)\, dy.
\end{align*}

We include the following auxiliary results on the operators $E_\ep, E$ and $M_{\ep}$, which will be instrumental in several estimates presented later. For this purpose, we introduce the coordinate transformation
\begin{align*}
l : (0,1) \times B_1 &\longrightarrow \Omega, \\
(x,z) &\longmapsto l(x,z) := (x, a(x)z).
\end{align*}
For each fixed \( x \in (0,1) \), we define
\[
l_x : B_1 \to \Omega, \quad l_x(z) := l(x, a(x)z).
\]
Note that each map \( l_x \) is a $C^1$-diffeomorphism in its image.

\begin{lemma}\label{teo21}
If a satisfies the condition \eqref{H1} and \( u \in H^1_\varepsilon(\Omega) \), then \( M_\varepsilon u \in H^1_a(0,1) \). Moreover, the derivative of \( M_\varepsilon u \) satisfies the identity
\[
(M_\varepsilon u)_x(x) = \frac{1}{|B_1|} \left( \int_{B_1} u_x(x, \varepsilon l_x(z))\,dz + \int_{B_1} \nabla_y u(x, \varepsilon l_x(z)) \cdot \varepsilon a'(x) z\, dz \right).
\]
\end{lemma}
\demo
See \cite[Proposition 4.1]{esperanza}.
\fimdemo
\begin{lemma}\label{lemma22}
If a satisfies the condition \eqref{H1}, then there exists \(\gamma > 0\) such that
\begin{align*}
\|u^\ep - EM u^\ep\|^2_{L^2(\Omega)} &\leq \gamma \|\nabla_y u^\ep\|^2_{L^2(\Omega)}, \quad \forall\ u^\ep \in H^1_\ep(\Omega), \\
\|w^\ep - E_\ep M_\ep w^\ep\|^2_{L^2(R^\ep)} &\leq \gamma\, \ep^2 \|\nabla_y w^\ep\|^2_{L^2(R^\ep)}, \quad \forall\ w^\ep \in H^1(R^\ep).
\end{align*}
\end{lemma}
\demo
See \cite[Proposition 2.4]{esperanza}.
\fimdemo

\

The evolution problems \eqref{eq:Pe'} and \eqref{eq:Po} admit an abstract operator formulation that will be useful for analysis. Let $L_\varepsilon : D(L_\varepsilon) \subset L^2(\Omega) \to L^2(\Omega)$ be defined by
\begin{equation} \label{Leps}
L_\varepsilon u = -\nabla^\varepsilon \cdot \left(A^\varepsilon(x, y) \nabla^\varepsilon u \right) + u,
\end{equation}
with 
$ D(L_\varepsilon) = \left\{ u \in H^2(\Omega) : \big( A^\ep(x,y)\, \nabla^\varepsilon u^\varepsilon \big)\cdot \nu = 0 \text{ on } \partial \Omega \right\}.$

The limiting operator $L_0 : D(L_0) \subset L^2_{a}(0,1) \to L^2_{a}(0,1)$ is given by
\begin{equation*}
L_0 u = -\frac{1}{a^n(x)} \left( A_{01}(x)\, a^n(x)\, u_x \right)_x + u,
\end{equation*}
with $D(L_0) = \left\{ u \in H^1_{a}(0,1) \; : \; A_{01} \, a^n \, u_x\in H^1_a(0,1), \; \lim\limits_{x \to 0^+} a^n(x)\, u_x(x) = 0,\; u_x(1) = 0 \right\}.$

Observe that $L_\varepsilon$ and $L_0$ are selfadjoint, positive linear operators with compact resolvent, sectorial (see \cite{cho} and Theorem \ref{baraocp}) and are defined on separable Hilbert spaces. For $\varepsilon\in (0, \varepsilon_0]$, let $X_{\varepsilon}=L^2(\Omega)$ with the usual norm and $X_0 = L^2_a(0,1)$ with the previously defined weighted norm.
 Hence, we can define the fractional powers $L_\varepsilon^\alpha$, $\al\in(0,1]$ and $\ep\in [0,\varepsilon_0]$, of the operator $L_\varepsilon$ whose domain is denoted by $X_\varepsilon^\alpha$ (see \cite{cho}). The values of interest for $\alpha$ in this paper are $\alpha = 1, \alpha = 0$ and $\alpha = 1/2.$ In these three cases, the characterization of the respective domains is classic (see \cite{cho}). We have $X_{\varepsilon}^1= D(L_\varepsilon)$ with the graph norm and similarly for $X_0^1=D(L_0)$, $X_{\varepsilon}^{\frac{1}{2}} = H_\ep^1(\Omega)$ and $X_{0}^{\frac{1}{2}} = H_a^1(0,1)$.  

Then the problems \eqref{eq:Pe'} and \eqref{eq:Po} can be written as a family of evolution problems, for $\ep \in [0,\ep_0]$,
\begin{align}\label{eq1}
\begin{cases}
u_t^{\varepsilon} + L_\varepsilon u^{\varepsilon}  = F_\ep(u^{\varepsilon}), \ t > 0 \\ \displaystyle u^{\varepsilon}(0)  = u^{\varepsilon}_0 \in X_{\varepsilon}^{\al},
\end{cases} 
\end{align}
where $F_\varepsilon : X^\alpha_\varepsilon \to X_\ep$ are the Nemitskii's operators associated with the nonlinearity $f$, defined by $F_\varepsilon(u) = f(u), \ \text{for all} \ u \in X^\alpha_\varepsilon.$


\section{Well-posedness and existence of a global attractor}\label{sec:3}




In order to rigorously justify the results stated in the previous section,
we now establish the well-posedness of the rescaled problem \eqref{eq:Pe'}
and prove the existence of its global attractor.
To this end, we reformulate \eqref{eq:Pe'} as an abstract evolution equation
associated with the operator $L_\varepsilon$ acting on $L^2(\Omega)$.


Under hypothesis \eqref{H3}, the function $a(\cdot)$ 
behaves near $x=0$ like $x^{\alpha_i}$, with exponents \(\alpha_1,\alpha_2\) satisfying \(0 \le \alpha_1 - \alpha_2 < 2/n\). In this case, the domain $\Omega$ is locally equivalent to a cuspidal domain in the sense of Maz'ya and Poborchi \cite{baddomains}. 
By the Sobolev embedding theorem for cusp domains (see Theorem 1.1 in \cite{baddomains}), there exists an exponent $q = q(n,\alpha_1,\alpha_2) > 2$ such that the embedding $H^1(\Omega) \hookrightarrow L^p(\Omega)$ is continuous for all $2 \le p \le q$, and compact whenever $p < q$. This result extends the classical Sobolev embedding to geometries with degenerating cross-sections, showing that the loss of regularity in the boundary only reduces the critical exponent $q$, but does not destroy the embedding property.

Moreover, since $\| u \|_{H^1(\Omega)} \leq \|u\|_{H_\varepsilon^1(\Omega)}$ for any $\ep \in (0,\ep_0]$, $H_\varepsilon^1(\Omega) \hookrightarrow L^p(\Omega)$, for $ 2 \le p \le q,$ and they are compact for $p < q$. 
In particular, one can choose $p$ such that $2\gamma \le p < q$, which is possible under \eqref{H3}. Then there exists a constant $C_S > 0$, independent of $\ep$, such that 
\begin{equation}\label{eq:emb-2gamma}
    \|u\|_{L^{2\gamma}(\Omega)} \le C_S\,\|u\|_{H_\varepsilon^1(\Omega)}.
\end{equation}


\begin{lemma}\label{lema:nemytskii}
Let $f$ satisfy the growth condition \eqref{eq:growth}
for some $\gamma\ge1$, and assume hypothesis \eqref{H3}.
Then 
the associated Nemitskii's operator
$F_\ep :X^{1/2}_\varepsilon\to X_\varepsilon$, is locally Lipschitz with
Lipschitz constant uniform with respect to $\varepsilon$.
More precisely, for every $R>0$ there exists a constant $L_F>0$,
independent of $\varepsilon$, such that
\[
  \|F_\ep(u^\varepsilon)-F_\ep(v^\varepsilon)\|_{X_\varepsilon}
  \le L_F\,\|u^\varepsilon-v^\varepsilon\|_{X^{1/2}_\varepsilon}
\]
for all $u^\varepsilon,v^\varepsilon\in B_R$, where $B_R:=\left\{w^\varepsilon\in X^{1/2}_\varepsilon:
\|w^\varepsilon\|_{X^{1/2}_\varepsilon}\le R\right\}.$
In addition, there exists a constant $C_F>0$, independent of $\varepsilon$,
such that
\[
  \|F_\ep(u^\varepsilon)\|_{X_\varepsilon}\le C_F,
  \ \ \text{for all } u^\varepsilon\in B_R.
\]
\end{lemma}

\begin{proof}
We first prove the local Lipschitz continuity of $F_\ep$.
By the mean value theorem and the growth condition \eqref{eq:growth},
for any $a,b\in\mathbb{R}$ there exists $\theta\in(0,1)$ such that
\[
|f(a)-f(b)|
= |f'(\theta a+(1-\theta)b)|\,|a-b|
\le \vartheta_1\bigl(1+|a|^{\gamma-1}+|b|^{\gamma-1}\bigr)\,|a-b|.
\]
Setting $a=u^\varepsilon(x)$ and $b=v^\varepsilon(x)$ and squaring, we obtain
\[
|F_\ep(u^\varepsilon)-F_\ep(v^\varepsilon)|^2
\le C\bigl(1+|u^\varepsilon|^{2(\gamma-1)}+|v^\varepsilon|^{2(\gamma-1)}\bigr)
      |u^\varepsilon-v^\varepsilon|^2,
\]
where $C>0$ depends only on $\vartheta_1$ and $\gamma$.
Integrating over $\Omega$ and applying Hölder’s inequality with conjugate
exponents $\gamma$ and $r=\frac{\gamma}{\gamma-1}$ (the case $\gamma=1$
being straightforward), we find
\begin{align*}
\|F_\ep(u^\varepsilon)-F_\ep(v^\varepsilon)\|_{L^2(\Omega)}^2
&\le C\,\|u^\varepsilon-v^\varepsilon\|_{L^{2\gamma}(\Omega)}^2 
\Bigl(1+\|u^\varepsilon\|_{L^{2\gamma}(\Omega)}^{2(\gamma-1)}
        +\|v^\varepsilon\|_{L^{2\gamma}(\Omega)}^{2(\gamma-1)}\Bigr).
\end{align*}

By the uniform embedding \eqref{eq:emb-2gamma}, for $u^\varepsilon,v^\varepsilon\in B_R$ we have
\[
\|u^\varepsilon-v^\varepsilon\|_{L^{2\gamma}(\Omega)}
\le C_S\|u^\varepsilon-v^\varepsilon\|_{X^{1/2}_\varepsilon},
\qquad
\|u^\varepsilon\|_{L^{2\gamma}(\Omega)},
\|v^\varepsilon\|_{L^{2\gamma}(\Omega)}\le C_S R.
\]
Therefore,
\[
\|F_\ep(u^\varepsilon)-F_\ep(v^\varepsilon)\|_{X_\varepsilon}
\le L_F(R)\,\|u^\varepsilon-v^\varepsilon\|_{X^{1/2}_\varepsilon},
\]
with, $L_F(R)=C\,C_S\bigl(1+(C_S R)^{2(\gamma-1)}\bigr)^{1/2}$ (which is independent of $\varepsilon$). This proves the local Lipschitz continuity.

We now establish the growth estimate on bounded sets.
From the growth condition \eqref{eq:growth} on $f'$, by Newton–Leibniz we obtain 
\begin{equation}\label{eq:poly-f}
|f(s)| \le |f(0)| + \int_0^{|s|} C\bigl(1+\xi^{\gamma-1}\bigr)\,d\xi
\le C_1\bigl(1+|s|^\gamma\bigr), \quad s\in\mathbb{R},
\end{equation}
for some constant $C_1>0$. Squaring \eqref{eq:poly-f} and integrating over $\Omega$ yields
\[
\|F_\varepsilon(u)\|_{L^2(\Omega)}^2
= \int_\Omega |f(u(x))|^2\,dx
\le C_2\int_\Omega \bigl(1+|u(x)|^{2\gamma}\bigr)\,dx
= C_2\bigl(|\Omega|+\|u\|_{L^{2\gamma}(\Omega)}^{2\gamma}\bigr).
\]
Taking square roots and using \eqref{eq:emb-2gamma}, 
\[
\|F_\varepsilon(u)\|_{L^2(\Omega)}
\le C_3\Bigl(1+\|u\|_{L^{2\gamma}(\Omega)}^\gamma\Bigr)
\le C_3\Bigl(1+(C\,\|u\|_{H^1_\varepsilon(\Omega)})^\gamma\Bigr)
\le C\bigl(1+\|u\|_{H^1_\varepsilon(\Omega)}^\gamma\bigr),
\]
where all constants are independent of $\varepsilon$.
In particular, if $u_\varepsilon\in B_R$, then
\[
\|F_\ep(u_\varepsilon)\|_{X_\varepsilon}\le C_F(R),
\]
with $C_F(R)>0$ independent of $\varepsilon$.
The proof is complete.
\end{proof}

\begin{remark} Below we state a basic well-posedness result adapted from \cite[Section 2.3]{cho}. For this, we fix $\al \in (0,\frac{1}{2}]$. Since the embedding $X^{1/2}_{\varepsilon} \hookrightarrow X^{\alpha}_{\varepsilon}$ is continuous, an analogous result as Lemma \ref{lema:nemytskii} substituting respectively $X^\frac12$ and $X$ with $X^\alpha_\varepsilon$ and $X_\varepsilon$. In particular, the boundedness of $u^\varepsilon$ in $X^\al_\ep$ 
ensures boundedness of $F_\ep(u^\varepsilon)$ in $X_\ep$. 
\end{remark}

\begin{theorem}[Local well-posedness]\label{thm:local-eps}
Under the hypotheses on $L_\ep$ and $f$, for each $u_0^\varepsilon \in X^\al_\ep$, with $0 \leq \alpha < \tfrac{1}{2}$, there exists a unique local solution $u^\varepsilon = u^\varepsilon(t,u_0^\varepsilon)$ of \eqref{eq1}, defined on its maximal interval of existence $[0,\tau_{u_0^\varepsilon})$, such that either $\tau_{u_0^\varepsilon} = +\infty$ or 
$$\limsup\limits_{t\to\tau_{u_0^\varepsilon}^-} 
\|u^\varepsilon(t,u_0^\varepsilon)\|_{X^\al_\ep} = \infty.$$
The local solution is classical and depends continuously on the initial data in the sense that the map $u_0^\varepsilon \mapsto u^\varepsilon(\cdot,u_0^\varepsilon)$ is continuous from $X_\ep^\al$ to $C([0,\tau_0],X^\al_\ep)$, for some $0<\tau_0\le\tau_{u_0^\varepsilon}$.
\end{theorem}



Once a local solution is shown to exist for the abstract problem associated with \eqref{eq:Pe'}, Theorem~\ref{thm:local-eps} gives a criterion to verify its globality, namely it suffices to show that 
\begin{equation}\label{limsup1-eps}
    \limsup_{t \to T^-} \|u^\varepsilon(t,u_0^\varepsilon)\|_{X^\al_\ep} < \infty,
\end{equation}
for all $T>0$. Hence, global existence is a direct consequence of good control of the solution $u^\varepsilon(\cdot,u_0^\varepsilon)$ at the $X^\al_\ep$-level for large times. Using the analyticity of the semigroup generated by $L_\varepsilon$ and the
theory of fractional powers of sectorial operators, it follows that, for
each $\alpha\in[0,1]$, there exists a constant $c_\alpha>0$, independent of
$\varepsilon$, such that
\[
\|L_\varepsilon^\alpha e^{-tL_\varepsilon}\|_{\mathcal{L}(X_\varepsilon)}
\le \frac{c_\alpha}{t^\alpha}e^{-\beta t},
\qquad t>0.
\]
Consequently, for any $u_0^\varepsilon\in
X_\varepsilon^\alpha$,
\begin{align*}
\|u^\varepsilon(t,u_0^\varepsilon)\|_{X_\varepsilon^\alpha}
\le\;
c_\alpha e^{-\beta t}\|u_0^\varepsilon\|_{X_\varepsilon^\alpha}
+\int_0^t \frac{c_\alpha}{(t-s)^\alpha}e^{-\beta(t-s)}
\|F_\varepsilon(u^\varepsilon(s,u_0^\varepsilon))\|_{X_\varepsilon}\,ds ,
\end{align*}
which, not surprisingly, indicates that the map $s \mapsto \|F_\ep(u(s,u_0^\ep))\|_{X_\ep}$ is the only possible reason why a local solution would not be global. Then if $F_\ep$ is uniformly bounded, say by $c_F > 0$, the estimate above implies 
\begin{align}\label{limsup_b}
    \limsup\limits_{t \to \infty} \|u(t, u_0^\ep)\|_{X^\al_\ep} & \leqslant 
	c_{\al}c_F \limsup\limits_{t \to \infty} \int_0^t (t-s)^{-\al}e^{-\beta(t-s)}ds < \infty
\end{align}
by a standard argument using $\Gamma$-functions. 
\medskip


\begin{theorem}[Boundedness of global solutions]\label{teosolgloblim1-eps}
Under assumptions \eqref{A1}--\eqref{A3}, \eqref{H1}--\eqref{H3}, and the nonlinear conditions \eqref{eq:growth}–\eqref{eq:diss}, the local solution $u^\varepsilon(\cdot, u_0^\varepsilon)$ of problem \eqref{eq1} is global. Moreover, there exists a constant $C = C(\|u_0^\varepsilon\|_{H_\varepsilon^1(\Omega)}) > 0$ such that $$\limsup_{t \to \infty} \|u^\varepsilon(t, u_0^\varepsilon)\|_{X^\al_\ep} \le C.$$
\end{theorem}


The well-posedness of bounded global solutions for \eqref{eq1} allows us to define the (nonlinear) semigroup of solutions given by $T_\varepsilon(t)u_0^\varepsilon = u^\varepsilon(t,u_0^\varepsilon), \ t \ge 0.$
By \cite[p.~71]{cho}, the family \(\{T_\varepsilon(t)\}_{t \ge 0}\) forms a strongly continuous semigroup on $X^\al_\ep$.

We now briefly discuss the existence of global attractors for problem \eqref{eq1}.
The existence of a compact global attractor for the semiflow 
$\{T_\varepsilon(t)\}_{t\ge0}$ follows from the abstract theory of 
infinite-dimensional dissipative dynamical systems developed in \cite[Theorems 1.1.2 and 4.1.1]{cho} 
and \cite[Chapter 3]{Hale1988}. 
In particular, the semiflow satisfies the classical hypotheses ensuring the existence of a global attractor: dissipativity, asymptotic compactness and strong continuity.

\begin{theorem}[Global attractor]\label{glob-eps}
Under assumptions \eqref{A1}--\eqref{A3},
\eqref{H1}--\eqref{H3},
and \eqref{eq:growth}--\eqref{eq:diss},
the semigroup $\{T_\varepsilon(t)\}_{t\ge0}$ generated by 
\eqref{eq1} has a compact global attractor 
$\mathcal{A}_\varepsilon \subset X^\al_\ep$,
which is invariant and attracts all bounded subsets of 
$X^\al_\ep$. Moreover,
$$\sup_{u\in\mathcal{A}_\varepsilon}\|u\|_{X^\al_\ep}\le C,$$
for some constant $C>0$ independent of time.
\end{theorem}

We now prove that the global attractor $\mathcal A_\varepsilon$ associated with
problem \eqref{eq:Pe'} is uniformly bounded in $L^\infty(\Omega)$.

Since the transversal section at a fixed $x\in(0,1)$ is the ball $a(x)B_1$ and
\eqref{H3} yields $a(x)\ge K_1 x^{\alpha_1}$ for $x\in(0,x_0)$, its
$n$--dimensional measure satisfies
\[
|a(x)B_1|=|B_1|\,a^n(x)\ge |B_1|K_1^n x^{n\alpha_1}.
\]
Therefore, the volume of the truncated domain
$\Omega_r:=\Omega\cap\{0<x<r\}$ can be estimated as
\[
|\Omega_r|
=\int_0^r |B_1|a^n(x)\,dx
\;\ge\;
|B_1|\int_0^r K_1^n x^{n\alpha_1}\,dx
=
\frac{|B_1|K_1^n}{1+n\alpha_1}\,r^{1+n\alpha_1}.
\]
This shows that the measure of regions concentrated near the degenerate endpoint
grows at least polynomially with exponent $\mu:=1+n\alpha_1.$

Choosing the critical exponent $p=\frac{2\mu}{\mu-2}$ and combining the above
estimate with the coercivity of the quadratic form associated with $L_\varepsilon$,
we obtain the functional inequality
\[
\|u\|_{L^{\frac{2\mu}{\mu-2}}(\Omega)}^2
\;\le\;
C\bigl(b_\ep(u, u)+\|u\|_{L^2(\Omega)}^2\bigr),
\qquad
u\in H^1_\varepsilon(\Omega),
\]
with a constant $C>0$ independent of $\varepsilon$.
Here $ b_\ep : H^1_{\ep}(\Omega) \times H^1_{\ep}(\Omega) \to \mathbb{R}$ denotes the closed, symmetric bilinear form associated with
the uniformly elliptic operator $L_\varepsilon$, given by
\[
b_\ep(u, u) :=
\int_\Omega A^\varepsilon(x,y)\,\nabla^\varepsilon u\cdot\nabla^\varepsilon u
\,dx\,dy,
\quad
u \in H^1_\varepsilon(\Omega).
\]
By the uniform ellipticity of $A^\varepsilon$, this form is coercive and controls
the $H^1_\varepsilon(\Omega)$--norm uniformly with respect to $\varepsilon$.

Since $\mu>2$ (equivalently, $\alpha_1>1/n$), \cite[Corollary 2.4.3, p. 77]{davies} applies and yields
the ultracontractive estimate
\begin{equation}\label{eq:ultraPe}
\|e^{-tL_\varepsilon}\varphi\|_{L^\infty(\Omega)}
\le
C\,t^{-\mu/4}\,\|\varphi\|_{L^2(\Omega)},
\qquad
0<t\le1,
\end{equation}
for all $\varphi\in L^2(\Omega)$, with $C>0$ independent of $\varepsilon$.

\medskip
Let $u^\varepsilon(t)$ be a global solution of \eqref{eq:Pe'}.
Using the variation of constants formula and \eqref{eq:ultraPe}, we obtain for
$t\in(0,1]$
\[
\|u^\varepsilon(t)\|_{L^\infty(\Omega)}
\le
C\,t^{-\mu/4}\|u^\varepsilon(0)\|_{L^2(\Omega)}
+
C\int_0^t (t-s)^{-\mu/4}
\|f(u^\varepsilon(s))\|_{L^2(\Omega)}\,ds.
\]

Since the semiflow generated by \eqref{eq:Pe'} is dissipative in
$H^1_\varepsilon(\Omega)$, there exists a bounded absorbing set in
$L^2(\Omega)$.
Together with the growth condition~\eqref{eq:growth} on $f$, this implies that
the right-hand side above remains uniformly bounded for large times.
Hence, there exist constants $T>0$ and $N>0$, independent of $\varepsilon$, such that $\|u^\varepsilon(t)\|_{L^\infty(\Omega)} \le N,
\  \forall\,t\ge T.
$

Finally, by invariance of the global attractor, we conclude that
\begin{equation}\label{eq:liminf}
\sup_{u\in\mathcal A_\varepsilon}\|u\|_{L^\infty(\Omega)} \le N.
\end{equation}

\section{Estimates for the Elliptic Problem}\label{sec:4}

In this section, we derive estimates that describe the asymptotic behavior of resolvents as $\ep$ tends to zero.

We consider the family of elliptic problems defined in thin domains \( R^\varepsilon \subset \mathbb{R}^{1+n} \)
\begin{equation} \label{Probpert}
\left\{
\begin{aligned}
- \nabla \cdot \left( A^\varepsilon\left(x, \tfrac{y}{\varepsilon} \right) \nabla w^\varepsilon \right) + w^\varepsilon &= f^\varepsilon, &&\text{in } R^\varepsilon, \\
\frac{\partial w^\varepsilon}{\partial \nu^\varepsilon} &= 0, &&\text{on } \partial R^\varepsilon,
\end{aligned}
\right.
\end{equation}
where the matrix \( A^\varepsilon \) is assumed to satisfy the assumptions \eqref{A1}--\eqref{A3}, and \( f^\varepsilon \in L^2(R^\varepsilon) \).

To carry out the analysis in a fixed reference domain, we apply the change of variables $(x, y) \mapsto (x, \tfrac{y}{\varepsilon}) \in \Omega.$ Under this transformation, the problem \eqref{Probpert} becomes the following scaled elliptic problem posed on the fixed domain \( \Omega \)
\begin{equation} \label{Probpertmud}
\left\{
\begin{aligned}
- \nabla^\varepsilon \cdot \left( A^\ep(x, y)\, \nabla^\varepsilon u^\varepsilon \right) + u^\varepsilon &= f^\varepsilon,
&& \text{in } \Omega, \\
\big( A^\ep(x,y)\, \nabla^\varepsilon u^\varepsilon \big)\cdot \nu &= 0,
&& \text{on } \partial \Omega.
\end{aligned}
\right.
\end{equation}

\noindent
The variational formulation associated with problem \eqref{Probpertmud} consists in finding \( u^\varepsilon \in H^1_{\ep}(\Omega) \) such that
\begin{equation} \label{eq:variational_formulation}
b_\ep(u^\varepsilon, \varphi) = \int_\Omega f^\varepsilon \varphi, \quad \forall\, \varphi \in H^1_{\ep}(\Omega).
\end{equation}
Using the structure of $A^\ep(x,y)$, the bilinear scalar product can be expanded as
\begin{align} \label{exp}
b_\ep(u^\varepsilon, \varphi) &= \int_\Omega \Bigg[
\partial_x u^\varepsilon A_{01}(x) \partial_x \varphi
+ \frac{1}{\varepsilon^2} \nabla_y u^\varepsilon \cdot A_{03}(x,y) \nabla_y \varphi \nonumber \\
&\quad+ \sum_{k=1}^\infty \varepsilon^k \bigg(
\partial_x u^\varepsilon A_{k1}(x,y) \partial_x \varphi
+ \frac{1}{\varepsilon} \partial_x u^\varepsilon \cdot A_{k2}(x,y) \nabla_y \varphi \nonumber \\
&\quad+ \frac{1}{\varepsilon} \nabla_y u^\varepsilon \cdot A_{k2}^T(x,y) \partial_x \varphi
+ \frac{1}{\varepsilon^2} \nabla_y u^\varepsilon \cdot A_{k3}(x,y) \nabla_y \varphi
\bigg)
\Bigg]
+ u^\varepsilon \varphi \,dxdy.
\end{align}

\noindent

Well-posedness of the variational problem \eqref{eq:variational_formulation} follows from assumptions \eqref{H1} and \eqref{A1}--\eqref{A3}. The bilinear form \( b_\ep(\cdot, \cdot) \) is both continuous and coercive on the weighted Sobolev space \( H^1_\varepsilon(\Omega) \), and the linear functional associated with \( f^\varepsilon \in L^2(\Omega) \) is continuous. These properties ensure that the variational formulation admits a unique weak solution \( u^\varepsilon \in H^1_\varepsilon(\Omega) \), according to the Lax-Milgram Theorem.

\vspace{0.1cm}
We now present the convergence theorem, which characterizes the asymptotic behavior of the family of solutions \( \{u^\varepsilon\} \) to the variational problem \eqref{eq:variational_formulation} as \( \varepsilon \to 0 \). The result confirms the convergence to a one-dimensional limit problem governed by an effective equation in the weighted space \( H^1_a(0,1) \). From now on, all results are proved under assumptions \eqref{A1}--\eqref{A3}.

\begin{theorem}\label{thm:convergencia}
Let $a$ be a function satisfying condition \eqref{H1} and $\{f^\ep\}_{\ep \in (0,\ep_0]} \subset L^2(\Omega)$ such that $\|f^\ep\|_{L^2(\Omega)} \leq C_f$ for some $C_f > 0$, and suppose there exists $\hat{f} \in L^2_a(0,1)$ such that $Mf^\ep \to \hat{f}$ weakly in $L^2_a(0,1)$ as $\ep \to 0$. If $u^\ep \in H^1_{\ep}(\Omega)$ solves
\begin{align*}\small
b_\ep(u^\ep, \varphi)  = \int_\Omega f^\ep \varphi, \quad \forall \varphi \in H^1_{\ep}(\Omega),
\end{align*}
then there exists $u^0 \in H^1_a(0,1)$ such that, $u^\ep \to E u^0$ in $H^1(\Omega)$, where $u^0$ satisfies, for all $\varphi \in H^1_a(0,1)$,
\[
\int_0^1a^n(x)\Big(A_{01}(x) u^{0}_x(x)\varphi_x(x) + u^0(x)\varphi(x)\Big) dx = \int_0^1 a^n(x) \hat{f}(x)\varphi(x) dx.
\]
\end{theorem}

\demo
Note that the uniform bound $\|f^\varepsilon\|_{L^2(\Omega)} \leq C_f$ implies, by the coercivity of the bilinear form $b_\ep(\cdot, \cdot)$, the existence of a constant $C > 0$, independent of $\ep$, such that $\|u^\varepsilon\|_{H^1(\Omega)} \leq C.$ Then, there exists a subsequence such that: $u^\ep \rightharpoonup u^0$ weakly in $H^1(\Omega)$, and $u^\ep \to u^0$ strongly in $L^2(\Omega)$.

Since, $\|\nabla_y u^\ep\|_{L^2(\Omega)} \leq C\ep \Rightarrow \nabla_y u^\ep \to 0 \text{ strongly in } L^2(\Omega),$ it follows that $u^0 = u^0(x)$, that is, $u^0$ is independent of $y$.

Let $\psi \in H^1_a(0,1)$ and define $\varphi(x,y) := \psi(x)$, constant in $y$. This choice is valid since $\varphi \in H^1(\Omega)$.

\paragraph{Terms converging to zero:}
\begin{itemize}
  \item $\nabla_y \varphi = 0 \Rightarrow$ all terms involving $\nabla_y \varphi$ vanish.
  \item Since \( \partial_x \psi \in L^\infty(0,1) \) and \( A_{k1}(x,y) \) is uniformly bounded by hypothesis, we estimate
\[
\left\| \varepsilon^k \partial_x u^\varepsilon A_{k1}(x,y) \partial_x \psi(x) \right\|_{L^1(\Omega)}
\leq \varepsilon^k C\|A_{k1}\|_{L^\infty(\Omega)} \|\partial_x \psi\|_{L^\infty(0,1)} \|\partial_x u^\varepsilon\|_{L^2(\Omega)}.
\]
As \( \|\partial_x u^\varepsilon\|_{L^2(\Omega)} \) is uniformly bounded and \( \varepsilon^k \to 0 \) for \( k \geq 1 \), it follows that
\[
\varepsilon^k \partial_x u^\varepsilon A_{k1} \partial_x \psi \to 0 \quad \forall k\geq1 \ \text{as } \ep \to 0.
\]
  \item Since $\nabla_y u^\ep \to 0$ strongly, and due to the structure of $A_{k2}(x,y)$, we have
  \[
  \left\| \ep^{k-1} \nabla_y u^\ep A_{k2}^T \partial_x \psi \right\|_{L^1(\Omega)} \leq \ep^{k-1}C\|A_{k2}\|_{L^\infty(\Omega)}\|\partial_x \psi\|_{L^\infty (0,1)}\|\nabla_y u^\ep\|_{L^2(\Omega)} \to 0, \forall k\geq1 \ \text{as } \ep \to 0.
  \]
\end{itemize}

Only the terms $\partial_x u^\ep A_{01} \partial_x \psi + u^\ep \psi$, remain in the limit, and since $\partial_x u^\ep\rightharpoonup \partial_x u^0$ weakly in $L^2$, we obtain
\begin{align*}
  \int_\Omega A_{01}(x) \partial_x u^\ep \partial_x \psi + u^\ep \psi dxdy
  &\to \int_\Omega A_{01}(x) \partial_x u^0(x) \partial_x \psi + u^0(x) \psi(x) dxdy \\
  &= |B_1| \int_0^1 a^n(x)\left( A_{01}(x) u^{0}_{x}(x) \psi_{x}(x) + u^0(x) \psi(x) \right) dx.
\end{align*}
\noindent Right-hand side
  $$\int_\Omega f^\ep \psi(x) dxdy = |B_1| \int_0^1 a^n(x) Mf^\ep(x) \psi(x) dx \to |B_1| \int_0^1 a^n(x) \hat{f}(x) \psi(x) dx.$$
\noindent Therefore,
$$  \int_0^1 a^n(x)\left( A_{01}(x) \partial_xu^0(x) \partial_x\psi(x) + u^0(x) \psi(x) \right) dx = \int_0^1 a^n(x) \hat{f}(x) \psi(x) dx.$$

This is exactly the limit equation satisfied by $u^0$.

To conclude that $u^\ep \to Eu^0$ strongly in $H^1(\Omega)$, we use the lower semicontinuity of the norm and the fact that we have weak convergence, $\|u^0\|_{H^1(\Omega)} \leq \liminf\limits_{\ep \to 0} \|u^\ep\|_{H^1(\Omega)}.$

Taking the limit
\begin{align*}
|B_1| \int_0^1 a^n(x) A_{01}(x) |u^0_x|^2 dx &= \int_\Omega A_{01}(x) |\partial_x u^0|^2 dxdy 
\leq \liminf_{\ep \to 0} \int_\Omega A_{01}(x) |\partial_x u^\ep|^2 dxdy \\
&\leq \limsup_{\ep \to 0} \int_\Omega A_{01}(x) |\partial_x u^\ep|^2 dxdy \leq \limsup_{\ep \to 0} \int_\Omega (f^\ep -u^\ep)u^\ep dxdy \\
&\leq \int_\Omega (\hat{f} - u^0)u^0 dxdy = \int_\Omega A_{01}(x) |\partial_x u^0|^2 dxdy.
\end{align*}

Therefore, $\lim\limits_{\ep \to 0} \|u^\ep\|_{H^1(\Omega)} = \|u^0\|_{H^1(\Omega)} \ \text{which implies } u^\ep \to Eu^0 \text{ strongly in } H^1(\Omega).$
\fimdemo

\medskip
If $u^0$ satisfies the variational formulation of the limit problem, then it solves the following one-dimensional equation
\begin{equation}\label{problim}
    \begin{cases}
- \frac{1}{a^n}\left( A_{01} a^n u^0_x \right)_x + u^0 = \hat{f} & x \in (0,1), \\
\lim\limits_{x \to 0^+} a^n(x) u^0_x(x) = 0, & \\
u_x^0(1) = 0. &
\end{cases}
\end{equation}



We now justify that the boundary condition
\begin{equation} \label{eq729}
\lim_{x \to 0^+} a^n(x) u^0_x(x) = 0
\end{equation}
is satisfied by the limit solution \( u^0 \). 
We consider the variational formulation associated with the limit problem. For all \( \varphi \in H_a^1(0,1) \), we have
\[
\int_0^1 A_{01}a^nu_x \varphi_x \, dx = \int_0^1 a^n (f-u) \varphi \, dx.
\]
Therefore, $-\left( A_{01}a^n u_x \right)_x = a^n(f-u), \ x \in (0,1).$

Now observe that
\[
\int_0^1 \left[ a^n(x)(f(x) - u(x)) \right]^2 dx \leq \int_0^1 a^n(x) \left[ f(x) - u(x) \right]^2 dx = \|f - u\|^2_{L^2_a(0,1)},
\]
so that $a^n(f - u) \in L^2(0,1), \ \text{and} \  A_{01}a^n u_x \in H^1(0,1).$

Suppose now, by contradiction, that the limit is not zero. Then there exists a constant \( K > 0 \) and a sequence \( x_k \to 0^+ \) such that $\left| a^n(x_k) A_{01}(x_k) u^0_x(x_k) \right| \geq K.$

Hence, we deduce
\[
a^n(x_k) A_{01}(x_k) \left| u^0_x(x_k) \right|^2 \geq \frac{K^2}{a^n(x_k) A_{01}(x_k)}.
\]
Since \( a^n(x_k) \to 0 \) and \( A_{01}(x_k) \) are strictly positive and bounded near zero, the right-hand side diverges as \( k \to \infty \), contradicting the fact that
\[
\int_0^1 a^n(x) A_{01}(x) \left| u^0_x(x) \right|^2 dx < \infty.
\]
Then, we can conclude \eqref{eq729} since $A_{01}$ is strictly positive. 
\medskip

We conclude this section with a convergence result that establishes a precise estimate for the difference between the solution \( u^\varepsilon \) of the perturbed problem \eqref{eq:Pe} and the extension \( E_\varepsilon u^0 \) of the limit solution \( u^0 \) defined in \eqref{eq:Po}. This result not only guarantees the convergence of \( u^\varepsilon \to E_\varepsilon u^0 \) in the energy norm, but also provides a sharp rate of convergence of order \( \mathcal{O}(\varepsilon) \). 

\begin{theorem}\label{Theo23}
Let $a$ be a function satisfying the condition \eqref{H1} and \( \{f^\ep\}_{\ep \in (0,\ep_0]} \subset L^2(R^\ep) \) a family of functions. Suppose that \( u^\ep \in H^1(R^\ep) \) and \( u^0 \in H^1_a(0,1) \) are solutions, respectively, to
\begin{align}
\int_{R^\ep} A^\ep(x, \tfrac{y}{\varepsilon})\nabla u^\ep \cdot \nabla \varphi + u^\ep \varphi \, dxdy &= \int_{R^\ep} f^\ep \varphi \, dxdy, \ \forall \varphi \in H^1(R^\ep), \label{eq:rep_problema}\\
\int_0^1 \tilde{A}(x) u^0_x \varphi_x +\! |B_1|a^nu^0\varphi\, dx &= |B_1|\!\!\int_0^1 a^n M_\ep f^\ep\varphi \, dx, \ \forall \varphi \in H^1_a(0,1), \label{eq:limite}
\end{align}
where $\tilde{A}(x) = |B_1|a^n(x)A_{01}(x)$. Then, there exists a constant \( C > 0 \), independent of \( \ep \) and \( f^\ep \), such that
\begin{equation} \label{teo23}
 \|u^\ep - E_\ep u^0\|_{H^1(R^\ep)} \leq \ep \,C \|f^\ep\|_{L^2(R^\ep)}.
\end{equation}

\end{theorem}
\begin{demo}
It is known that the minima

$$\lambda_{\ep} = \min_{\varphi \in H^{1}(R^{\ep})} \left\lbrace \frac{1}{2} \int_{R^{\ep}} (A^{\ep}(x, \tfrac{y}{\varepsilon}) \nabla \varphi) \cdot \nabla \varphi + |\varphi|^{2} dy dx - \int_{R^{\ep}} f^{\ep} \varphi \,dydx \right\rbrace$$       
$$\tau_{\ep} = \min_{\varphi \in H^{1}_{a}(0,1)} \left\lbrace \frac{1}{2} \int_{0}^{1} \tilde{A} (x) |\varphi_{x}|^{2} + a^{n}(x)|B_{1}||\varphi|^{2} dx - |B_{1}|\int_{0}^{1} a^{n}(x) M_{\ep} f^{\ep} \varphi dx \right\rbrace$$
are unique and attained at $u^{\ep}$ and $u^{0}$ respectively. We want to compare both solutions $u^{\ep}$ and $u^{0}$. Since $E_\ep u^{0} \in H^{1}(R^{\ep})$, we have

\begin{eqnarray*}
\hspace{-1cm}\lambda_{\ep} & \leq & \frac{1}{2} \int_{R^{\ep}}(A^{\ep}(x, \tfrac{y}{\varepsilon}) \nabla u^{0}) \cdot \nabla u^{0} + |u^{0}|^{2}dydx - \int_{R^{\ep}} f^{\ep} u^{0} \, dydx\\
& = & \frac{1}{2} \int_0^1 \ep^{n}|B_1|a^n\Big(A_{01}|u_{x}^{0}|^{2} + |u^{0}|^{2} + \sum_{k = 1}^{\infty} \ep^{k} A_{k1} |u_{x}^{0}|^{2}\Big)dx - \int_{R^{\ep}} f^{\ep} u^{0} \, dydx \\
& = & \left(\frac{1}{2}  \int_{0}^{1} \ep^{n}|B_1|a^n\Big(A_{01} |u_{x}^{0}|^{2} + |u^{0}|^{2}\Big) dx - \int_{0}^{1} \ep^{n}|B_1|a^nu^{0} M_{\ep} f^{\ep} dx \right) + \frac{1}{2} \int_{R^\ep} \sum_{k = 1}^{\infty} \ep^{k} A_{k1} |u_{x}^{0}|^{2} dy dx\\
& = & \ep^{n} \tau_{\ep} + \frac{1}{2} \int_{R^\ep} \sum_{k = 1}^{\infty} \ep^{k} A_{k1} |u_{x}^{0}|^{2}dydx.
\end{eqnarray*}

Hence, we obtain the following estimate
\begin{equation} \label{limsupautovalor}
   \lambda_{\ep} \leq \ep^{n} \tau_{\ep} + \frac{1}{2} \int_{R^\ep} \sum_{k = 1}^{\infty} \ep^{k} A_{k1} |u_{x}^{0}|^{2}dydx.   
\end{equation}

To obtain a lower estimate, we proceed as follows: we consider the decomposition $u^{\ep} = z^{\ep} + u^{0}$ where $z^{\ep} = u^{\ep} - u^{0}$. This yields

\begin{eqnarray*}
\lambda_{\ep} & = & \frac{1}{2} \int_{R^{\ep}} \ A^{\ep}(x, \tfrac{y}{\varepsilon}) \nabla (z^{\ep} + u^{0}) \cdot \nabla (z^{\ep} + u^{0}) + (z^{\ep} + u^{0})^2 \, dxdy- \int_{R^{\ep}} f^{\ep} (z^{\ep} + u^{0}) \, dydx \\
& = & \frac{1}{2} \int_{R^{\ep}} A^{\ep}(x, \tfrac{y}{\varepsilon}) \nabla z^{\ep} \cdot \nabla z^{\ep} + 2 A^{\ep}(x, \tfrac{y}{\varepsilon}) \nabla z^{\ep} \cdot \nabla u^{0} + A^{\ep}(x, \tfrac{y}{\varepsilon}) \nabla u^{0} \cdot \nabla u^{0}  + |z^{\ep}|^{2} + 2 z^{\ep}  u^{0}  + |u^{0}|^{2} \, dxdy\\
&  & - \int_{R^{\ep}} f^{\ep} z^{\ep} dxdy- \int_{R^{\ep}} f^{\ep} u^{0} dxdy \\
& = & \frac{1}{2} \int_{R^{\ep}} A^{\ep}(x, \tfrac{y}{\varepsilon}) \nabla (u^{\ep} - u^{0}) \cdot \nabla (u^{\ep} - u^{0}) + |u^{\ep} - u^{0}|^{2} dxdy + I_1 + I_2 - I_3 + \ep^n \tau_{\ep} \\
& & + \frac{1}{2} \int_{R^\ep} \sum_{k = 1}^{\infty} \ep^{k} A_{k1} |u_{x}^{0}|^{2} dxdy
\end{eqnarray*}

where
\begin{equation*}
    I_{1} := \int_{R^{\ep}} A^{\ep}(x, \tfrac{y}{\varepsilon}) \nabla (u^{\ep} - u^{0}) \cdot \nabla u^{0} dxdy, \ \  I_{2} := \int_{R^{\ep}} (u^{\ep} - u^{0}) u^{0} dxdy 
    \   \textrm{ and } \   I_{3} :=  \int_{R^{\ep}} f^{\ep} (u^{\ep} - u^{0}) dxdy. 
\end{equation*}

If we examine each term in detail, we observe the following

\begin{eqnarray*}
 I_{1} & = & 
  \int_{R^\ep} \Big( A_{01} u_{x}^{\ep} u_{x}^{0} + \!\sum_{k = 1}^{\infty} \ep^{k} A_{k1} u_{x}^{\ep} u_{x}^{0} + \!\sum_{k = 1}^{\infty} \ep^{k - 1} A_{k2} \nabla_{y} u^{\ep} u_{x}^{0} \Big) dxdy \\ 
 & -& \int_{R^\ep} \Big( A_{01} (u_{x}^{0})^{2} +\!\sum_{k = 1}^{\infty} \ep^{k} A_{k1} (u_{x}^{0})^{2} \Big) dxdy.
\end{eqnarray*}

Note that

\begin{eqnarray*}
 \int_{R^\ep} A_{01} u_{x}^{\ep} u_{x}^{0} & = & \int_{0}^{1} u_{x}^{0} \frac{ \ep^{n} a^{n} |B_{1}|}{\ep^{n} a^{n} |B_1|} \int_{\ep a(x) B_{1}} A_{01} u_{x}^{\ep} \, dydx  \ = \int_{R^{\ep}} M_{\ep} (A_{01} u_{x}^{\ep}) u_{x}^{0} \, dydx.
\end{eqnarray*}
Thus,
\begin{eqnarray*}
    I_{1} = \int_{R^{\ep}} (M_{\ep} (A_{01} u_{x}^{\ep}) - A_{01} u_{x}^{0}) u_{x}^{0} \, dydx + (*)
\end{eqnarray*}
with
\begin{equation*}
    (*) = \int_{R^\ep} \Big(\sum_{k = 1}^{\infty} \ep^{k} A_{k1} u_{x}^{\ep} u_{x}^{0} + \sum_{k = 1}^{\infty} \ep^{k - 1} A_{k2} \nabla_{y} u^{\ep} u_{x}^{0} - \sum_{k = 1}^{\infty} \ep^{k} A_{k1} (u_{x}^{0})^{2}\Big)dydx  .
\end{equation*}

Let us now add and subtract $A_{01} (M_{\ep} u^{\ep})_{x} u^{0}_{x}$ from the first term in $I_{1}$:
\begin{eqnarray*}
    I_{1} & = & \int_{R^{\ep}} (M_{\ep} (A_{01} u_{x}^{\ep}) - A_{01} (M_{\ep} u^{\ep})_{x}) u_{x}^{0} \, dydx + \int_{R^{\ep}} (A_{01} (M_{\ep} u^{\ep})_{x} - A_{01} u_{x}^{0}) u_{x}^{0} \, dydx + (*) \\
    & = & \int_{R^{\ep}} (M_{\ep} (A_{01} u_{x}^{\ep}) - A_{01} (M_{\ep} u^{\ep})_{x}) u_{x}^{0} \, dydx + \ep^{n} \int_{0}^{1}  \tilde{A}(x) (M_{\ep} u^{\ep} - u^{0})_{x} u_{x}^{0} dx + (*) \\
    & = & \tilde{I}_{1} +\ep^{n} \int_{0}^{1}  \tilde{A}(x) (M_{\ep} u^{\ep} - u^{0})_{x} u_{x}^{0} dx + (*).
\end{eqnarray*}

As for $I_{2}$ and $I_{3}$,
\begin{equation*}
    I_{2} = \ep^{n} |B_{1}| \int_{0}^{1}  a^{n} (M_{\ep} u^{\ep} - u^{0}) u^{0} dx
\end{equation*}
    and
\begin{eqnarray*}
    I_{3} & = &\int_{R^{\ep}} (f^{\ep} - M_{\ep} f^{\ep}) (u^{\ep} - u^{0}) \, dydx + \ep^{n} |B_{1}| \int_{0}^{1} a^{n} M_{\ep} f^{\ep} (M_{\ep} u^{\ep} - u^{0}) dx \\
    & = & \tilde{I}_{3} + \ep^{n} |B_{1}| \int_{0}^{1} a^{n} M_{\ep} f^{\ep} (M_{\ep} u^{\ep} - u^{0}) dx.
\end{eqnarray*}

Since $u^{0}$ is the solution of the limit problem and $M_{\ep} u^{\ep} \in H^{1}_{a} (0,1)$, we have
\begin{equation*}
    \int_{0}^{1} \tilde{A} (x) (M_{\ep} u^{\ep} - u^{0})_{x} u_{x}^{0} + |B_{1}|a^{n} (M_{\ep} u^{\ep} - u^{0}) u^{0} = |B_{1}| \int_{0}^{1} a^{n} M_{\ep} f^{\ep} (M_{\ep} u^{\ep} - u^{0}) dx.
\end{equation*}

Consequently, from the previous identities, $ I_{1} - \tilde{I}_{1} - (*) + I_{2} = I_{3} - \tilde{I}_{3}.$ Therefore, it follows that
\begin{equation*}
    I_{1} + I_{2} - I_{3} = \tilde{I}_{1} - \tilde{I}_{3} + (*).
\end{equation*}

Substituting this relation into the expression for $\lambda_\ep$, we obtain
\begin{equation} \label{expautovalor}
\begin{gathered}
     \lambda_{\ep} =   \frac{1}{2} \int_{R^{\ep}} A^{\ep} \nabla (u^{\ep} - u^{0}) \cdot \nabla (u^{\ep} - u^{0})  + |u^{\ep} - u^{0}|^{2} dy dx  \\
\qquad \qquad + \tilde{I}_{1} - \tilde{I}_{3} + (*) + \ep^{n} \tau_{\ep} + \frac{1}{2} \int_{R^\ep} \sum_{k = 1}^{\infty} \ep^{k} A_{k1} |u_{x}^{0}|^{2}dydx.
\end{gathered}
\end{equation}

\noindent So, we only need to estimate $\tilde{I}_1$, $\tilde{I}_3$, and the remainder term $(*)$.

\smallskip

By adding and subtracting the term $A_{01} M_{\varepsilon} u^\varepsilon_x$ within $\tilde{I}_1$,
\begin{align*}
\tilde{I}_1 
&= \int_{R^\varepsilon} \left( M_{\varepsilon}(A_{01} u^\varepsilon_x) - A_{01} M_{\varepsilon} u^\varepsilon_x + A_{01} M_{\varepsilon} u^\varepsilon_x - A_{01}(M_{\varepsilon} u^\varepsilon)_{x} \right) u^0_x \, dy dx \\
&= \int_{R^\varepsilon} A_{01} \left( M_{\varepsilon} u^\varepsilon_x - (M_{\varepsilon} u^\varepsilon)_{x} \right) u^0_x \, dy dx.
\end{align*}

\medskip

From Lemma \ref{teo21}, we know that
\begin{equation}\label{eq:313}
    (M_\varepsilon u)_x(x) = \frac{1}{|B_1|} \left( \int_{B_1} u_x(x, \varepsilon l_x(z)) \, dz + \int_{B_1} \nabla_y u(x, \varepsilon l_x(z)) \cdot \varepsilon a'(x) z \, dz \right).
\end{equation}

For the first integral in \eqref{eq:313}, we have
\[
\frac{1}{|B_1|} \int_{B_1} u^\varepsilon_x(x, \varepsilon l_x(z)) \, dz = \frac{1}{\varepsilon^n a^n(x) |B_1|} \int_{\varepsilon a(x) B_1} u^\varepsilon_x(x, y) \, dy = M_\varepsilon u^\varepsilon_x,
\]

and for the second integral in \eqref{eq:313}, since \( \|a'(x) z\| \leq C_4\) for all \( z \in B_1 \),
\[
\left| \int_{B_1} \nabla_y u^\varepsilon(x, \varepsilon l_x(z)) \cdot \varepsilon a'(x) z \, dz \right|
\leq C_4 \,\varepsilon \int_{B_1} |\nabla_y u^\varepsilon(x, \varepsilon l_x(z))| \, dz.
\]

Changing the variables back to \( y = \varepsilon a(x) z \), we obtain
\[
\left| \frac{1}{|B_1|} \int_{B_1} \nabla_y u^\varepsilon(x, \varepsilon l_x(z)) \cdot \varepsilon a'(x) z \, dz \right|
\leq \frac{C_4 \varepsilon}{|B_1| \varepsilon^n a^n(x)} \int_{\varepsilon a(x) B_1} |\nabla_y u^\varepsilon(x, y)| \, dy.
\]

Therefore,
\[
\left| M_\varepsilon u^\varepsilon_x - (M_\varepsilon u^\varepsilon)_x \right|
\leq \frac{C_4 \varepsilon}{|B_1| \varepsilon^n a^n(x)} \int_{\varepsilon a(x) B_1} |\nabla_y u^\varepsilon(x, y)| \, dy.
\]

Thus,
%
\begin{align*}
|\tilde{I}_1| & \leq \int_0^1 C_4 |A_{01}(x)| \varepsilon \int_{\varepsilon a(x) B_1} |\nabla_y u^\varepsilon| |u^0_x| \, dy dx.
\end{align*}

Finally, applying Hölder's inequality, we obtain
\begin{equation}\label{tildeI1}
|\tilde{I}_1| \leq C_4 \|A_{01}\|_{L^\infty(0,1)} \varepsilon \|\nabla_y u^\varepsilon\|_{L^2(R^\varepsilon)} \|u^0_x\|_{L^2(R^\varepsilon)}.
\end{equation}

And $\tilde{I}_3$ can be estimated as follows
\begin{align*}
    \tilde{I}_3 &= \int_{R^\varepsilon} (f_\varepsilon - M_\varepsilon f_\varepsilon)(u^\varepsilon - M_\varepsilon u^\varepsilon)\,dydx 
    + \underbrace{
\int_{R^\varepsilon} (f_\varepsilon - M_\varepsilon f_\varepsilon)(M_\varepsilon u^\varepsilon - u^0) \, dxdy
}_{= 0}\\
    &= \int_{R^\varepsilon} (f_\varepsilon - M_\varepsilon f_\varepsilon)(u^\varepsilon - M_\varepsilon u^\varepsilon)\,dydx.
\end{align*}

By Hölder's inequality, we have
\begin{equation*}
    |\tilde{I}_3| \leq \|f_\varepsilon - M_\varepsilon f_\varepsilon\|_{L^2(R^\varepsilon)} \|u^\varepsilon - M_\varepsilon u^\varepsilon\|_{L^2(R^\varepsilon)}.
\end{equation*}

By Lemma \ref{lemma22}, $\|u^\ep - E_\ep M_\ep u^\ep\|^2_{L^2(R^\ep)} \leq \gamma \ep^2 \|\nabla_y u^\ep\|^2_{L^2(R^\ep)}.$ Consequently, 
\begin{equation} \label{tildeI3}
    |\tilde{I}_{3}| \leq \ep\, C_{9} || f^{\ep} - M_{\ep} f^{\ep} ||_{L^{2}(R^{\ep})} || \nabla_{y} u^{\ep} ||_{L^{2}(R^{\ep})},
\end{equation} where $C_9=\gamma^{\frac{1}{2}}.$
\

At this point, we need to estimate $|| u_{x}^{0} ||_{L^{2}(R^{\ep})}$ and $|| \nabla_{y} u^{\ep} ||_{L^{2}(R^{\ep})}$. Note that
\begin{equation} \label{nablayuep}
    || \nabla_{y} u^{\ep} ||_{L^{2}(R^{\ep})}^{2} 
   \leq \int_{R^{\ep}} | \nabla_{y} u^{\ep} - \nabla_{y} u^{0} |^{2} + | \nabla_{y} u^{0} |^{2} \, dydx \ \ \leq || \nabla u^{\ep} - \nabla u^{0} ||_{L^{2}(R^{\ep})}
\end{equation}

For $|| u_{x}^{0} ||_{L^{2}(R^{\ep})}$, since $u^{0}$ satisfies the limit problem,
\begin{eqnarray*}
    \int_{0}^{1} \tilde{A}(x) (u_{x}^{0})^{2} dx + |B_{1}|\int_{0}^{1} a^{n} (u^{0})^{2} dx & = & |B_{1}| \int_{0}^{1} a^{n} M_{\ep} f^{\ep} u^{0} dx \\
    & \leq & |B_{1}| \left( \int_{0}^{1} a^{n} (M_{\ep} f^{\ep})^{2} dx \right)^{\frac{1}{2}} \left( \int_{0}^{1} a^{n} (u^{0})^{2} dx \right)^{\frac{1}{2}}\\
    & \leq & |B_{1}| \left( \frac{1}{4 \delta}\int_{0}^{1} a^{n} (M_{\ep} f^{\ep})^{2} dx +  \delta \int_{0}^{1} a^{n}(u^{0})^{2} dx \right)
\end{eqnarray*}
\\
\noindent where $\delta$ is a constant that arises from the application of Young's inequality. Setting $\delta = \frac{1}{2}$ and grouping the terms in the previous inequality, we obtain

\begin{equation*}
    \frac{1}{2} \left( \int_{0}^{1} \tilde{A}(x) (u_{x}^{0})^{2} dx + |B_{1}| \int_{0}^{1} a^{n} (u^{0})^{2} dx \right) \leq \frac{|B_{1}|}{2} \int_{0}^{1} a^{n} (M_{\ep} f^{\ep})^{2} dx.
\end{equation*}
\\
\noindent Now, since $\int_{B_{1}} A_{01} dy  = A_{01}(x) |B_1| > 0$, we have
\begin{equation*}
    \min_{x \in [0,1]} \left( \int_{B_{1}} A_{01} dy \right) \int_{0}^{1} a^{n} (u_{x}^{0})^{2} \leq |B_{1}| ||M_{\ep} f^{\ep} ||^{2}_{L^{2}_{a}(0,1)}.
\end{equation*}
\noindent Moreover, since $|| u_{x}^{0} ||^{2}_{L^{2}(R^{\ep})} = \ep^{n} |B_{1}| \int_{0}^{1} a^{n} (u_{x}^{0})^{2} dx,$ it follows that
\begin{equation*}
    || u_{x}^{0} ||^{2}_{L^{2}(R^{\ep})} \leq \frac{\ep^{n} |B_{1}|^{2}}{C_{5}} ||M_{\ep} f^{\ep} ||^{2}_{L^{2}_{a}(0,1) }
\end{equation*}
\noindent where $C_{5} = \min\limits_{x \in [0,1]} \left( \int_{B_{1}} A_{01} dy \right)$. Since $|| M_{\ep} f^{\ep} ||^{2}_{L^{2}_{a}(0,1)} \leq \frac{\ep^{-n}}{|B_{1}|} ||f^{\ep} ||^{2}_{L^{2}(R^{\ep})}$, we obtain

\begin{equation} \label{ux0}
    || u_{x}^{0} ||_{L^{2}(R^{\ep})} \leq \Big(\frac{|B_{1}|}{C_{5}}\Big)^{\frac{1}{2}}||f^{\ep} ||_{L^{2}(R^{\ep})} = C_{6} ||f^{\ep} ||_{L^{2}(R^{\ep})}.
\end{equation}
\\
Combining the estimates \eqref{nablayuep} and \eqref{ux0} with \eqref{tildeI1} and \eqref{tildeI3}, we derive the following bounds.

For $\tilde{I}_{1}$, we obtain
\begin{eqnarray} \label{tildei1fim}
    |\tilde{I}_{1}| & \leq & C_8 \, \varepsilon \, \| \nabla u^{\varepsilon} - \nabla u^{0} \|_{L^2(R^{\varepsilon})} \, \| f^{\varepsilon} \|_{L^2(R^{\varepsilon})} \nonumber \\
    & \leq & \frac{3}{\alpha} C_8^2 \varepsilon^2 \| f^{\varepsilon} \|_{L^2(R^{\varepsilon})}^2 + \frac{\alpha}{12} \| \nabla u^{\varepsilon} - \nabla u^{0} \|_{L^2(R^{\varepsilon})}^2,
\end{eqnarray}
where \( C_8 = C_4 C_6 \| A_{01} \|_{L^\infty(0,1)} \), and we applied Young's inequality with parameter \( \frac{3}{\alpha} \), where $\alpha > 0$ denotes the coercivity constant of the matrix $A^\ep$ (see \eqref{coer}).

Similarly, for $\tilde{I}_{3}$, we have
\begin{eqnarray} \label{tildei3fim}
    |\tilde{I}_{3}| & \leq & C_9 \, \varepsilon \, \| f^{\varepsilon} - M_{\varepsilon} f^{\varepsilon} \|_{L^2(R^{\varepsilon})} \, \| \nabla_y u^{\varepsilon} \|_{L^2(R^{\varepsilon})} \nonumber \\
    & \leq & \frac{3}{\alpha}\, C_9^2\, \varepsilon^2 \| f^{\varepsilon} - M_{\varepsilon} f^{\varepsilon} \|_{L^2(R^{\varepsilon})}^2 + \frac{\alpha}{12} \| \nabla u^{\varepsilon} - \nabla u^{0} \|_{L^2(R^{\varepsilon})}^2.
\end{eqnarray}

It remains to estimate the remainder term $(*)$ appearing in the expression of $I_{1}$. We begin by bounding each component separately.

\begin{eqnarray*}
    \left|\int_{R^\ep} \sum_{k = 1}^{\infty} \ep^{k} A_{k1} (u_{x}^{\ep} - u_{x}^{0})u_{x}^{0} \right| & \leq &  \sum_{k = 1}^{\infty} \ep^{k}  \left| \int_{R^{\ep}}  A_{k1} (u_{x}^{\ep} - u_{x}^{0})u_{x}^{0} \right| \\
    & \leq & \frac{\ep}{1 - \ep} \sup_{k, \ep} ||A_{k1}||_{L^{2}(R^{\ep})}  || u_{x}^{\ep} - u_{x}^{0} ||_{L^{2}(R^{\ep})}  ||u_{x}^{0}||_{L^{2}(R^{\ep})}  \\
    & \leq & C_{7} \ep || \nabla u^{\ep} - \nabla u^{0} ||_{L^{2}(R^{\ep})} C_{6} ||f^{\ep}||_{L^{2}(R^{\ep})} \\
    & \leq & \frac{3}{\alpha} C_{7}^{2} \ep^{2} ||f^{\ep} ||^{2}_{L^{2}(R^{\ep})} + \frac{\alpha}{12} || \nabla u^{\ep} - \nabla u^{0} ||^{2}_{L^{2}(R^{\ep})} 
\end{eqnarray*}
and
\begin{eqnarray*}
    \left|\int_{R^\ep}\sum_{k = 1}^{\infty} \ep^{k-1} A_{k2} \nabla_{y} u^{\ep} u_{x}^{0} \right| & \leq &  \left| \sum_{k = 1}^{\infty} \ep^{k-1} \int_{R^{\ep}} A_{k2} \nabla_{y} u^{\ep} u_{x}^{0} \right|  \\
    & \leq &  \frac{\ep}{1 - \ep} \sup_{k, \ep} ||A_{k2}||_{L^{2}(R^{\ep})} || \nabla_{y} u^{\ep} ||_{L^{2}(R^{\ep})} ||u_{x}^{0}||_{L^{2}(R^{\ep})} \\
    & \leq & C_{8} \ep || \nabla u^{\ep} - \nabla u^{0} ||_{L^{2}(R^{\ep})} C_{6} ||f^{\ep}||_{L^{2}(R^{\ep})} \\
    & \leq & \frac{3}{\alpha} C_{11}^{2} \ep^{2} ||f^{\ep}||_{L^{2}(R^{\ep})} + \frac{\alpha}{12} || \nabla u^{\ep} - \nabla u^{0} ||^{2}_{L^{2}(R^{\ep})}  
\end{eqnarray*}
\noindent where Young's inequality was applied with the constant $\frac{3}{\alpha}$. Thus, we obtain
\begin{equation} \label{*}
    |(*)| \leq \frac{3}{\alpha}\ep^{2}(C_{7}^{2} + C_{11}^{2}) ||f^{\ep}||^{2}_{L^{2}(R^{\ep})} + \frac{2 \alpha}{12} || \nabla u^{\ep} - \nabla u^{0} ||^{2}_{L^{2}(R^{\ep})}  
\end{equation}
From (\ref{tildei1fim}), (\ref{tildei3fim}), and (\ref{*}), we derive
\begin{equation}\label{eq:318}
    |\tilde{I}_{1} - \tilde{I}_{3} + (*)| \leq \frac{3}{\alpha} ( C_{7}^{2} + C_{8}^{2} + C_{11}^{2}) \ep^{2} ||f^{\ep}||^{2}_{L^{2}(R^{\ep})} + \frac{\alpha}{3} || \nabla u^{\ep} - \nabla u^{0} ||^{2}_{L^{2}(R^{\ep})} + \frac{3}{\alpha} C_{9}^{2} \ep^{2} ||f^{\ep}  - M_{\ep}f^{\ep}||^{2}_{L^{2}(R^{\ep})}
\end{equation}
\\
Then, using \eqref{limsupautovalor}, \eqref{eq:318} and the uniform coercivity of the matrix $A^\varepsilon$ in \eqref{expautovalor}, we obtain

\begin{eqnarray*}
    \lambda_{\ep} 
    & \geq & \frac{\alpha}{2} \int_{R^{\ep}} \nabla (u^{\ep} - u^{0}) \cdot \nabla (u^{\ep} - u^{0})  + |u^{\ep} - u^{0}|^{2} dy dx - \frac{3}{\alpha} ( C_{7}^{2} + C_{8}^{2} + C_{11}^{2}) \ep^{2} ||f^{\ep}||^{2}_{L^{2}(R^{\ep})} \\
    & - & \frac{\alpha}{3} || \nabla u^{\ep} - \nabla u^{0} ||^{2}_{L^{2}(R^{\ep})} - \frac{3}{\alpha} C_{9}^{2} \ep^{2} ||f^{\ep}  - M_{\ep}f^{\ep}||^{2}_{L^{2}(R^{\ep})}  \\
    & \geq & \frac{\xi}{6} \int_{R^{\ep}} \nabla (u^{\ep} - u^{0}) \cdot \nabla (u^{\ep} - u^{0}) + |u^{\ep} - u^{0}|^{2} dy dx -\frac{3}{\alpha} ( C_{7}^{2} + C_{8}^{2} + C_{11}^{2}) \ep^{2} ||f^{\ep}||^{2}_{L^{2}(R^{\ep})} \\
    & - & \frac{3}{\alpha} C_{9}^{2} \ep^{2} ||f^{\ep}  - M_{\ep}f^{\ep}||^{2}_{L^{2}(R^{\ep})},
\end{eqnarray*}
where $\xi=\min\{\al, 3\}$, which yields
\begin{eqnarray*}
    \frac{\xi}{6} ||u^{\ep} - E_{\ep}u^{0} ||^{2}_{H^{1}(R^{\ep})} & \leq & \frac{3}{\alpha} ( C_{7}^{2} + C_{8}^{2} + C_{11}^{2}) \ep^{2} ||f^{\ep}||^{2}_{L^{2}(R^{\ep})} + \frac{3}{\alpha} C_{9}^{2}  \ep^{2} 4 ||f^{\ep}||^{2}_{L^{2}(R^{\ep})} \\
    & \leq & \frac{3}{\alpha} ( C_{7}^{2} + C_{8}^{2} + (2C_{9})^{2} + C_{11}^{2}) \ep^{2} ||f^{\ep}||^{2}_{L^{2}(R^{\ep})}
\end{eqnarray*}
\noindent where $|| f^{\ep} - M_{\ep} f^{\ep} ||_{L^{2}(R^{\ep})} \leq 2 ||f^{\ep}||_{L^{2}(R^{\ep})}$.
\

So,
\begin{equation*}
    ||u^{\ep} - E_{\ep}u^{0} ||^{2}_{H^{1}(R^{\ep})} \leq \frac{18}{\xi^{2}} ( C_{7}^{2} + C_{8}^{2} + (2C_{9})^{2} + C_{11}^{2}) \ep^{2} ||f^{\ep}||^{2}_{L^{2}(R^{\ep})}.
\end{equation*}
\noindent
Therefore, we conclude that
\begin{equation*}
    \|u^{\varepsilon} - E_{\varepsilon} u^{0} \|_{H^{1}(R^{\varepsilon})} 
    \leq C \varepsilon \|f^{\varepsilon} \|_{L^{2}(R^{\varepsilon})},
\end{equation*}
\noindent
where $C= (\frac{18}{\xi^{2}} ( C_{7}^{2} + C_{8}^{2} + (2C_{9})^{2} + C_{11}^{2}))^{\frac{1}{2}}$.\fimdemo
\end{demo}
\begin{remark}
    We note that the above result follows an approach that strongly depends on the boundedness of the derivative $a'$. Although similar convergence properties can be obtained under weaker assumptions, the explicit rate of convergence derived through this method requires the regularity condition $a' \in L^\infty(0,1)$.
\end{remark}

To compare the norms of the functions defined in the thin domain \( R^\varepsilon \) and in the fixed domain \( \Omega  \), we introduce the scaling operator
\[
i_\varepsilon : H^1(R^\varepsilon) \to H^1(\Omega), \quad i_\varepsilon w(x,y) := w(x, \varepsilon y).
\]
A straightforward change of variables yields the following norm equivalences:
\begin{equation*} \label{normi}
    \| i_\varepsilon w \|_{L^2(\Omega)}^2 = \varepsilon^{-n} \| w \|_{L^2(R^\varepsilon)}^2,
\quad
\| i_\varepsilon w \|_{H^1_\varepsilon(\Omega)}^2 = \varepsilon^{-n} \| w \|_{H^1(R^\varepsilon)}^2.
\end{equation*}

Note that if \( w^\varepsilon \) is a solution of (\ref{eq:rep_problema}), then \( u^\varepsilon = i_\varepsilon w^\varepsilon \) is a solution of
\[
   \int_\Omega A^\varepsilon\left(x, y\right) \nabla^\ep u^\varepsilon \cdot \nabla^\ep \varphi + u^\varepsilon \varphi\, dxdy = \int_\Omega g^\ep\varphi\, dxdy, \ \ \ \forall \varphi \in H^1(\Omega),
\]
where $i_\ep f^\ep=g^\ep$. Therefore, proving \eqref{teo23} is equivalent to showing that
\[
\|u^\varepsilon - E v^\varepsilon\|_{H^1_\varepsilon(\Omega)} \leq C \varepsilon \|g^\varepsilon\|_{L^2(\Omega)},
\]
where \( v^\varepsilon \) satisfies
\[
  \begin{cases}
    - \dfrac{1}{a^n(x)}\left( a^n(x) A_{01}(x) v^\varepsilon_x \right)_x + v^\varepsilon = M g^\varepsilon, & x \in (0,1) \\
    \lim\limits_{x \to 0^+}a^n(x) v^\varepsilon_x(x) = 0, \\
    v^\varepsilon_x(1) = 0.
  \end{cases}
\]


%

\begin{corollary}\label{cor:conresolv}
Let \( f^\varepsilon \in L^2(\Omega) \). Then the following estimate holds:
\[
\left\| L_\varepsilon^{-1} f^\varepsilon - E L_0^{-1} M f^\varepsilon \right\|_{H^1_\ep(\Omega)} 
\leq C \varepsilon \left\| f^\varepsilon \right\|_{L^2(\Omega)}.
\]
\end{corollary}

\demo
This result is a direct consequence of Theorem \ref{Theo23}.
\fimdemo

\begin{remark}
The results established above do not require that the limit operator $L_0$ has compact resolvent.
\end{remark}

\subsection{Compacity of the resolvent of the limit operator}
In this section, our aim is to show that the limit operator has compact resolvent in the appropriate weighted Sobolev space. From now on, we assume the hyphotesis \eqref{H2} so that the proof follows the strategy of \cite{campiti}, adapting the techniques to the weighted framework imposed by the geometry of the domain and the degeneracy introduced by the function $a(\cdot)$.

We introduce the auxiliary subspace
\begin{equation*}
    F_{2} = \Big\{ u \in L^{2}_{a} (0,1) ; \int_{0}^{1} a^{n} u  = 0 \Big\},
    \end{equation*}
    which consists of functions orthogonal to the constant functions in the space $L^{2}_{a}(0,1)$.

\begin{theorem} \label{baraocp}
    Assume \eqref{H2} holds; in particular, suppose $W\in L^{2}_{a}(0,1)$.  Then the operator $\tilde{L}_0|_{F_2}$ is invertible and $( \tilde{L}_0|_{F_2})^{-1}$ is compact, where $\tilde{L}_{0}= - \frac{1}{a^{n}} (A_{01} a^{n} u_{x})_{x}$.
\end{theorem}

\demo 
The ideas of this proof can be found in \cite[Theorem 3.1]{campiti}, using the function $W$ as defined above and adopting the case to the space $L^{2}_{a}(0,1)$.
\fimdemo

\

To be able to affirm that our operator $L_{0} = \tilde{L}_0 + I$ has compact resolvent, we need to observe the following
\begin{proposition} \label{dirsuml2a}
    The space $L^{2}_{a}(0,1)$ admits the orthogonal decomposition $$L^{2}_{a}(0,1) = \ker(\tilde{L}_0) \oplus F_{2},$$ where $\ker(\tilde{L}_0)=$ {\it span}$\{1\},$ the space of constant functions.
\end{proposition}

\demo 
We start by showing that the kernel of the operator $\tilde{L}_0$ is given by $\ker(\tilde{L}_0) = \textit{span}\{1\}$.

Let $ u \in D(\tilde{L}_0)$ such that $\tilde{L}_0 u = 0$. It follows that $\left( A_{01}(x) a^n(x) u_x(x) \right)_x = 0.$ Thus, there exists a constant $ C \in \mathbb{R}$ such that $A_{01}(x) a^n(x) u_x(x) = C.$ Since the function $u$ satisfies the boundary condition, we conclude $C = 0.$

Therefore, as $A_{01}(x) a^n(x) > 0$ in $(0,1)$, we have $u_x(x) = 0$ in $(0,1)$, so $u(x) = \text{constant}$, which proves that $\ker(\tilde{L}_0) = \textit{span}\{1\} $.

Next, we show that every $f \in L^2_a(0,1)$ can be written as the sum of a constant and a function in $F_2$ . Let $f \in L^2_a(0,1)$ and decompose it as $f = c + g,$ where $c \in \mathbb{R}$ and  $g := f - c.$ Define $$c := \frac{\int_0^1 a^n(x) f(x)\,dx}{\int_0^1 a^n(x)\,dx}.$$

Now, we verify that $g \in F_2$ by checking that it is orthogonal to the constant function $1$ in $ L^2_a(0,1)$, that is
$$\int_0^1 a^n(x) g(x)\,dx =  \int_0^1 a^n(x) f(x)\,dx - c \int_0^1 a^n(x)\,dx.$$
\noindent But from the definition of $c $, we have
$c \int_0^1 a^n(x)\,dx = \int_0^1 a^n(x) f(x)\,dx.$
Therefore,
$$\int_0^1 a^n(x) g(x)\,dx = 0,$$
which proves that $g \in F_2$.

 \noindent\textit{Uniqueness.} Suppose there exists another decomposition $f = c' + g'$  with $g' \in F_2$. Then
$$c + g = c'+ g' \quad \Longrightarrow\quad (c - c')  = g' - g.$$
 Since $g, g' \in F_2$, it follows that $g' - g \in F_2$, and thus $(c - c')  \in F_2$. Hence, $c = c'$ and consequently $g = g' $. This proves that the decomposition is unique.
\fimdemo

\

By Proposition \ref{dirsuml2a} given $u \in L^{2}_{a}(0,1)$, it follows that $\tilde{L}_0 u = \tilde{L}_0|_{F_2}(u)$. Therefore, by Theorem \ref{baraocp}, $\tilde{L}_{0}$ has compact resolvent. As the operator $L_{0}$ is the sum of an operator with compact resolvent and a bounded operator, we conclude that $L_{0}$ has compact resolvent. 

\section{Abstract settings and  analysis of the nonlinear terms}\label{sec:5}

For each $\varepsilon \in [0, \varepsilon_0]$, to rigorously compare solutions in varying Hilbert spaces $Z_\varepsilon$ with those in the limit space $Z_0$, we adopt the framework of $E$–convergence, a standard tool for problems in domains depending on a small parameter; see, e.g., \cite{Carvalho-Piskarev:06,vainikkko}.

\begin{definicao}
    We say that the sequence $\{u^{\varepsilon}\}_{\varepsilon \in (0,\varepsilon_0]}$ $E$-converges to $u^0$ if $\Vert u^{\varepsilon} - Eu^0\Vert_{Z_\ep}\stackrel{\varepsilon \rightarrow 0}{\longrightarrow} 0.$ We write this as $u^{\varepsilon}\stackrel{E}{\longrightarrow}u^0.$
\end{definicao}

\begin{definicao}
    Let $\mathcal{B}_{\varepsilon} \subset Z_\ep$ and $\mathcal{B}_{0} \subset X_{0}$. Denote by $\mathrm{dist}(\cdot,\cdot)$ the metric induced by the norm in $Z_\ep$, that is, $\mathrm{dist}(u^{\varepsilon},v^{\varepsilon})=\Vert u^{\varepsilon} - v^{\varepsilon}\Vert_{Z_\ep}$.

    \emph{(i)} We say that the family of sets $\{\mathcal{B}_{\varepsilon}\}_{\varepsilon \in [0, \varepsilon_0]}$ is $E$-upper semicontinuous at $\varepsilon=0$ if $$\sup\limits_{u^{\varepsilon} \in \mathcal{B}_{\varepsilon}}\mathrm{dist}(u^{\varepsilon}, E\mathcal{B}_0)\overset{\varepsilon\rightarrow 0}{\longrightarrow} 0;$$

    \emph{(ii)} We say that the family of sets $\{\mathcal{B}_{\varepsilon}\}_{\varepsilon \in [0, \varepsilon_0]}$ is $E$-lower semicontinuous at $\varepsilon=0$ if$$\sup\limits_{u^{\varepsilon} \in \mathcal{B}_{\varepsilon}}\mathrm{dist}(Eu^{\varepsilon},\mathcal{B}_{\varepsilon})\overset{\varepsilon\rightarrow 0}{\longrightarrow} 0;$$
\end{definicao}

\begin{remark}
   The following characterizations are useful  to show the upper or lower semicontinuity of sets $($for more details see \cite[Lemma 1]{Carvalho-Piskarev:06}$)$:
   
   \emph{(i)} If any sequence $\{u^{\varepsilon}\}_{\varepsilon \in (0,\varepsilon_0]}$ with $u^{\varepsilon} \in \mathcal{B}_{\varepsilon}$ has an $E$-convergent subsequence with limit belonging to $\mathcal{B}_0$, then $\{\mathcal{B}_{\varepsilon}\}$ is $E$-upper semicontinuous at zero.

    \emph{(ii)} If $\mathcal{B}_0$ is compact and for any $u^0 \in \mathcal{B}_0$ there is sequence $\{u^{\varepsilon}\}_{\varepsilon \in (0,\varepsilon_0]}$ with $u^{\varepsilon} \in \mathcal{B}_{\varepsilon}$ which $E_{\varepsilon}$-converges to $u^0$, then $\{\mathcal{B}_{\varepsilon}\}$ is $E$-lower semicontinuous at zero.
\end{remark}
From this notion of convergence, we introduce the definition of compactness.

\begin{definicao}
\emph{(i)} A sequence $\{u^{\varepsilon_n}\}_{n \in \mathbb{N}} \subset Z_{\varepsilon_n}$ with $\varepsilon_n \to 0$ is said to be pre-compact if every subsequence has a further subsequence that $E$-converges to some $u^0 \in Z_0$.

\emph{(ii)} A family of bounded operators $B_\varepsilon \colon Z_\varepsilon \to Z_\varepsilon$ is said to converge in the $EE$-sense to a bounded operator $ B_0 \colon Z_0 \to Z_0$, if $B_{\varepsilon}u^{\varepsilon}\stackrel{E}{\longrightarrow}B_0u^0$ whenever $u^{\varepsilon}\stackrel{E}{\longrightarrow}u^0 \in Z_0.$ We write this as $B_{\varepsilon}\stackrel{EE}{\longrightarrow}B_0.$

\emph{(iii)} A family of compact operators $B_\varepsilon \colon Z_\varepsilon \to Z_\varepsilon$ is said to compactly converge to a compact operator $B_0 \colon Z_0 \to Z_0$, denoted by $B_\varepsilon \stackrel{CC}{\longrightarrow} B_0$, if $B_\varepsilon \stackrel{EE}{\longrightarrow} B_0$ and, for every uniformly bounded family $\{u^\varepsilon\}_{\ep\in(0,\ep_0]} \subset Z_\varepsilon$, the set $\{ B_\varepsilon u^\varepsilon \}_{\ep\in(0,\ep_0]}$ is pre-compact.

\end{definicao}
\begin{lemma}
\label{lemma:compactconv}
The family of compact operators $\{ L_\varepsilon^{-1}\}_{\ep\in(0,\ep_0]}$ converges compactly to $L_0^{-1}$ in $X_\ep^\al$ as $\varepsilon \to 0$.
\end{lemma}
\demo
Let \( \{ f^\varepsilon \}_{\ep \in(0,\ep_0]}\) be such that \( \| f^\varepsilon \|_{X_\ep} \leq 1 \). We want to show that \( \{ L_\varepsilon^{-1} f^\varepsilon \}_{\ep\in(0,\ep_0]} \) is pre-compact in \( X^{\frac{1}{2}}_\ep \) via the extension operator \( E \). 

We observe that for any subsequence, the following estimate holds $\| M f^\varepsilon \|_{X_0}^2 \leq C \| f^\varepsilon \|_{X^{\frac{1}{2}}_\ep},$ so \( \{ M f^\varepsilon \}_{\ep \in(0,\ep_0]} \) is bounded in \( X_0 \), and there exists a subsequence (still denoted \( f^\varepsilon \)) such that \( M f^\varepsilon \rightharpoonup f^0 \) in \( X^{\frac{1}{2}}_0 \).

By Theorem \ref{thm:convergencia} and Corollary \ref{cor:conresolv}, there exists \( u^\varepsilon = L_\varepsilon^{-1} f^\varepsilon \) such that \( u^\varepsilon \to Eu^0 \) strongly in \( H^1(\Omega) \), with \( u^0 = L_0^{-1} f^0 \). 

This shows that the family \( \{ L_\varepsilon^{-1} \}_{\ep\in(0,\ep_0]} \) converges compactly to \( L_0^{-1} \) in $X^\al_\ep$ as \( \varepsilon \to 0 \), where we used that \( H^1_{\varepsilon}(\Omega)=X^{1/2}_{\varepsilon} \hookrightarrow X^{\alpha}_{\varepsilon} \) for all \( \alpha\in[0,1/2) \).
\fimdemo

\

The following result is a direct consequence of Theorem \ref{Theo23}.
\begin{corollary}\label{cor:convergencia}
For $\ep\in(0,\ep_0]$,  there exist  a function \( \nu(\varepsilon) \to 0 \) as \( \varepsilon \to 0 \), such that 
\begin{equation*}\label{corconv}
    \| L_\varepsilon^{-1} - E L_0^{-1} M \|_{\mathcal{L}(X_\ep, X^{\al}_\ep)} \leq \nu(\varepsilon).
\end{equation*}
\end{corollary}

We now turn to the nonlinear part of the problem. In order to carry out the perturbation analysis of the semigroups and to quantify the convergence of equilibria and attractors, we  present uniform properties of the nonlinearities induced by the reaction term $f$ via the associated Nemitskii's operators $F_\varepsilon$ (and $F_0$) on the spaces $X^\alpha_\varepsilon$ and $X^\alpha_0$. 

\begin{remark}
Under \eqref{H3}, the weighted structure induced by $a(\cdot)$ and $A_{01}(\cdot)$ guarantees the integrability needed to control the compositions with $f$ in the relevant spaces. In this setting, the Nemitskii's operators $F_\varepsilon(u)=f(u)$ and $F_0(u)=f(u)$ are Fréchet differentiable on $H_\ep^1(\Omega)$ (and on the corresponding weighted spaces for the limit problem). The proof follows from Lemma \ref{lema:nemytskii}, standard arguments and can be adapted from classical results in the literature, see, e.g., \cite[Lemma 6.7]{patricia} and \cite[Lemma 4.2]{flank}.
\end{remark}


       
\section{Rate of the distance of attractors}\label{sec:6}

We prove that the perturbed and limit attractors converge in the Hausdorff distance with an explicit rate. The argument is built in a simple chain: (i) rates for equilibria; (ii) convergence estimates for linear semigroups; (iii) corresponding estimates for nonlinear semigroups; (iv) continuity with rate of local unstable manifolds. Putting these pieces together yields the rate for the attractors. 

\subsection{Continuity of the equilibria set}

In this subsection we compare the equilibria of the perturbed problems with those of the limit one–dimensional problem. 
Let us denote by
\[
\mathcal{E}_\varepsilon = \left\{ u^\varepsilon \in X_\ep^{\frac{1}{2}} : u^\varepsilon \text{ solves the problem (\ref{Probpertmud})} \right\},
\]
the set of equilibrium solutions of \eqref{eq1}. Our objective is to prove that the family \( \mathcal{E}_\varepsilon \) exhibits semicontinuity (both upper and lower) as \( \varepsilon \to 0 \), where the limit set corresponds to the equilibria of the effective one-dimensional problem.





To this end, we divide the argument into two theorems, addressing the upper and lower semicontinuity of the equilibrium set $\mathcal{E}_\varepsilon$ as $\varepsilon \to 0$. 

We begin by analyzing the upper semicontinuity of solutions $u^\varepsilon$.







\begin{theorem}\label{thm:53}
    Let $a$ satisfy the condition \eqref{H1} and $\{u^{\ep}\}_{\ep \in (0, \ep_{0}]}$ a sequence such that $u^{\ep} \in \mathcal{E}_\varepsilon$ for each $\varepsilon \in (0, \ep_{0}]$. Then, there exists a subsequence, still denoted by $\{u^{\ep}\}_{\ep \in (0, \ep_{0}]}$, and a function $u^0 \in H^{1}_{a}(0,1)$ such that $|| u^{\ep} - Eu^0||_{X^\al_\ep} \rightarrow 0$ as $\ep \rightarrow 0$ and $u^0 \in \mathcal{E}_0$.
\end{theorem}

\demo The proof is based on \cite [Proposition 4.1] {barros}. 

First, we observe that a solution $u^{\ep}$ of (\ref{Probpertmud}) is uniformly bounded in $X^\al_\ep$ with respect to $\ep$. For this, consider $u^{\ep} \in \mathcal{E}_\varepsilon$ and the test function $\varphi = u^{\varepsilon}$ in the weak formulation of \eqref{Probpertmud}, it follows that
\[
\int_\Omega A^\varepsilon\left(x, y \right) \nabla^\ep u^\ep\cdot \nabla^\ep u^\ep + u^\ep u^\ep = \int_\Omega f (u^\ep) u^\ep.
\]

\noindent Then, using the definition of the norm $|| \cdot ||_{X^{\frac{1}{2}}_\ep}$,

\begin{equation*}
    ||u^{\ep}||^{2}_{X^\frac{1}{2}_\ep}\leq \sup_{x \in \Omega}|f(x)| |\Omega|^{\frac{1}{2}} ||u^{\varepsilon}||_{X_\ep} \leq \sup_{x \in \Omega}|f(x)| |\Omega|^{\frac{1}{2}} ||u^{\varepsilon}||_{H^{1}(\Omega)}.
\end{equation*}
\
\noindent Since $|| \cdot ||_{H^{1}(\Omega)} \leq || \cdot ||_{H^{1}_\ep(\Omega)}$ for all $\varepsilon \in (0,\ep_0]$, there exists a constant $K(\Omega, f) = K $ such that 
$||u^{\ep}||_{X^\frac{1}{2}_\ep} \leq K$
and then, 
\begin{equation*}
    \max \left\{ ||u^{\ep}||_{X_\ep}, ||u_{x}^{\ep}||_{X_\ep}, \tfrac{1}{\ep} ||\nabla_{y}u^{\ep}||_{X_\ep} \right\} \leq K, \quad \forall \varepsilon \in (0,\ep_0].
\end{equation*}
So, there exists a subsequence, still denoted $u^{\ep}$, such that: $u^\ep \rightharpoonup u^0$ weakly in $X_\ep^\al$, $u^\ep \to u^0$ strongly in $X_\ep$ and $\nabla_y u^\ep \to 0 \text{ strongly in } X_\ep$ as $\varepsilon \rightarrow 0$, for some $u^{0} \in X_\ep^\al$. Moreover, it follows that $u^{0}(x,y) = u^{0}(x)$, that is, $u^{0}$ does not depend on $y$ and so, $u^{0} \in X_0$.

It is simple to see that $u^{0}$ satisfies the limit equation due to the convergences above and the fact that $\sup |f'| \leq \infty$.
On one side, we have
\begin{equation*}
    \int_{\Omega} |f(u^{\ep}) - f(u^{0})| |\varphi| \leq \sup ||f'||_{\infty} ||u^{\ep} - u^{0}||_{X_\ep} ||\varphi||_{X_\ep} \rightarrow 0
\end{equation*}
\noindent as $\ep \rightarrow 0$, for all $\varphi  \in X_\ep$. Now, on the other side, it is straightforward,
\begin{equation*}
    \int_{\Omega} A_{01} u_{x}^{0} \varphi_{x} + u^{0} \varphi = \int_{\Omega} f(u^{0}) \varphi.
\end{equation*}
\noindent Then, $u^{0}$ is a solution of (\ref{problim}). If $\varphi$
does not depend on $y$,
\begin{equation*}
    |B_{1}|\int_{0}^{1} A_{01} a^{n} u_{x}^{0} \varphi_{x} + a^{n}u^{0} \varphi = |B_{1}|\int_{0}^{1} a^{n} f(u^{0}) \varphi.
\end{equation*}

It remains to prove strong convergence in $X^\al_\ep$. This follows analogously to the proof of the Theorem \ref{thm:convergencia} since $\al \in (0,\tfrac{1}{2}]$. Therefore, $||u^{\ep} - Eu^{0}||_{X^\al_\ep} \to 0$ as $\ep \to 0$.
    
\fimdemo

We now investigate whether the family $\{ \mathcal{E}_\varepsilon \}_{\varepsilon \in (0,\ep_0]}$ also satisfies lower semicontinuity. 
To address this question, we consider the structure of the nonlinear elliptic problem and its limiting configuration. Suppose $u^\varepsilon \in \mathcal{E}_\varepsilon$ converges to $u^0 \in X_0$, and that $u^0 \in \mathcal{E}_0$ is hyperbolic in the sense that $0 \notin \sigma(L_0 - F_0'(u^0))$ with $L_0$ being an operator of compact resolvent. This non-degeneracy condition guarantees that $u^0$ is an isolated equilibrium. Our goal is to show that for sufficiently small $\varepsilon > 0$, the approximate equilibria $u^\varepsilon$ inherit this property and are also isolated. This uniform control on the structure of the equilibria allows us to conclude the lower semicontinuity of the family $\{ \mathcal{E}_\varepsilon \}_{\varepsilon \in (0,\ep_0]}$ at $\varepsilon = 0$.

We begin by noting that if $F$ is continuously Fréchet, then the operators $L_\varepsilon - F_\varepsilon'(u^0)$ are invertible for small $\varepsilon > 0$, provided the limit operator $L_0 - F_0'(u^0)$ is invertible with compact resolvent. This leads to the following result

\begin{proposition}
Let us assume conditions \eqref{H1} and \eqref{H2} for the function $a$. Then, the family of operators \( \{ L_\varepsilon^{-1}F_\varepsilon \}_{\varepsilon \in (0,\ep_0]} \) converges compactly to \( L_0^{-1}F_0 \) in \( X^\al_\ep \) as \( \varepsilon \to 0 \).
\end{proposition}

\demo For each \( \varepsilon \in (0,\ep_0] \), the operator \( L_\varepsilon^{-1} F_\varepsilon : X^\al_\ep \to X^\al_\ep \) is compact due to the compactness of the resolvent \( L_\varepsilon^{-1} \) and the continuity of $F_\varepsilon$.

Given a sequence \( u^\varepsilon \in X^\al_\ep \) such that \( \| u^\varepsilon \|_{X^\al_\ep} \leq C \), the uniform boundedness of \( \{ F_\varepsilon(u^\varepsilon) \}_{\ep\in(0,\ep_0]} \subset X_\ep \), together with the compactness of \( L_\varepsilon^{-1} \), implies that the sequence \( \{ L_\varepsilon^{-1} F_\varepsilon(u^\varepsilon) \}_{\ep\in(0,\ep_0]} \) is precompact in \( X^\al_\ep \).

Now applying Theorem \ref{thm:53} to $L_\varepsilon^{-1}F_\varepsilon(u^\varepsilon)$ and $L_0^{-1}F_0(u^0)$, we obtain the following
\[
\bigl\|L_\varepsilon^{-1}F_\varepsilon(u^\varepsilon)-E\,L_0^{-1}F_0(u^0)\bigr\|_{X^\al_\ep}\to 0,
\quad \text{as } \varepsilon \to 0.
\]
This yields the desired compact convergence. 
\fimdemo

\begin{proposition}\label{propconcompoplinea}
Let $u_0^* \in \mathcal{E}_0$ be a hyperbolic equilibrium. Then, for $\varepsilon \in (0,\ep_0]$, we have $0 \notin \sigma(L_\varepsilon - F_\varepsilon'(Eu_0^*))$ and $  \left\| (L_\varepsilon - F_\varepsilon'(E u_0^*))^{-1} \right\|_{\mathcal{L}(X^\alpha_\ep, X_\ep)} \leq C,$ where $C > 0$ is independent of $\varepsilon$. Moreover, the family of inverses $\{ (L_\varepsilon - F_\varepsilon'(E u^*_0))^{-1} \}_{\varepsilon \in (0,\ep_0]}$ converges compactly to $(L_0 - F_0'(u^*_0))^{-1}$.
\end{proposition}

\demo The proof follows from the compact convergence of the resolvent operators and the Fréchet differentiability of $F_\varepsilon$. Formally, we observe the identity
\begin{equation*}
(L_\varepsilon - F_\varepsilon'(E u^*_0))^{-1} = (I - L_\varepsilon^{-1} F_\varepsilon'(E u^*_0))^{-1} L_\varepsilon^{-1}.
\end{equation*}
The compact convergence $L_\varepsilon^{-1} F_\varepsilon'(E u^*_0) \to L_0^{-1} F_0'(u_0^*) \ \ \text{in } \ \mathcal{L}(X^\alpha_\ep)$ can be established by an argument analogous to that used in the previous proposition for the sequence \( w^\varepsilon = L_\varepsilon^{-1} F'_\varepsilon(u^\varepsilon) \).

We now justify the uniform boundedness of the inverse operator. Since $u_0^* \in \mathcal{E}_0$ is a hyperbolic equilibrium, we have that $0 \notin \sigma(L_0 - F_0'(u_0^*))$, and hence the limit operator is invertible with bounded inverse. 

Using the compact convergence $L_\varepsilon^{-1} F_\varepsilon'(E u^*_0) \to L_0^{-1} F_0'(u_0^*)
\ \ \text{in }\ \mathcal{L}(X^\alpha_\ep),$ we obtain that the operators $I - L_\varepsilon^{-1} F_\varepsilon'(E u^*_0)$
converge to an invertible limit. Therefore, by the continuity of the inverse under compact convergence, there exists $\varepsilon_0 > 0$ and $C > 0$ such that
\[
\left\| (L_\varepsilon - F_\varepsilon'(E u^*_0))^{-1} \right\|_{\mathcal{L}(X^\alpha_\ep, X_\ep)} \leq C \quad \forall \ep \in (0,\ep_0].
\]

Let now \( \{ u^\varepsilon \}_{\varepsilon \in (0,\ep_0]} \subset X^\alpha_\ep\) with \( \| u^\varepsilon \|_{X^\alpha_\ep} \leq 1 \), and define $z^\varepsilon := (L_\varepsilon - F_\varepsilon'(E u^*_0))^{-1} u^\varepsilon,$ so that
\[
u^\varepsilon = (L_\varepsilon - F_\varepsilon'(E u^*_0)) z^\varepsilon = L_\varepsilon z^\varepsilon - F_\varepsilon'(E u^*_0) z^\varepsilon.
\]
Rewriting, $z^\varepsilon = L_\varepsilon^{-1} u^\varepsilon + L_\varepsilon^{-1} F_\varepsilon'(E u^*_0) z^\varepsilon.$

Using the boundedness of \( \{ u^\varepsilon \}_{\ep \in (0,\ep_0]} \) in \( X^\alpha_\ep \), the uniform boundedness of \( L_\varepsilon^{-1} \), and the compact convergence \( L_\varepsilon^{-1} F_\varepsilon'(E u^*_0) \to L_0^{-1} F_0'(u_0^*) \), we conclude (up to a subsequence) that \( z^\varepsilon \to z^0 \) in \( X^\alpha_\ep \) with $u^0 = L_0 z^0 - F_0'(u_0^*) z^{0} = (L_0 - F_0'(u_0^*)) z^0.$ Thus, \( z^0 = (L_0 - F_0'(u_0^*))^{-1} u^0 \), and the compact convergence follows.
\fimdemo

\begin{proposition}\label{propcontra}
Let $u_0^* \in \mathcal{E}_0$ be a hyperbolic equilibrium. Then, 
for all $\varepsilon \in (0,\ep_0]$, the map
\[
\Phi(u^\varepsilon) := (L_\varepsilon - F_\varepsilon'(Eu_0^*))^{-1} \left( F_\varepsilon(u^\varepsilon) - F_\varepsilon'(Eu_0^*) u^\varepsilon\right)
\]
is a contraction in a neighborhood of $u_0^*$ in $X^\alpha_\ep$, and hence $u^\varepsilon$ is an isolated equilibrium.
\end{proposition}

\demo This follows from Banach's Fixed Point Theorem, using the continuous Fréchet differentiability of $F_\varepsilon$ and the compact convergence of the inverse operators $(L_\varepsilon - F_\varepsilon'(u_0^*))^{-1} \to (L_0 - F_0'(u_0^*))^{-1}$ as $\varepsilon \to 0$. A suitable estimate shows that $\Phi$ becomes a contraction map for small enough $\varepsilon$, which completes the argument.
\fimdemo

\begin{lemma}
If all equilibria in $\mathcal{E}_0$ are isolated, then for $\varepsilon \in (0,\ep_0]$, the set $\mathcal{E}_\varepsilon$ is finite and every equilibrium is hyperbolic.
\end{lemma}

\demo This follows from Propositions \ref{propconcompoplinea} and \ref{propcontra}, which establish the compact convergence of linearized operators, together with the continuous Fréchet differentiability of $F_\varepsilon$.

In particular, the invertibility of $L_0 - F_0'(u_0^*)$ at each $u_0^* \in \mathcal{E}_0$ implies that the operators $L_\varepsilon - F_\varepsilon'(u^\varepsilon)$ remain invertible for small $\varepsilon$ and $u^\varepsilon$ close to $u_0^*$. This ensures that the nearby equilibria is hyperbolic. Since $\mathcal{E}_0$ is discrete and compactness, it guarantees that the set $\mathcal{E}_\varepsilon$ must be finite.
\fimdemo


    


\begin{corollary}
Suppose $u_0^* \in \mathcal{E}_0$ is a hyperbolic equilibrium. Then, there exist $\delta > 0$ and $\varepsilon_0 > 0$ such that, for all $\varepsilon \in (0,\varepsilon_0]$, the set of equilibria $\mathcal{E}_\varepsilon$ in the $\delta$-neighborhood of $E u_0^*$ consists of a unique equilibrium $u^\varepsilon$ satisfying $u^\varepsilon \to Eu_0^*$ in $X^\al_\ep$.
\end{corollary}

Finally, we obtain the rate of convergence under condition \eqref{H3}. 
\begin{theorem}\label{teoconvequil}
 Assume conditions \eqref{H1}, \eqref{H2} and \eqref{H3}. Suppose that \( u_0^* \in \mathcal{E}_0 \) is a hyperbolic equilibrium. Then, there exist  a constant \( C > 0 \) such that, for all \( \varepsilon \in (0, \varepsilon_0] \), the set of equilibria \( \mathcal{E}_\varepsilon \) of the equation $L_\varepsilon u^\varepsilon = F_\varepsilon(u^\varepsilon)$ contains exactly one equilibrium \( u_\varepsilon^* \) sufficiently close to \( E u_0^* \), and this equilibrium satisfies the following estimate $\| u_\varepsilon^* - E u_0^* \|_{X^\al_\ep} \leq C \varepsilon.$
\end{theorem}

\demo Since \( u_0^* \in \mathcal{E}_0 \) is hyperbolic, the linearized operator \( L_0 - F_0'(u_0^*) \) is invertible with a bounded inverse. By Proposition \ref{propconcompoplinea}, $L_\varepsilon - F_\varepsilon'(E u_0^*)$ remains invertible, with uniformly bounded inverse, $\left\| (L_\varepsilon - F_\varepsilon'(E u_0^*))^{-1} \right\|_{\mathcal{L}(X^\alpha_\ep, X_\ep)} \leq C,$ for all $\ep \in(0,\ep_0].$

By Proposition \ref{propcontra}, we find that \( \Phi \) is a contraction in a neighborhood of \( Eu_0^* \). Then, by the Banach's Fixed Point Theorem, there exists a unique fixed point \( u_\varepsilon\) sufficiently nearby of the $Eu_0^*$, which is a solution to the equation \( L_\varepsilon u_\ep^* = F_\varepsilon(u_\ep^*) \).

To obtain the convergence rate of the solutions, observe that
\[
u^*_0 = (L_0 - F_0'(u^*_0))^{-1} (F_0(u^*_0) - F_0'(u^*_0)(u^*_0)),
\]
\[
u_\varepsilon^* = (L_\varepsilon - F_\varepsilon'(E u^*_0))^{-1} (F_\varepsilon(u_\varepsilon^*) - F_\varepsilon'(E u^*_0)(u_\varepsilon^*)).
\]

Adding and subtracting \( E (L_0 - F_0'(u^*_0))^{-1} M (F_\varepsilon(u_\varepsilon^*) - F_\varepsilon'(E u^*_0)(u_\varepsilon^*)) \), we obtain
\begin{align*}
u_\varepsilon^* - E u^*_0 
&= (L_\varepsilon - F_\varepsilon'(E u^*_0))^{-1} [F_\varepsilon(u_\varepsilon^*) - F_\varepsilon'(E u^*_0)(u_\varepsilon^*)] \\
&\quad - E (L_0 - F_0'(u^*_0))^{-1} [F_0(u^*_0) - F_0'(u^*_0)(u^*_0)] \\
&= \left( (L_\varepsilon - F_\varepsilon'(E u^*_0))^{-1} - E (L_0 - F_0'(u^*_0))^{-1} M \right) (F_\varepsilon(u_\varepsilon^*) - F_\varepsilon'(E u^*_0)(u_\varepsilon^*)) \\
&\quad + E (L_0 - F_0'(u^*_0))^{-1} \left( F_0(u^*_0) - F_0'(u^*_0)(u^*_0) - M (F_\varepsilon(u_\varepsilon^*) - F_\varepsilon'(E u^*_0)(u_\varepsilon^*)) \right).
\end{align*}
Therefore, we obtain the following estimate
\begin{align*}
\| u_\varepsilon^* - E u^*_0 \|_{X^\al_\ep} 
&\leq \| (L_\varepsilon - F_\varepsilon'(E u^*_0))^{-1} - E (L_0 - F_0'(u^*_0))^{-1} M  F_\varepsilon(u_\varepsilon^*) - F_\varepsilon'(E u^*_0)(u_\varepsilon^*) \|_{X^\al_\ep} \\
&\quad + \| E (L_0 - F_0'(u^*_0))^{-1} F_0(u^*_0) - F_0'(u^*_0)(u^*_0) - M(F_\varepsilon(u_\varepsilon^*) - F_\varepsilon'(E u^*_0)(u_\varepsilon^*)) \|_{X^\al_\ep}.
\end{align*}

To conclude the result, it is necessary to combine the estimates obtained in the previous results, particularly those controlling the convergence of the operators \( L_\varepsilon^{-1} \to L_0^{-1} \), and to apply the continuity of the nonlinear operator \( F \) and its derivatives. These arguments ensure that both terms on the right-hand side tend to zero as \( \varepsilon \to 0 \), and more precisely, they lead to the estimate $\| u_\varepsilon^* - E u_0^* \|_{X^\al_\ep} \leq C \varepsilon,$ which establishes the convergence rate. For more details, see for instance \cite[Theorem 6.8]{patricia}.
\fimdemo

\subsection{Convergence of linear semigroups}
Since $L_\varepsilon$ is a sectorial operator, the operator $-L_\varepsilon$ is the infinitesimal generator of a linear analytic semigroup, denoted by $e^{-L_\varepsilon t}$, which is given by the Dunford integral representation
\begin{equation*}
e^{-L_\varepsilon t} = \frac{1}{2\pi i} \int_{\Gamma} (\lambda I + L_\varepsilon)^{-1} e^{\lambda t} \, d\lambda,
\end{equation*}
where $\Gamma$ is a contour in the resolvent set of $-L_\varepsilon$, that is, $\Gamma \subset \rho(-L_\varepsilon)$, with $\arg \lambda \to \pm \theta$ as $|\lambda| \to \infty,$ for some $\theta \in \left( \tfrac{\pi}{2}, \pi \right),$ (see \cite{henry}).

Since $L_\varepsilon$, $\ep\in[0,\ep_0]$, is self-adjoint, positive with compact resolvent, its spectrum is positive with, $$\sigma(L_\varepsilon) = \{ \lambda_{m}^{\varepsilon} \}^{\infty}_{m=1} \,\,\, \text{and} \,\,\, 0 < c \leq \lambda_{1}^{\varepsilon} \leq \lambda_{2}^{\varepsilon} \leq \cdots \leq \lambda_{m}^{\varepsilon} \leq \cdots$$ 
Moreover, if $\Re \left( \sigma(L_{\ep}) \right) > \delta > 0$, there exists a constant $M_{\delta}$, independent of $\ep$, such that
\begin{equation}\label{des:semHenry}
    ||e^{-L_{\ep}t}||_{\mathcal{L}(X_{0},X_{\varepsilon}^{\alpha})} \leq M_{\delta} e^{-\delta t} t^{-\alpha}, \,\, t > 0
\end{equation}
\noindent (see \cite{henry}).

\begin{lemma} \label{cpto}
If $K_0$ is a compact set of the complex plane with $K_0 \subset \rho(-L_0)$ and the estimate in the Corollary \ref{cor:convergencia} is satisfied, then there exists $\varepsilon_0(K_0) > 0$ such that $K_0 \subset \rho(-L_\varepsilon)$ for all $0 < \varepsilon \leq \varepsilon_0(K_0)$. Moreover, we have the estimates
\begin{equation*}
    \left\| (\lambda I + L_\varepsilon)^{-1} \right\|_{\mathcal{L}(X_\varepsilon, X_\varepsilon^\al)} \leq C(K_0), \quad
    \left\| (\lambda I + L_\varepsilon)^{-1} \right\|_{\mathcal{L}(X_\varepsilon, X_\varepsilon)} \leq C(K_0),
\end{equation*}
for all $\lambda \in K_0$ and $0 < \varepsilon \leq \varepsilon_0(K_0)$.
\end{lemma}
\demo
See \cite[Lemma 3.2]{esperanza1}.
\fimdemo

\

Now we want to estimate $\|(\lambda I + L_\varepsilon)^{-1}E - E(\lambda I + L_0)^{-1}\|_{\mathcal{L}(X_0, X_\varepsilon^\al)}$. 

\begin{lemma}\label{lemma:COnvResolvent}
Assuming that the estimate in the Corollary \ref{cor:convergencia} is satisfied, if $\lambda \in \rho(-L_0)$  and $\ep$ is small enough so that $\lambda \in \rho(-L_\varepsilon)$, we have
$$\|(\lambda I + L_\varepsilon)^{-1}E - E(\lambda I + L_0)^{-1}\|_{\mathcal{L}(X_0, X_\varepsilon^\al)} 
\leq C^\varepsilon(\lambda)\, \nu(\ep),$$
where $C^\varepsilon(\lambda) = \left( 1 + \frac{|\lambda|}{\mathrm{dist}(\lambda, \sigma(-L_\varepsilon))} \right)
\left( 1 + \frac{|\lambda|}{\mathrm{dist}(\lambda, \sigma(-L_0))} \right).$
\end{lemma}
\demo
See \cite[Lemma 3.4]{esperanza1}.
\fimdemo

\begin{remark}\label{remark54}
In this setting, the parameter $\nu(\ep)$ appearing in the previous lemma admits the explicit form $\nu(\ep) = C_0\,\ep$, where $C_0 > 0$ is a constant independent of $\ep$.
\end{remark}

\begin{corollary}\label{cor:UnifConstant}
\emph{(i)} If $K_{0} \subset \rho (-L_{0})$ as in Lemma \ref{cpto} and $\Sigma_{-\kappa,\phi}$ is the set of the complex plane described by $\Sigma_{-\kappa,\phi} = \{\lambda \in \mathbb{C}, \,\, |\lambda + \kappa| \leq \pi - \phi \}$  with $\kappa \geq 0$, then $$\sup_{ \lambda \in K_{0} \cup \Sigma_{-\kappa, \phi} } C^{\ep} (\lambda) \leq C  $$ for some constant $C$ independent of $\lambda$. Note that, due to the spectral continuity, we can also consider $C^{\ep} (\lambda)$ independent of $\ep$.

\noindent \emph{(ii)} If we take $\kappa = 0$ and $\phi = \frac{\pi}{4}$ then $$C^{\ep}(\lambda) \leq \left( 1 + \tfrac{1}{\sin(\phi)} \right)^{2} \leq 6, \,\, \text{for all} \,\, \lambda \in \Sigma_{0, \frac{\pi}{4}}.$$
\end{corollary}
\demo
See \cite[Corollary 3.5]{esperanza1}.
\fimdemo

\

Building upon the resolvent estimates established in the previous results, we now derive estimates for the corresponding linear semigroups. These bounds are essential for the subsequent analysis of the nonlinear system.

\begin{lemma}\label{lemma:ConvSemigruposLin}
Let $\theta \in (0, 1)$. Then, there exists a constant $C>0$ independent of $\ep$, such that 
\begin{equation*}
 \left\| e^{-L_\varepsilon t} E - E e^{-L_0 t} \right\|_{\mathcal{L}(X_0, X^\alpha_\varepsilon)} 
\leq C e^{-\delta(1-\theta)t}t^{-\al(1-\theta)-\theta}\ep^\theta.    
\end{equation*}
\end{lemma}

\demo
Let $\Sigma_{0,\phi} = \left\{ \lambda \in \mathbb{C} : |\arg(\lambda)| \leq \pi - \phi \right\}, \ \text{with} \ \phi = \frac{\pi}{4}$, and let $\Gamma$ be the boundary of $\Sigma_{0,\frac{\pi}{4}}$, that is, the curve consisting of the following segments $\Gamma^1$ and $\Gamma^2$,$$\Gamma = \Gamma^1 \cup \Gamma^2 = \left\{ re^{-i(\pi - \phi)} : 0 \leq r < +\infty \right\} \cup \left\{ re^{i(\pi - \phi)} : 0 \leq r < +\infty \right\},$$ oriented such that the imaginary part grows as $\lambda$ runs in $\Gamma$. 

We know that
$$e^{-L_\varepsilon t}E - E e^{-L_0 t} = \frac{1}{2\pi i} \int_\Gamma \left( (\lambda I + L_\varepsilon)^{-1}E - E(\lambda I + L_0)^{-1} \right) e^{\lambda t} d\lambda.$$
Thus, using Lemma \ref{lemma:COnvResolvent}, we obtain
\begin{equation}\label{des:semigLin}
\left\| e^{-L_\varepsilon t}E - E e^{-L_0 t} \right\|_{\mathcal{L}(X_0, X^\alpha_\varepsilon)} 
\leq \frac{1}{2\pi} \left|\int_\Gamma C_1 \,\ep\, |e^{\lambda t}|\, |d\lambda| \right|\leq\left|\int_\Gamma C_1 \,\ep\, e^{\mathrm{Re}\lambda t}\, |d\lambda| \right|
\end{equation}
with $C_1 = C_0\sup\limits_{\lambda \in \Gamma} C^\varepsilon(\lambda)$ (as in Corollary \ref{cor:UnifConstant} and Remark \ref{remark54}). Making the change of variables $\mu = \lambda t$, we have 
\begin{eqnarray*}
\int_{\Gamma}e^{\mathrm{Re}\lambda t}\;\mathrm{d}\left|\lambda\right| = \int_{\Gamma}e^{\mathrm{Re}\mu}\;\mathrm{d}\frac{\left|\mu\right|}{t} \leq \frac{1}{t}\int_{\Gamma}e^{\left|\mu\right|}\;\mathrm{d}\left|\mu\right| = C_2\; t^{-1}. 
\end{eqnarray*}
Thus, from \eqref{des:semigLin}, we deduce that
\begin{equation}\label{des:semigLin2}
\left\| e^{-L_\varepsilon t}E - E e^{-L_0 t} \right\|_{\mathcal{L}(X_0, X^\alpha_\varepsilon)} 
\leq \frac{1}{2\pi}C_1C_2\,\ep\,t^{-1}.
\end{equation}
On the other hand, from \eqref{des:semHenry}, we have
\begin{align}\label{eqsemiglin2}
\big\|e^{-L_\varepsilon t}E - Ee^{-L_0t}\big\|_{\mathcal{L}(X_0, X^{\al}_{\varepsilon})} \leq \big\|e^{-L_\varepsilon t}E\big\|_{\mathcal{L}(X_0, X^{\al}_{\varepsilon})}  + \big\|Ee^{-L_0t}\big\|_{\mathcal{L}(X_0, X^{\al}_{\varepsilon})} \ \leq \ C_3\,e^{-\delta t}t^{-\al}.
\end{align}
For $\theta \in (0,1)$, interpolating between estimates \eqref{des:semigLin2} and \eqref{eqsemiglin2} with weights $\theta$ and $1 - \theta$, respectively (see \cite[p. 59, 103]{triebel}), we conclude that

\begin{equation*}
 \left\| e^{-L_\varepsilon t} E - E e^{-L_0 t} \right\|_{\mathcal{L}(X_0, X^\alpha_\varepsilon)} 
\leq C e^{-\delta(1-\theta)t}t^{-\al(1-\theta)-\theta}\ep^\theta    
\end{equation*}
where $C=(\frac{1}{2\pi}C_1C_2)^{\theta}(C_3)^{(1-\theta)}.$
\fimdemo

\subsection{Continuity of nonlinear semigroups}

In this section, we extend the analysis to the nonlinear setting and demonstrate the upper semicontinuity of the global attractors. 

We will show the following result
\begin{theorem}\label{contsemig}
Let $T_{\varepsilon}(t)$ denote the nonlinear semigroup generated by the problem \eqref{eq1}. Then, for any $R > 0$ sufficiently large and for all $t \in (0, \tau)$, where $\tau > 0$ is fixed, there exist constants $\varepsilon_0<1$, $C' > 0$ and $\varrho > 0$, independent of $\ep$, such that for any $u_0^0 \in X_0$ satisfying $\|u_0^0\|_{X_0} < R$, we have
\[
\|T_{\varepsilon}(t) E u_0^0 - E T_0(t) u_0^0\|_{X_{\varepsilon}^{\alpha}} 
\leq C' e^{\varrho t} t^{-\alpha(1 - \theta) - \theta}\varepsilon^{\theta}, \; \textrm{ for all } \varepsilon \in [0, \varepsilon_0].
\]
\end{theorem}
\demo
Let $\tau >0$ and $t \in (0, \tau)$. Since
\begin{equation*}
T_{\varepsilon}(t, Eu_{0}^0) := T_\varepsilon(t)Eu_{0}^0 = e^{-t L_\varepsilon} E u_{0}^0 + \int_0^t e^{-(t-s) L_\varepsilon} F_\varepsilon\big(T_{\varepsilon}(s, Eu_0^0)\big)\;\mathrm{d}s, \quad \ep \in [0,\ep_0],
\end{equation*}
we may apply the Lemma \ref{lemma:ConvSemigruposLin} to estimate the linear part. This yields
\begin{eqnarray*}
||T_{\ep}(t, Eu_{0}^{0}) - ET_{0} (t, u_{0}^{0})||_{X^{\alpha}_{\ep}}  & \leq & ||(e^{-tL_{\ep}}E - Ee^{-tL_{0}}) u^{0}_0||_{X^{\alpha}_{\ep}} \\
& + & \int_{0}^{t} ||e^{-(t - s) L_{\ep}} F_{\ep}(T_{\ep}(s, Eu_{0}^{0})) - Ee^{-(t - s) L_{0}} F_{0}(T_{0}(s, u_{0}^{0}))||_{X^{\alpha}_{\ep}}ds \\
& \leq & C e^{-\delta(1 - \theta)t}t^{-\alpha(1 - \theta)} \ep^{\theta}t^{-\theta} ||u_{0}^{0}||_{X_{0}} \\
& + & \int_{0}^{t} ||(e^{-(t-s)L_{\ep}}E - Ee^{-(t-s)L_{0}}) F_{0}(T_{0}(s,u_{0}^{0}))||_{X^{\alpha}_{\ep}}ds \\
& + & \int_{0}^{t} ||e^{-(t - s)L_{\ep}}(F_{\ep}(E(T_{0}(s, u_{0}^{0})) - EF_{0}(T_{0}(s, u_{0}^{0})))||_{X^{\alpha}_{\ep}} ds \\
& + & \int_{0}^{t} ||e^{-(t-s)L_{\ep}} (F_{\ep}(T_{\ep}(s, Eu^{0}_{0})) - F_{\ep}(E(T_{0}(s, u_{0}^{0}))))||_{X^{\alpha}_{\ep}} ds
\end{eqnarray*}
In what follows, we are going to analyze these three last integrals. Since $F_\ep$ and $F_0$ are the Nemitskii's operators generated by the same nonlinearity $f:\mathbb{R}\to\mathbb{R}$, they are compatible with the extension $E$, that is, $F_\varepsilon(Ev)=E\,F_0(v),\ \text{for all }v.$ In particular, $F_{\ep}(E(T_{0}(s, u_{0}^{0}))=EF_{0}(T_{0}(s, u_{0}^{0})),$ hence the second integral is identically zero. For the first integral, using Lemma \ref{lema:nemytskii} and Lemma \ref{lemma:ConvSemigruposLin},
\begin{eqnarray*}
  & & \int_{0}^{t} ||(e^{-(t-s)L_{\ep}}E - Ee^{-(t-s)L_{0}}) F_{0}(T_{0}(s,u_{0}^{0}))||_{X^{\alpha}_{\ep}}ds \\
  & \leq & \int_{0}^{t} C e^{-\delta(1 - \theta)(t - s)}(t - s)^{-\alpha(1 - \theta)} \ep^{\theta}(t -s)^{-\theta} ||F_{0}(T_{0}(s,u_{0}^{0}))||_{X_{0}} ds \\
    & \leq & C_{1} \ep^{\theta} \int_{0}^{t}  e^{-\delta(1 - \theta)(t - s)}(t - s)^{-\alpha(1 - \theta) - \theta} ds \\
    & \leq & C_{1} \ep^{\theta} \int_{0}^{\delta t(1 - \theta)} e^{-y} \left(\tfrac{y}{\delta(1 - \theta)}\right)^{-\alpha(1 - \theta) - \theta} \tfrac{1}{(\delta(1 - \theta))} dy \\
    & \leq & C_{1} \ep^{\theta} \left(\tfrac{1}{\delta(1 - \theta)}\right)^{-\al(1 - \theta) -\theta + 1} \int_{0}^{\infty} e^{-y} y^{-\alpha(1 - \theta) - \theta} dy \ \leq \ C_{2} \ep^{\theta}
\end{eqnarray*}

\noindent where $C_{1}=CC_F$ and $C_{2}=C_1\small{\Big(\big(\tfrac{1}{\delta(1 - \theta)}\big)^{-\al(1 - \theta) -\theta + 1} \Gamma (-\alpha(1 - \theta) - \theta + 1)\Big)}$. And, with respect to the third integral, using (\ref{des:semHenry}) and Lemma \ref{lema:nemytskii},
\begin{eqnarray*}
   & & \int_{0}^{t} ||e^{-(t-s)L_{\ep}} (F_{\ep}(T_{\ep}(s, Eu^{0}_{0})) - F_{\ep}(E(T_{0}(s, u_{0}^{0}))))||_{X^{\alpha}_{\ep}} ds \\
   & & \quad \leq  \int_{0}^{t} M_{\delta} e^{-\delta(t - s)} (t-s)^{-\alpha} L_{F} ||T_{\ep}(s, Eu_{0}^{0}) - ET_{0}(s, u^{0}_{0})||_{X^{\alpha}_{\ep}} ds.
\end{eqnarray*}

Therefore, 
\begin{eqnarray*}
   ||T_{\ep}(t, Eu_{0}^{0}) - ET_{0} (t, u_{0}^{0})||_{X^{\alpha}_{\ep}} & \leq & C e^{-\delta(1 - \theta)t}t^{-\alpha(1 - \theta)} \ep^{\theta}t^{-\theta} R + C_{2} \ep^{\theta} t^{-\alpha(1 - \theta) - \theta} t^{\alpha(1 - \theta) + \theta} e ^{-\delta t} e^{\delta t} \\
   & + & \int_{0}^{t} M_{\delta} e^{-\delta(t - s)} (t-s)^{-\alpha} L_{F} ||T_{\ep}(s, Eu_{0}^{0}) - ET_{0}(s, u^{0}_{0})||_{X^{\alpha}_{\ep}} ds \\
   & \leq & \ep^{\theta} t^{-\alpha (1 - \theta) - \theta} e^{-\delta t} ( C_{3}e^{\delta \theta \tau} + C_{2} \tau^{\alpha(1 - \theta) + \theta} e^{\delta \tau}) \\
   & + & M_{\delta}L_{F}\int_{0}^{t}  e^{-\delta(t - s)} (t-s)^{-\alpha}  ||T_{\ep}(s, Eu_{0}^{0}) - ET_{0}(s, u^{0}_{0})||_{X^{\alpha}_{\ep}} ds \\
   & \leq & C_{4}\ep^{\theta} t^{-\alpha (1 - \theta) - \theta} e^{-\delta t} \\
   & + & M_{\delta}L_{F}\int_{0}^{t}  e^{-\delta(t - s)} (t-s)^{-\alpha}  ||T_{\ep}(s, Eu_{0}^{0}) - ET_{0}(s, u^{0}_{0})||_{X^{\alpha}_{\ep}} ds,
\end{eqnarray*}

\noindent where $C_{3} = CR$  and $C_{4} =( C_{3}e^{\delta \theta \tau} + C_{2} \tau^{\alpha(1 - \theta) + \theta} e^{\delta \tau})$. Consequently,
\begin{eqnarray*}
    e^{\delta t} ||T_{\ep}(t, Eu_{0}^{0}) - ET_{0} (t, u_{0}^{0})||_{X^{\alpha}_{\ep}} & \leq &  C_{4}\ep^{\theta} t^{-\alpha (1 - \theta) - \theta}  \\
    & + &  M_{\delta}L_{F}\int_{0}^{t}  e^{\delta s} (t-s)^{-\alpha}  ||T_{\ep}(s, Eu_{0}^{0}) - ET_{0}(s, u^{0}_{0})||_{X^{\alpha}_{\ep}} ds.
\end{eqnarray*}
By Gronwall's Inequality, there exists $C'> 0$ such that,
\begin{equation*}
    e^{\delta t} ||T_{\ep}(t, Eu_{0}^{0}) - ET_{0} (t, u_{0}^{0})||_{X^{\alpha}_{\ep}} \leq  C'\ep^{\theta} t^{-\alpha (1 - \theta) - \theta}.
\end{equation*}
Suppose $\beta > 1$ such that, $ e^{-\beta \delta t} e^{\delta t} ||T_{\ep}(t, Eu_{0}^{0}) - ET_{0} (t, u_{0}^{0})||_{X^{\alpha}_{\ep}} \leq  C'\ep^{\theta} t^{-\alpha (1 - \theta) - \theta}.$ Then, we conclude that
\begin{equation*}
  ||T_{\ep}(t, Eu_{0}^{0}) - ET_{0} (t, u_{0}^{0})||_{X^{\alpha}_{\ep}} \leq  C'e^{\varrho t} t^{-\alpha (1 - \theta) - \theta}\ep^{\theta},
\end{equation*} where $\varrho=\delta(\beta - 1)$.
\fimdemo
\begin{theorem}
Assume conditions \eqref{H1}--\eqref{H3}.
Then, the family of attractors $\left\lbrace\mathcal{A}_{\varepsilon}\right\rbrace_{\varepsilon \in [0, \varepsilon_0]}$ is upper semicontinuous at $\varepsilon = 0$ in $X_{\ep}^{\al}$. 
\end{theorem}

\demo
The upper semicontinuity follows from Theorem \ref{contsemig} and the uniform $L^\infty(\Omega)$ bounds on the attractors \eqref{eq:liminf}.  For more details, see for instance \cite[Proposition 5.1]{german1}.\fimdemo

\subsection{Linearization}\label{section:5.4}

We are interested in analyzing the compact convergence of the resolvent operators associated with the family
$
\bar{L}_\varepsilon := L_\varepsilon + V_\varepsilon,$
where $V_\varepsilon := -F_\ep'(u_\varepsilon^*) \in \mathcal{L}(X_\varepsilon^{\alpha}, X_\varepsilon)$ denotes the linearization of the nonlinearity around an equilibrium $u_\varepsilon^* \in \mathcal{E}_\varepsilon$, for $\varepsilon \in [0, \varepsilon_0]$.

Since the family $\{u_\varepsilon^*\}_{\ep\in[0,\ep_0]}$ is uniformly bounded in $L^\infty(\Omega)$, the operators $V_\varepsilon$ are uniformly bounded in operator norm. Moreover, from the convergence $u_\varepsilon^* \to u_0^*$ established in Theorem~\ref{teoconvequil}, we deduce the estimate $\|V_\varepsilon - V_0\|_{\mathcal{L}(X_\varepsilon^\alpha, X_\varepsilon)} \leq C \varepsilon.$

As $L_\varepsilon$ is sectorial and $V_\varepsilon$ is bounded, it follows that $\bar{L}_\varepsilon$ is also sectorial with uniform constants.
Defining the auxiliary operator $B_\varepsilon := V_\varepsilon L_\varepsilon^{-1}$, we exploit the regularity of the resolvent $L_\varepsilon^{-1}$ and the uniform boundedness of $V_\varepsilon$ to deduce the compact convergence $B_\varepsilon \to B_0$ in an appropriate sense. In particular, for any $u_0 \in X_0$, we have
$$
B_\varepsilon E u_0 = V_\varepsilon L_\varepsilon^{-1} E u_0 \longrightarrow E B_0 u_0 \ \ \text{in } X_\varepsilon,
$$
due to the strong convergence $L_\varepsilon^{-1} E u_0 \to E L_0^{-1} u_0$ and the uniform bound $\|V_\varepsilon\|_{\mathcal{L}(X_\varepsilon^\alpha, X_\varepsilon)}  \leq C$.

Using the resolvent identity
$$
\bar{L}_\varepsilon^{-1} - \bar{L}_0^{-1} = (L_\varepsilon^{-1} - L_0^{-1})(I + B_0)^{-1} - L_\varepsilon^{-1}(I + B_0)^{-1}(B_\varepsilon - B_0)(I + B_\varepsilon)^{-1},
$$
we conclude the compact convergence $\bar{L}_\varepsilon^{-1} \rightarrow \bar{L}_0^{-1} \ \text{in } \mathcal{L}(X_\varepsilon, X_\varepsilon^\alpha),$ with convergence rate $\ep$.

Finally, the compactness of $L_\varepsilon^{-1} \colon X_\varepsilon \to X_\varepsilon^\alpha$, together with the stability of the extension operator $E$, allows us to obtain, for $\lambda \in \rho(-\bar{L}_\varepsilon)$,
\begin{equation}\label{ConvResolLin}
  \left\| (\lambda I + \bar{L}_\varepsilon)^{-1} E - E(\lambda I + \bar{L}_0)^{-1} \right\|_{\mathcal{L}(X_0, X^\alpha_\varepsilon)} \leq C \varepsilon.  
\end{equation}

\begin{lemma}\label{lemma:ConvSemigroupsLinPert}
Let $\theta \in (0, 1)$. Then, there exists a constant $C>0$, independent of $\varepsilon$, such that
$$\left\| e^{-\bar{L}_\varepsilon t} E - E e^{-\bar{L}_0 t} \right\|_{\mathcal{L}(X_0, X^\alpha_\varepsilon)} 
\leq C e^{-\delta(1-\theta)t}t^{-\alpha(1-\theta)-\theta} \varepsilon^\theta, \ \ \text{for all } t>0.$$
\end{lemma}
\demo
The result follows by applying the sectorial estimates for the perturbed operators $\bar{L}_\varepsilon$ (as in \eqref{des:semHenry}), together with \eqref{ConvResolLin} and similar arguments to those used in the proof of Lemma \ref{lemma:ConvSemigruposLin} for the linear semigroups.
\fimdemo
\medskip

For each $\varepsilon\in[0,\varepsilon_0]$, the spectrum of $\bar{L}_\varepsilon$ (ordered and counted with multiplicity) reads $$\sigma(\bar{L}_\varepsilon) = \{ \lambda_{m}^{\varepsilon} \}^{\infty}_{m=1} \,\,\, \text{and} \,\,\, 0 < c' \leq \lambda_{1}^{\varepsilon} \leq \lambda_{2}^{\varepsilon} \leq \cdots \leq \lambda_{m}^{\varepsilon} \leq \cdots$$ 
with corresponding eigenfunctions $\{\varphi_i^\varepsilon\}_{i=1}^\infty$. 
Let $\bar\Gamma$ be a smooth, closed, simple, rectifiable curve contained in
$\{z\in\mathbb C:\operatorname{Re} z>0\}$, oriented counterclockwise, such that the
bounded connected component of $\mathbb C\setminus\bar\Gamma$ contains a finite number of eigenvalues of $\bar L_0$.
By the continuity of the spectrum with respect to $\varepsilon$
there exists $\varepsilon_{\bar\Gamma}>0$ such that $\bar\Gamma\subset \rho(\bar L_\varepsilon)\ \text{for all }\,0\le\varepsilon\le\varepsilon_{\bar\Gamma}.$
Hence the (Riesz) spectral projection is given by 
\begin{equation}\label{def:proj}
Q_\varepsilon^{+}
=\frac{1}{2\pi i}\int_{\bar\Gamma} (I\lambda+\bar L_\varepsilon)^{-1}\,d\lambda,
\quad 0\le\varepsilon\le\varepsilon_{\bar\Gamma}.
\end{equation}

Using the resolvent estimate \eqref{ConvResolLin}, there exist
$C>0$ (independent of $\varepsilon$) such that
\begin{equation}\label{eq:QeQ0}
\bigl\|\,Q_\varepsilon^{+}E - E Q_{0}^{+}\,\bigr\|_{\mathcal L(X_0, X_\varepsilon^\alpha)}
\leq \frac{|\bar\Gamma|}{2\pi} \
\sup\limits_{\lambda\in\bar\Gamma}
\bigl\|(I\lambda+\bar L_\varepsilon)^{-1}E - E(I\lambda+\bar L_0)^{-1}\bigr\|_{\mathcal L(X_0, X_\varepsilon^\alpha)}
\leq C\varepsilon,
\end{equation}
(see \cite {flank, arrieta, carvalho,leonardo}).

\subsection{Rate of convergence and attraction of local unstable manifolds}

In this section, we analyze the continuity and approximation of local unstable manifolds associated with a hyperbolic equilibrium of \eqref{eq1}.

Let $u_{\varepsilon}^*$ be a hyperbolic equilibrium of \eqref{eq1} for each $\varepsilon \in [0, \varepsilon_0]$. Introducing the perturbation variable $\omega^\varepsilon := u^\varepsilon - u_\varepsilon^*$, we rewrite the system \eqref{eq1} as
\begin{equation}\label{eq:perturbed_shifted}
    \omega^{\varepsilon}_t + \bar{L}_\varepsilon \omega^\varepsilon  = \mathcal{F}_\varepsilon(\omega^\varepsilon),
\end{equation}
where $\mathcal{F}_\varepsilon(\omega^\varepsilon) := F_\ep(\omega^\varepsilon + u_\varepsilon^*) - F_\ep(u_\varepsilon^*) - F'_\ep(u_\varepsilon^*)\omega^\varepsilon.$

We consider the spectral decomposition of $X_\varepsilon^{\al}$ induced by the spectral projection $Q_\varepsilon^+$ associated with the unstable spectrum of $\bar{L}_\varepsilon$, as in \eqref{def:proj}. Consequently, we divide $\omega^\varepsilon = v^\varepsilon + z^\varepsilon$ with $v^\varepsilon := Q_\varepsilon^+ \omega^\varepsilon$ and $z^\varepsilon := (I - Q_\varepsilon^+)\omega^\varepsilon$, and decompose the equation \eqref{eq:perturbed_shifted} into the system
\begin{equation}\label{eq:unstable-stable-decomp}
\left\{
\begin{aligned}
    v^\ep_t + (\bar{L}_\varepsilon)^+ v^\varepsilon &= H_\varepsilon(v^\varepsilon, z^\varepsilon), \\
    z^\ep_t + (\bar{L}_\varepsilon)^- z^\varepsilon &= G_\varepsilon(v^\varepsilon, z^\varepsilon),
\end{aligned}
\right.
\end{equation}
where $(\bar{L}_\varepsilon)^\pm$ denote the restrictions of $\bar{L}_\varepsilon$ to $\mathrm{Ran}(Q_\varepsilon^+)$ and $\mathrm{Ker}(Q_\varepsilon^+)$, respectively.
The nonlinear terms in \eqref{eq:unstable-stable-decomp} are given by
\begin{align*}
H_\varepsilon(v^\varepsilon, z^\varepsilon) &:= Q^+_\varepsilon \mathcal{F}_\varepsilon(\omega^\varepsilon), \quad
G_\varepsilon(v^\varepsilon, z^\varepsilon) := (I - Q^+_\varepsilon) \mathcal{F}_\varepsilon(\omega^\varepsilon)
\end{align*}
and note that $H_{\varepsilon}(0, 0) = 0 = G_{\varepsilon}(0, 0)$. Moreover, the functions $H_{\varepsilon}$ and $G_{\varepsilon}$ are continuously differentiable, with vanishing derivatives at the origin, that is, 
$H_{\varepsilon}^{'}(0, 0) = 0 = G_{\varepsilon}^{'}(0, 0).$

Therefore, given any $\rho > 0$, there exists $\delta > 0$ such that, if $(v^{\varepsilon}, z^{\varepsilon}), (\tilde{v}^{\varepsilon}, \tilde{z}^{\varepsilon})\in B_{\delta}(0,0)$ and $\varepsilon \in (0, \varepsilon_0]$, we have
\begin{eqnarray}\label{eq:511}
	\ds \left\|H_{\varepsilon}(v^{\varepsilon}, z^{\varepsilon})\right\|_{X_{\varepsilon}^{\al}} &\!\!\leq\!\!& \rho, \quad 
	\left\|G_{\varepsilon}(v^{\varepsilon}, z^{\varepsilon})\right\|_{X_{\varepsilon}^{\al}} \,\leq\, \rho, \nonumber\\
	\left\|H_{\varepsilon}(v^{\varepsilon}, z^{\varepsilon}) - H_{\varepsilon}(\tilde{v}^{\varepsilon}, \tilde{z}^{\varepsilon}) \right\|_{X_{\varepsilon}^{\al}} &\!\!\leq\!\!& \rho\left(\left\|v^{\varepsilon}- \tilde{v}^{\varepsilon}\right\|_{X_{\varepsilon}^{\al}} + \left\|z^{\varepsilon}- \tilde{z}^{\varepsilon}\right\|_{X_{\varepsilon}^{\al}}\right), \label{eq620}\\
	\left\|G_{\varepsilon}(v^{\varepsilon}, z^{\varepsilon}) - G_{\varepsilon}(\tilde{v}^{\varepsilon}, \tilde{z}^{\varepsilon}) \right\|_{X_{\varepsilon}^{\al}} &\!\!\leq\!\!& \rho\left(\left\|v^{\varepsilon}- \tilde{v}^{\varepsilon}\right\|_{X_{\varepsilon}^{\al}} + \left\|z^{\varepsilon}- \tilde{z}^{\varepsilon}\right\|_{X_{\varepsilon}^{\al}}\right).\nonumber
\end{eqnarray}

Since our analysis focuses on the behavior of solutions near $(0, 0)$, we may truncate the nonlinearities $H_{\varepsilon}$ and $G_{\varepsilon}$ outside the ball $B_{\delta}(0 ,0)$ to ensure that the estimates in \eqref{eq620} hold globally.
Moreover, since $\bar{L}_{\varepsilon}$ is a sectorial operator, we can take suitable constants $\tilde{M}$, $\beta > 0$ (independent of $\varepsilon \in [0,\varepsilon_0]$), such that the following semigroup estimates hold
\begin{eqnarray}\label{eq:512}
	\ds \Vert e^{-t(\bar{L}_{\varepsilon})^+}v\Vert_{X_{\varepsilon}^{\al}} &\!\!\leq\!\!& \tilde{M} e^{\beta t}\left\|v\right\|_{X_{\varepsilon}^{\al}}, \quad t \leq 0,\nonumber\\
	\Vert e^{-t(\bar{L}_{\varepsilon})^-}z\Vert_{X_{\varepsilon}^{\al}} &\!\!\leq\!\!& \tilde{M} e^{-\beta t}\left\|z\right\|_{X_{\varepsilon}^{\al}}, \quad t \geq 0,\label{eq621}\\
	\Vert e^{-t(\bar{L}_{\varepsilon})^-}z\Vert_{X_{\varepsilon}^{\al}} &\!\!\leq\!\!& \tilde{M} t^{-\al}e^{-\beta t}\left\|z\right\|_{X_{\varepsilon}^{\al}}, \quad t > 0.\nonumber
\end{eqnarray}

These bounds enable us to construct the local unstable manifold of the equilibrium $u_\varepsilon^*$ as a graph over the unstable eigenspace of $\bar{L}_\varepsilon$, with Lipschitz dependence on the initial data.

\begin{theorem}\label{teo:unstable-manifold}
Let $u_0^*$ be a hyperbolic equilibrium of the limit problem \eqref{eq1}. Then, for any given constants $D > 0$, $\tilde{M} > 0$, $\Delta > 0$ and $0 < \nu < 1$, there exists $\ep_0>0$ and $\rho_0 > 0$ such that for all $0 < \rho \leq \rho_0$, the following inequalities are satisfied:
\begin{eqnarray*}
&& \rho \tilde{M} \beta^{-\al} \Gamma\left(\al\right) \leq D, \quad \rho \tilde{M}^2 (1+\Delta) \beta^{-\al} \Gamma\left(\al\right) \leq \Delta, \quad \frac{\beta}{2} \leq \beta - \rho \tilde{M} (1+\Delta) \leq \beta, \\
&& \rho \tilde{M} \beta^{-\al} \Gamma\left(\al\right) \!\!\left[\!1 +\!\frac{\rho \tilde{M} (1+\Delta)}{\beta - \rho \tilde{M} (1+\Delta)}\!\right]\! \leq \nu < 1, \quad \rho_1\! := \!\beta -\!\! \left[ \!\rho \tilde{M}\! - \!\frac{\rho^2 \tilde{M}^2 (1+\Delta)(1+\tilde{M})}{2\beta - \rho \tilde{M} (1+\Delta)} \!\right]\! > 0, \\
&& \rho \tilde{M} \beta^{-\al} \Gamma\left(\al\right) \left[1 + \rho \tilde{M} (1+\Delta) \beta^{-\al} \left(2\beta - \rho \tilde{M}(1+\Delta)\right)^{-\al} \right] \leq \tfrac{1}{2}
\end{eqnarray*}
and for the choice of $\rho$ above, let us assume that the nonlinearities $H_\ep$ and $G_\ep$ satisfy the global Lipschitz estimates in \eqref{eq620} for all $\ep \in [0,\ep_0]$.

Then, there exists a Lipschitz continuous map
$S_\varepsilon^* : Q_\varepsilon^+ X_\varepsilon^{\al} \to (I - Q_\varepsilon^+) X_\varepsilon^{\al}$
such that
\begin{eqnarray*}
\left\||S_{\varepsilon}^*|\right\|:=\sup\limits_{v^{\varepsilon}\in Q_{\varepsilon}^+X_{\varepsilon}^{\al}}\left\|S_{\varepsilon}^*(v^{\varepsilon})\right\|_{X_{\varepsilon}^{\al}}\leq D, \ \ \ \ \left\|S_{\varepsilon}^*(v^{\varepsilon}) - S_{\varepsilon}^*(\tilde{v}^{\varepsilon})\right\|_{X_{\varepsilon}^{\al}} \leq \Delta\left\|v^{\varepsilon}-\tilde{v}^{\varepsilon}\right\|_{Q_{\varepsilon}^+X_{\varepsilon}^{\al}},
\end{eqnarray*}
and the unstable manifold of $u_\varepsilon^*$ is given by the graph
$$W_\varepsilon^u := \left\{ \omega^\ep = v^\varepsilon + z^\varepsilon \in X_\varepsilon^{\al} : z^\varepsilon = S_\varepsilon^*(v^\varepsilon), \ v^\varepsilon \in Q_\varepsilon^+ X_\varepsilon^{\al} \right\}.$$

Moreover, for any $t \geq t_0$ and for orbits $(v^\varepsilon(t), z^\varepsilon(t))$ remaining sufficiently close to $u_\varepsilon^*$, we have the exponential attraction estimate
$$\left\| z^\varepsilon(t) - S_\varepsilon^*(v^\varepsilon(t)) \right\|_{X_\varepsilon^{\al}} \leq \tilde{M} e^{-\rho_1(t - t_0)} \left\| z^\varepsilon(t_0) - S_\varepsilon^*(v^\varepsilon(t_0)) \right\|_{X_\varepsilon^{\al}}.$$

Finally, for every $\theta \in \left(0, 1\right)$, there exists a constant $C > 0$ independent of $\varepsilon$ such that
\begin{equation}\label{eqn:estimate-Se-So}
    \left\||S_{\varepsilon}^*E - ES_0^*|\right\| \leq C \varepsilon^{\theta}.
\end{equation}
\end{theorem}
\demo
Following Henry’s framework \cite{henry}, we establish existence applying Banach's fixed-point theorem and (singular) Gronwall's inequality at various stages of the proof.

For $v^\ep,\tilde{v}^\ep \in Q_{\varepsilon}^+X_{\varepsilon}^{\alpha}$, consider the space
\[
\mathcal{X}_{\varepsilon} := \left\{ S_\varepsilon : Q_{\varepsilon}^+X_{\varepsilon}^{\alpha} \to (I - Q_{\varepsilon}^+)X_{\varepsilon}^{\alpha}: \ \||S_\varepsilon|\| \leq D, \ \|S_\varepsilon(v^\varepsilon) - S_\varepsilon(\tilde{v}^\varepsilon)\|_{X^\al_\ep} \leq \Delta \|v_\varepsilon - \tilde{v}_\varepsilon\|_{Q_{\varepsilon}^+X_{\varepsilon}^{\alpha}} \right\}.
\]
This space with the norm $\||\cdot|\|$ is a complete metric space.

We will divide the proof em four parts.

\vspace{0.3cm}

\noi \textit{Part 1. Existence}. Let $S_\varepsilon\in \mathcal{X}_{\varepsilon}$ and $v^{\varepsilon}(t)=\psi_{\varepsilon}(t, \tau, \eta, S_{\varepsilon})$ be the solution of
\begin{equation*}
\begin{cases} v_t^{\varepsilon} + (\bar{L}_{\varepsilon})^+v^{\varepsilon} = H_{\varepsilon}(v_{\varepsilon}, S_{\varepsilon}(v^{\varepsilon})), \ t < \tau \\  v^{\varepsilon}(t) = \eta,
\end{cases}
\end{equation*}
where $\eta \in Q_{\varepsilon}^+X_{\varepsilon}^{\al}$. We define the map $\Phi:\mathcal{X}_{\varepsilon}\to\mathcal{X}_{\varepsilon}$ by
\[
\Phi(S_\varepsilon)(\eta) := \int_{-\infty}^{\tau} e^{-(\bar{L}_\varepsilon)^-(\tau - s)} G_\varepsilon(v^\varepsilon(s), S_\varepsilon(v^\varepsilon(s)))\, \mathrm{d}s.
\]
Using the semigroup estimates for $(\bar{L}_\varepsilon)^\pm$ \eqref{eq:512}, the Lipschitz bounds for $G_\varepsilon$ \eqref{eq:511}, Gronwall's inequality and that the constants $D$, $\nu>1$ and $\Delta$ satisfy the conditions stated in the assumptions of this theorem, one shows that $\Phi$ maps $\mathcal{X}_\varepsilon$ into itself, provided that $\rho > 0$ is sufficiently small. Moreover, $\Phi$ is a contraction:
\[
\|\Phi(S_\varepsilon)(\eta) - \Phi(\tilde{S}_\varepsilon)(\tilde{\eta})\|_{X^\al_\ep} \leq \Delta \|\eta - \tilde{\eta}\|_{X^\al_\ep} + \nu \|S_\varepsilon - \tilde{S}_\varepsilon\|_{\mathcal{X}_\varepsilon}.
\]

Thus, there is a unique fixed point $S_\varepsilon^*=\Phi(S_\ep^*) \in \mathcal{X}_\varepsilon$ (for further details of the existence, see \cite{arrieta,elaine}).

\vspace{0.3cm}

\noindent\textit{Part 2. Invariance of the manifold.} We show that $W_{\varepsilon}^u := \left\{(v^{\varepsilon}, z^{\varepsilon}) \in X_{\varepsilon}^{\al} : z^{\varepsilon} = S_{\varepsilon}^*(v^{\varepsilon}),\ v^{\varepsilon} \in Q_{\varepsilon}^+X_{\varepsilon}^{\al} \right\}$ is an invariant manifold for the system \eqref{eq:unstable-stable-decomp}. Let an initial condition \((v^0, z^0) \in W_{\varepsilon}^u\), so that \(z^0 = S_\varepsilon^*(v^0)\), and let \(v_\varepsilon^*(t)\) be the solution of
\[
\begin{cases}
v^\ep_t + (\bar{L}_{\varepsilon})^+ v^{\varepsilon} = H_{\varepsilon}(v^{\varepsilon}, S_{\varepsilon}^*(v^{\varepsilon})), \\
v^{\varepsilon}(0) = v^0.
\end{cases}
\]
By uniqueness, this defines a trajectory \((v_\varepsilon^*(t), S_\varepsilon^*(v_\varepsilon^*(t)))\) in the graph of \(S_\varepsilon^*\). 

To verify invariance, it remains to show that \(z_\varepsilon^*(t) := S_\varepsilon^*(v_\varepsilon^*(t))\) solves the complementary equation $z^\ep_t + (\bar{L}_\varepsilon)^- z^\varepsilon = G_\varepsilon(v_\varepsilon^*(t), z^\varepsilon(t)).$

By the fixed point characterization of $S_\varepsilon^*$, we know that this is the unique bounded solution of the above equation backward in time, given by the variation of constants formula:
\[
z_\varepsilon^*(t) = \int_{-\infty}^{t} e^{-(\bar{L}_{\varepsilon})^-(t-s)} G_\varepsilon(v_\varepsilon^*(s), S_\varepsilon^*(v_\varepsilon^*(s)))\,\mathrm{d}s.
\]
Thus, $(v_\varepsilon^*(t), z_\varepsilon^*(t)) \in W_\varepsilon^u$ for all $t \in \mathbb{R}$. Hence, the trajectory remains in \(W_\varepsilon^u\) for all \(t\), proving that the graph of \(S_\varepsilon^*\) is invariant under the semiflow associated to \eqref{eq:unstable-stable-decomp}.

\vspace{0.3cm}

\noindent\textit{Part 3. Exponential attraction property.} Let $(v^\ep,z^\ep) \in X_\ep^\al$ be the solution of \eqref{eq:unstable-stable-decomp} and define $\xi_{\varepsilon}(t)=z^{\varepsilon}(t) - S_{\varepsilon}^*(v^{\varepsilon}(t))$ and  $y^{\varepsilon}(s,t)$, $s\leq t$, the solution  
\begin{equation*}
\begin{array}{rr}
\left\{ \begin{array}{lll} y^{\varepsilon}_t + (\bar{L}_{\varepsilon})^+y^{\varepsilon} = H_{\varepsilon}(y^{\varepsilon}, S^*_{\varepsilon}(y^{\varepsilon})), \ s \leq t \\  y^{\varepsilon}(t,t) = v^{\varepsilon}(t).\end{array}\right.
\end{array}
\end{equation*}
Using \eqref{eq620},  \eqref{eq621} and  Gronwall's inequality, we get
\begin{align}
\left\|y^{\varepsilon}(s,t) -v^{\varepsilon}(s)\right\|_{X_{\varepsilon}^{\al}} 
&\leq \rho \tilde{M}\int^{t}_{s}e^{-[\beta-\rho \tilde{M}(1+\Delta)](\theta-s)} \left\|\xi_{\varepsilon}(\theta)\right\|_{X_{\varepsilon}^{\al}}\mathrm{d}\theta, \ s\leq t.\label{eq625}
\end{align}
Now, consider $s\leq t_0\leq t$. Then, by  \eqref{eq625} and  Gronwall's inequality, we have 
\begin{eqnarray}
	\ds&&\hspace{-1.5cm}\left\|y^{\varepsilon}(s,t) -y^{\varepsilon}(s,t_0)\right\|_{X_{\varepsilon}^{\al}}  \leq \rho \tilde{M}^2\int_{t_0}^{t}e^{-[\beta-\rho \tilde{M}(1+\Delta)](\theta-s)}\left\|\xi_{\varepsilon}(\theta)\right\|_{X_{\varepsilon}^{\al}}\mathrm{d}\theta.\label{eq626}
\end{eqnarray}
\noindent Let us now estimate the term $\left\|\xi_{\varepsilon}(t)\right\|_{X_{\varepsilon}^{\al}}$. Note that 
\begin{eqnarray*}
\ds \xi_{\varepsilon}(t)\! -\! e^{-(\bar{L}_{\varepsilon})^-(t-t_0)}\xi_{\varepsilon}(t_0)\!\!\!\!\!\!\!\!\! && =z^{\varepsilon}(t) - S_{\varepsilon}^*(v^{\varepsilon}(t)) - e^{-(\bar{L}_{\varepsilon})^-(t-t_0)}\left[z^{\varepsilon}(t_0) - S_{\varepsilon}^*(v^{\varepsilon}(t_0))\right]\\ 
 &&= \int_{t_0}^{t}\!e^{-(\bar{L}_{\varepsilon})^-(t-s)}\left[G_{\varepsilon}(v^{\varepsilon}(s),z^{\varepsilon}(s))-G_{\varepsilon}(y^{\varepsilon}(s,t),S^*_{\varepsilon}(y^{\varepsilon}(s,t)))\right]\!\mathrm{d}s\\ && -\!\!\int_{-\infty}^{t}\!\!\! e^{-(\bar{L}_{\varepsilon})^-(t-s)}\!\left[ G_{\varepsilon}(y^{\varepsilon}(s,t),S^*_{\varepsilon}(y^{\varepsilon}(s,t)))\!-\!G_{\varepsilon}(y^{\varepsilon}(s,t_0),S^*_{\varepsilon}(y^{\varepsilon}(s,t_0)))\right]\!\mathrm{d}s.
\end{eqnarray*}
From \eqref{eq620} and \eqref{eq621}, we have 
\begin{align*}
\big\|\xi_{\varepsilon}(t) - e^{-(\bar{L}_{\varepsilon})^-(t-t_0)}\xi_{\varepsilon}(t_0)\big\|_{X_{\varepsilon}^{\al}}  
 &\leq \rho \tilde{M}(1+\Delta)\!\int_{t_0}^{t}e^{-\beta(t-s)}\left\|v^{\varepsilon}(s)-y^{\varepsilon}(s,t)\right\|_{X_{\varepsilon}^{\al}}\mathrm{d}s
  \end{align*} \begin{align*} &+ \rho \tilde{M}\!\int_{t_0}^{t}e^{-\beta(t-s)}\left\|\xi_{\varepsilon}(s)\right\|_{X_{\varepsilon}^{\al}}\mathrm{d}s+\rho \tilde{M}(1+\Delta)\!\int_{-\infty}^{t}\!\! e^{-\beta(t-s)}\!\left\|y^{\varepsilon}(s,t)-y^{\varepsilon}(s,t_0)\right\|_{X_{\varepsilon}^{\al}}\mathrm{d}s.\end{align*}
Using \eqref{eq625}, \eqref{eq626}, and Fubini's theorem, we get 
\begin{align*}
 \|\xi_{\varepsilon}(t) &- e^{-(\bar{L}_{\varepsilon})^-(t-t_0)}\xi_{\varepsilon}(t_0)\|_{X_{\varepsilon}^{\al}}  \leq \rho \tilde{M}\!\int_{t_0}^{t}\!e^{-\beta(t-s)}\!\left\|\xi_{\varepsilon}(s)\right\|_{X_{\varepsilon}^{\al}}\!\mathrm{d}s + \rho^2 \tilde{M}^2(1+\Delta)e^{-\beta t}\times \\
 & \int_{t_0}^{t} e^{-[\beta-\rho \tilde{M}(1+\Delta)]\theta}\left\|\xi_{\varepsilon}(\theta)\right\|_{X_{\varepsilon}^{\al}}\int_{t_0}^{\theta}e^{[2\beta-\rho \tilde{M}(1+\Delta)]s}\!\mathrm{d}s\mathrm{d}\theta\\ 
 & +\rho^2 \tilde{M}^3(1+\Delta)e^{-\beta t}\!\int_{t_0}^{t} e^{-[\beta-\rho \tilde{M}(1+\Delta)]\theta}\left\|\xi_{\varepsilon}(\theta)\right\|_{X_{\varepsilon}^{\al}}\int_{-\infty}^{t_0}e^{[2\beta-\rho \tilde{M}(1+\Delta)]s} \mathrm{d}s\mathrm{d}\theta.
\end{align*}
Thus, 
\begin{eqnarray*}
&&\left\|\xi_{\varepsilon}(t) - e^{-(\bar{L}_{\varepsilon})^-(t-t_0)}\xi_{\varepsilon}(t_0)\right\|_{X_{\varepsilon}^{\al}} \leq \left[\rho \tilde{M} + \tfrac{\rho^2\tilde{M}^2(1+\delta)}{2\beta-\rho \tilde{M}(1+\Delta)}\right] \int_{t_0}^{t}\!e^{-\beta(t-s)}\!\left\|\xi_{\varepsilon}(s)\right\|_{X_{\varepsilon}^{\al}}\!\mathrm{d}s\\ &&\hspace{3.65cm}+ \tfrac{\rho^2 \tilde{M}^3(1+\Delta)}{2\beta-\rho \tilde{M}(1+\Delta)}e^{-\beta (t-t_0)}\!\int_{t_0}^{t}\!e^{-[\beta-\rho \tilde{M}(1+\Delta)](\theta-t_0)}\left\|\xi_{\varepsilon}(\theta)\right\|_{X_{\varepsilon}^{\al}}\!\mathrm{d}\theta.
\end{eqnarray*}
Using \eqref{eq621} and $\beta -\rho \tilde{M}(1+\Delta) \leq \beta$, we have 
\begin{align*}
&e^{\beta(t-t_0)}\left\|\xi_{\varepsilon}(t)\right\|_{X_{\varepsilon}^{\al}} \leq \tilde{M}\left\|\xi_{\varepsilon}(t_0)\right\|_{X_{\varepsilon}^{\al}} + \left[\rho \tilde{M} + \tfrac{\rho^2\tilde{M}^2(1+\Delta)}{2\beta-\rho \tilde{M}(1+\Delta)}\right] \int_{t_0}^{t}\!e^{\beta(s-t_0)}\!\left\|\xi_{\varepsilon}(s)\right\|_{X_{\varepsilon}^{\al}}\!\mathrm{d}s\\ 
&\hspace{4.15cm}+ \tfrac{\rho^2 \tilde{M}^3(1+\Delta)}{2\beta-\rho \tilde{M}(1+\Delta)}\!\int_{t_0}^{t}\!e^{-[2\beta-\rho \tilde{M}(1+\Delta)](s-t_0)}e^{\beta (s-t_0)}\left\|\xi_{\varepsilon}(s)\right\|_{X_{\varepsilon}^{\al}}\!\mathrm{d}s\\ 
&\leq \tilde{M}\left\|\xi_{\varepsilon}(t_0)\right\|_{X_{\varepsilon}^{\al}} + \left[\rho \tilde{M} +\tfrac{\rho^2 \tilde{M}^2(1+\Delta)(1+\tilde{M})}{2\beta-\rho \tilde{M}(1+\Delta)}\right]\!\int_{t_0}^{t}\!e^{\beta(s-t_0)}\left\|\xi_{\varepsilon}(s)\right\|_{X_{\varepsilon}^{\al}}\!\mathrm{d}s.
\end{align*}
Again by Gronwall's inequality it follows that $\left\|\xi_{\varepsilon}(t)\right\|_{X_{\varepsilon}^{\al}} \leq \tilde{M}\left\|\xi_{\varepsilon}(t_0)\right\|_{X_{\varepsilon}^{\al}}e^{-\rho_1(t-t_0)},$ where $\rho_1 = \beta - \big[\rho \tilde{M} +\tfrac{\rho^2 \tilde{M}^2(1+\Delta)(1+\tilde{M})}{2\beta-\rho \tilde{M}(1+\Delta)}\big] > 0$ is independent of $\varepsilon$. 

Therefore, 
 \begin{eqnarray*}
	\ds \left\|z^{\varepsilon}(t) - S_{\varepsilon}^*(v^{\varepsilon}(t))\right\|_{X_{\varepsilon}^{\al}} \leq \tilde{M}e^{-\rho_1(t-t_0)}\left\|z^{\varepsilon}(t_0) - S_{\varepsilon}^*(v^{\varepsilon}(t_0))\right\|_{X_{\varepsilon}^{\al}}, \ t\geq t_0.
\end{eqnarray*}
Since the full orbit $\{T_\varepsilon(t)u_0^\varepsilon:\,t\in\mathbb R\}$ is bounded, letting $t_0\to-\infty$ in the last estimate gives
$z^\varepsilon(t)=S_\varepsilon^*(v^\varepsilon(t))$ for all $t\in\mathbb R$, that is, $T_\varepsilon(t)u_0^\varepsilon\in W_\varepsilon^u:=\mathrm{graph}\,S_\varepsilon^*$.
\medskip

\noindent\textit{Part 4. Estimate.} Now, we will prove the estimate \eqref{eqn:estimate-Se-So}. 
For $\eta \in  Q_0^+ X_0^{\al}$, we have
 \begin{align*}
 &\left\|S^*_{\varepsilon}(E\eta) - ES^*_{0}(\eta)\right\|_{X_{\varepsilon}^{\al}} \\ 
  &\leq \int_{-\infty}^{\tau}\big\|e^{-(\bar{L}_{\varepsilon})^-(\tau-s)}G_{\varepsilon}(v^{\varepsilon}(s),S^*_{\varepsilon}(v^{\varepsilon}(s)))-e^{-(\bar{L}_{\varepsilon})^-(\tau-s)}G_{\varepsilon}(Ev^{0}(s),ES_0^*(v^{0}(s)))\big\|_{X_{\varepsilon}^{\al}}\mathrm{d}s\\ 
   &+ \int_{-\infty}^{\tau}\big\|e^{-(\bar{L}_{\varepsilon})^-(\tau-s)}G_{\varepsilon}(Ev^{0}(s),ES_0^*(v^{0}(s)))-e^{-(\bar{L}_{\varepsilon})^-(\tau-s)}EG_{0}(v^{0}(s),S_0^*(v^{0}(s)))\big\|_{X_{\varepsilon}^{\al}}\mathrm{d} \end{align*}\begin{align*}
  & + \int_{-\infty}^{\tau}\big\|e^{-(\bar{L}_{\varepsilon})^-(\tau-s)}EG_{0}(v^{0}(s),S_0^*(v^{0}(s)))-Ee^{-(\bar{L}_0)^-(\tau-s)}G_{0}(v^{0}(s),S_0^*(v^{0}(s)))\big\|_{X_{\varepsilon}^{\al}}\mathrm{d}s\\ 
  &:= I_1+I_2+I_3.
 \end{align*}
Thus,
\begin{eqnarray}
&&\hspace{-1.5cm}I_1
 \leq \tilde{M}\rho \Big(\beta^{-\al}\Gamma\left(\al\right)\!\left\|\vert S^*_{\varepsilon}- S_0^*\vert\right\|\! + \!(1+\Delta)\!\!\!\int_{-\infty}^{\tau}\!\!\!\!\!e^{-\beta(\tau-s)}(\tau-s)^{-\al}\!\left\|v^{\varepsilon}(s)- Ev^{0}(s)\right\|_{X_{\varepsilon}^{\al}}\!\!\mathrm{d}s\Big)\label{eq627}
\end{eqnarray}
On the other hand,
\begin{equation}
	\hspace{-0.2cm}\ds I_2 \leq \!\tilde{M}\!\!\int_{-\infty}^{\tau}e^{-\beta(\tau-s)}(\tau-s)^{-\al}\left\| G_{\varepsilon}(Ev^{0}(s),ES_0^*(v^{0}(s)))\!-EG_{0}(v^{0}(s),S_0^*(v^{0}(s)))\right\|_{X_{\varepsilon}^{\al}}\mathrm{d}s.\label{eq:I2-start}
\end{equation}

Writing, with $v^0=v^0(s)$ and $S_0=S_0^*(v^0(s))$,
\[
G_{\varepsilon}(E v^{0},E S_0)-E G_{0}(v^{0},S_0)
= (I-Q_{\varepsilon}^+)\,\mathcal Z_{\varepsilon}
+ (Q_{\varepsilon}^+E-E Q_{0}^+)\,\mathcal R_0(v^0,S_0), \ \text{where}
\]
\[
\mathcal Z_{\varepsilon}\!\!
:= \!\!\bigl[F_\varepsilon(Ev^0+E S_0+u_\varepsilon^*)-F_\varepsilon(u_\varepsilon^*)
       -F_\varepsilon'(u_\varepsilon^*)(Ev^0+E S_0)\bigr]
   \!-\!E\bigl[F_0(v^0+S_0+u_0^*)-F_0(u_0^*)-F_0'(u_0^*)(v^0+S_0)\bigr],
\] and $\mathcal R_0(v^0,S_0):=F_0(v^0+S_0+u_0^*)\!-\!F_0(u_0^*)-F_0'(u_0^*)(v^0+S_0).$

By the $C^{1}-$regularity (uniform in $\varepsilon$) of the Nemitskii's maps on bounded sets (see Lemma \eqref{lema:nemytskii} and \cite[Lemma 6.7]{patricia}),
\begin{equation}\label{eq:C1theta-remainder}
\|\mathcal Z_{\varepsilon}\|_{X_{\varepsilon}}
\;\le\; C\,\bigl(\|u_{\varepsilon}^*-E u_0^*\|_{X_{\varepsilon}^{\alpha}}
      + \|F'_{\varepsilon}-E F'_0\|_{\mathcal L(X_{\varepsilon}^{\alpha},X_{\varepsilon})}\bigr)\,
      \|E v^0+E S_0\|_{X_{\varepsilon}^{\alpha}} .
\end{equation}
Using the already proved equilibrium convergence
\(\|u_{\varepsilon}^*-E u_0^*\|_{X_{\varepsilon}^{\alpha}}\le C\varepsilon\),
we improve it to \(\varepsilon^{\theta}\) via the standard interpolation inequality
\begin{equation}\label{eq:interp}
\|w\|_{X_{\varepsilon}}
\;\le\; C\,\|w\|_{X_{\varepsilon}^{\alpha}}^{\,1-\theta}\,
            \|w\|_{X_{\varepsilon}^{1}}^{\,\theta}
\qquad (0<\theta\le 1).
\end{equation}
Moreover, by the resolvent estimate (see \eqref{ConvResolLin}) and
interpolation between $X_0$ and $X_0^{\alpha}$, 
\begin{equation}\label{eq:proj-diff}
\|F'_{\varepsilon}-E F'_0\|_{\mathcal L(X_{\varepsilon}^{\alpha},X_{\varepsilon})}
\;\le\; C\,\varepsilon^{\theta}.
\end{equation}
Since $v^0,S_0$ stay in a bounded ball, \(\|E v^0+E S_0\|_{X_{\varepsilon}^{\alpha}}\le C\).
Hence, from \eqref{eq:C1theta-remainder}--\eqref{eq:proj-diff},
\begin{equation}\label{eq:Ze-theta}
\|(I-Q_{\varepsilon}^{+})\mathcal Z_{\varepsilon}\|_{X_{\varepsilon}^{\alpha}}
\;\le\; C\,\varepsilon^{\theta}.
\end{equation}
By the Riesz formula and the resolvent difference bound (see \eqref{ConvResolLin} and \eqref{eq:QeQ0}),
\begin{equation}\label{eq:QeQ0-theta}
\|Q_{\varepsilon}^{+}E-EQ_0^{+}\|_{\mathcal L(X_0,X_{\varepsilon}^{\alpha})}
\;\le\; C\,\varepsilon^{\theta},
\end{equation}here the exponent $\theta$ comes from interpolation (combining the $\mathcal{O}(\varepsilon)$ resolvent difference in $X_\varepsilon$ with uniform sectorial bounds yields an estimate in the fractional space $X_\varepsilon^\alpha$) and since \(\mathcal R_0(v^{0},S_0)\) is bounded by the \(C^{1,\theta}\) regularity of the nonlinearity
on a bounded set, \(\|\mathcal R_0(v^{0},S_0)\|_{X_0}\le C\). Therefore
\begin{equation}\label{eq:proj-term}
\|(Q_{\varepsilon}^{+}E-EQ_0^{+})\,\mathcal R_0(v^{0},S_0)\|_{X_{\varepsilon}^{\alpha}}
\;\le\; C\,\varepsilon^{\theta}.
\end{equation}
Combining \eqref{eq:Ze-theta} and \eqref{eq:proj-term} in \eqref{eq:I2-start} gives
\begin{equation}\label{eq629}
    I_2 \;\le\; C\,\varepsilon^{\theta}\!
\int_{-\infty}^{\tau}\! e^{-\beta(\tau-s)}(\tau-s)^{-\alpha}\,\mathrm{d}s
\;=\; C\,\varepsilon^{\theta}\,\beta^{\alpha-1}\Gamma(1-\alpha)
\;\le\; C\,\varepsilon^{\theta}.
\end{equation}

Consider $r=\tau-s$. Using the linear semigroup estimate (see Lemma \ref{lemma:ConvSemigruposLin}),
there exist $C,c>0$ independent of $\varepsilon$ and fixed $\theta\in(0,1)$ such that
\[
\bigl\|e^{-(\bar L_\varepsilon)^- r}E - E e^{-(\bar L_0)^- r}\bigr\|_{\mathcal L(X_0,X_\varepsilon^\alpha)}
\;\le\;
C\,e^{-c r}\,r^{-\alpha(1-\theta)-\theta}\,\varepsilon^\theta,\quad r>0.
\]
Hence, since $G_0(v^0(s),S_0^*(v^0(s)))$ is uniformly bounded in $X_0$ (the orbit stays in a bounded set and
$F_0\in C^2$ there),
\begin{align}\label{eq630}
I_3
&\le C\,\varepsilon^\theta \int_{0}^{\infty} e^{-c r}\, r^{-\alpha(1-\theta)-\theta}
\bigl\|G_0\bigl(v^{0}(\tau-r),S_0^*(v^{0}(\tau-r))\bigr)\bigr\|_{X_0}\,dr \nonumber\\
&\le C\,\varepsilon^\theta \int_{0}^{\infty} e^{-c r}\, r^{-\alpha(1-\theta)-\theta}\,dr \ \le C\,\varepsilon^\theta.
\end{align}

Therefore, using \eqref{eq627}, \eqref{eq629} and \eqref{eq630}, we have 
\begin{eqnarray*}
\ds\left\|S^*_{\varepsilon}(E\eta) \!-\! S^*_{0}(\eta)\right\|_{X_{\varepsilon}^{\al}} &\leq &\ds C\,\varepsilon^{\theta}+ \tilde{M}\rho \beta^{-\al}\Gamma(\al)\left\|\vert S^*_{\varepsilon}- S_0^*\vert\right\| +\\ \ds&+& \rho \tilde{M}(1+\Delta)\!\!\int_{-\infty}^{\tau}\!e^{-\beta(\tau-s)}(\tau-s)^{-\al}\!\left\|v^{\varepsilon}(s)- Ev^{0}(s)\right\|_{X_{\varepsilon}^{\al}}\mathrm{d}s.
\end{eqnarray*}

Now, let us estimate $\left\|v^{\varepsilon}(s)- Ev^{0}(s)\right\|_{X_{\varepsilon}^{\al}}$. Note that
\begin{eqnarray*}
&&\big\|v^{\varepsilon}(t) - Ev^{0}(t)\big\|_{X_{\varepsilon}^{\al}} \ \, \leq \big\|e^{-(\bar{L}_{\varepsilon})^+(\tau-s)}E\eta-Ee^{-(\overline{L}_0)^+(\tau-s)}\eta\big\|_{X_{\varepsilon}^{\al}}\mathrm{d}s\\ &&+ \int_{t}^{\tau}\big\|e^{-(\bar{L}_{\varepsilon})^+(t-s)}H_{\varepsilon}(v^{\varepsilon}(s),S^*_{\varepsilon}(v^{\varepsilon}(s)))-e^{-(\bar{L}_{\varepsilon})^+(t-s)}H_{\varepsilon}(Ev^{0}(s),ES_0^*(v^{0}(s)))\big\|_{X_{\varepsilon}^{\al}}\mathrm{d}s\\ &&+ \int_{t}^{\tau}\big\|e^{-(\bar{L}_{\varepsilon})^+(t-s)}H_{\varepsilon}(Ev^{0}(s),ES_0^*(v^{0}(s)))-e^{-(\bar{L}_{\varepsilon})^+(t-s)}EH_{0}(v^{0}(s),S_0^*(v^{0}(s)))\big\|_{X_{\varepsilon}^{\al}}\mathrm{d}s\\ &&+ \int_{t}^{\tau}\big\|e^{-(\bar{L}_{\varepsilon})^+(t-s)}EH_{0}(v^{0}(s),S_0^*(v^{0}(s)))-Ee^{-(\overline{L}_0)^+(t-s)}H_{0}(v^{0}(s),S_0^*(v^{0}(s)))\big\|_{X_{\varepsilon}^{\al}}\mathrm{d}s\\
&&\leq \ds C\,\varepsilon^{\theta}e^{\beta(t-\tau)} + \tilde{M}\rho\left\|\vert S^*_{\varepsilon}- S_0^*\vert\right\| \int_t^{\tau}e^{\beta(t-s)}\mathrm{d}s + \rho \tilde{M}(1+\Delta)\!\!\int_{t}^{\tau}\!e^{\beta(t-s)}\!\left\|v^{\varepsilon}(s)- Ev^{0}(s)\right\|_{X_{\varepsilon}^{\al}}\mathrm{d}s.
\end{eqnarray*}

Again, by Gronwall's inequality for $\phi_{\varepsilon}(t) = e^{\beta(\tau-t)}\!\left\|v^{\varepsilon}(s)- Ev^{0}(s)\right\|_{X_{\varepsilon}^{\al}}$, we have 
\begin{eqnarray*}
	\ds\left\|v^{\varepsilon}(t) - Ev^{0}(t)\right\|_{X_{\varepsilon}^{\al}} \ \leq \ds \left(C\,\varepsilon^{\theta}+ \tilde{M}\rho\beta^{-1}\left\|\vert S^*_{\varepsilon}- S_0^*\vert\right\|\right) e^{[\rho M(1+\Delta)-\beta](\tau-t)}.
\end{eqnarray*}
Therefore,
\begin{align*}
\left\|S^*_{\varepsilon}(E\eta) \!-\! S^*_{0}(\eta)\right\|_{X_{\varepsilon}^{\al}} \ &\leq \ds C\,\varepsilon^{\theta}+ \tilde{M}\rho \beta^{-\al}\Gamma(\al)\left\| \vert S^*_{\varepsilon}- S_0^*\vert \right\| +\\ &+ \rho \tilde{M}(1+\Delta)\!\!\int_{-\infty}^{\tau}\! e^{[\rho \tilde{M}(1+\Delta)-2\beta](\tau-s)}(\tau-s)^{-\al}\!\left(C\,\varepsilon^{\theta}+ \tilde{M}\rho\beta^{-1}\left\|\vert S^*_{\varepsilon}- S_0^*\vert\right\|\right)\mathrm{d}s. \\
& \leq \ds C\,\varepsilon^{\theta}+\rho \tilde{M}\beta^{-\al}\Gamma(\al)\left[1+\rho \tilde{M}\left(1+\Delta\right)\beta^{-\al}\left(2\beta - \rho \tilde{M}\left(1+\Delta\right)\right)^{-\al}\right]\left\|\vert S^*_{\varepsilon}- S_0^*\vert\right\|.
\end{align*}

\noindent Since, $\rho \tilde{M}\beta^{-\al}\Gamma(\al)\big[1+\rho \tilde{M}\left(1+\Delta\right)\beta^{-\al}\left(2\beta - \rho \tilde{M}\left(1+\Delta\right)\right)^{-\al}\big] \leq \tfrac{1}{2},$ we have  
\begin{align*}
\left\|S^*_{\varepsilon}(E\eta) \!-\! ES^*_{0}(\eta)\right\|_{X_{\varepsilon}^{\al}} \leq C\varepsilon^{\theta},
\end{align*}
for some constant $C>0$, independent of $\varepsilon$. This completes the proof.
\fimdemo

\begin{corolario}\label{cor533}
Assume the conditions of Theorem $\ref{teo:unstable-manifold}$ and $\mathcal{E}_0=\left\lbrace u^*_{0,1}, \cdot\cdot\cdot, u^*_{0,n_0}\right\rbrace$ where each $u^*_{0,i}$, $i=1,\cdot\cdot\cdot,n_0$ is hyperbolic. Then, there are $\delta>0$ and $\varepsilon_0 < 1$ such that, for all $\varepsilon \in [0, \varepsilon_0]$, $\mathcal{E}_{\varepsilon}=\left\lbrace u^*_{\varepsilon,1}, \cdot\cdot\cdot, u^*_{\varepsilon,n_0}\right\rbrace$, and its local unstable manifolds $W_{\varepsilon, \mathrm{loc}}^u(u^*_{\varepsilon,i})$, $i=1,\cdot\cdot\cdot,n_0$, behave continuously in $\varepsilon$ as $\varepsilon \rightarrow 0$. 
\end{corolario}

\subsection{Rate of convergence and attraction of attractors}

We now state our main theorem, which ensures persistence of hyperbolic equilibria and an explicit rate convergence of the perturbed attractors to the limit attractor.

\begin{theorem}\label{teoprincipal}
Let $\left\lbrace T_{\ep}(t): t \geq 0\right\rbrace$ be the gradient semigroup associated to the problem $\eqref{eq1}$ with global attractor $\mathcal{A}_{\ep}$. Let us assume that, in bounded sets of $X_{\varepsilon}^{\alpha}$, there are $\theta \in (0,1)$, $\varepsilon_0>0$, $\varrho>0$, $\beta >0$ and $C'>0$ such that
\[
\|T_{\varepsilon}(t) E - E T_0(t)\|_{X_{\varepsilon}^{\alpha}} 
\leq C' e^{\varrho t} t^{-\alpha(1 - \theta) - \theta}\varepsilon^{\theta},
\]
for all $t >0$ and $\varepsilon \in (0,\varepsilon_0]$. Furthermore, we assume that the equilibrium points of the limit problem, $\mathcal{E}_0=\left\lbrace u^*_{0,1}, \cdot\cdot\cdot, u^*_{0,n_0}\right\rbrace$ are finite and hyperbolic. Then, for all $\varepsilon \in (0, \varepsilon_0]$:
	
\noindent \textbf{$\bf(i)$} The semigroup $\left\lbrace T_{\varepsilon}(t): t \geq 0\right\rbrace$ has a finite number of equilibrium points $\mathcal{E}_{\varepsilon}=\left\lbrace u^*_{\varepsilon,1}, \cdot\cdot\cdot, u^*_{\varepsilon,n_0}\right\rbrace$ and $\mathcal{A}_{\varepsilon} = \bigcup\limits_{i=1}^{n_0}W_{\varepsilon}^u(u^*_{\varepsilon,i})$. Moreover, the family of attractors $\left\lbrace\mathcal{A}_{\varepsilon}\right\rbrace_{\varepsilon\in [0, \varepsilon_0]}$ is lower semicontinuous at $\varepsilon =0;$

\noindent \textbf{$\bf(ii)$} There exists $\rho>0$ such that, given $B\subset X_{\varepsilon}^{\al}$ bounded, there exists $C_1=C_1(B)>0$, such that $$\mathrm{dist}(T_{\varepsilon}(t)B, \mathcal{A}_{\varepsilon})\leq C_1\,e^{-\rho t}, \ \forall t \geq 1;$$

\noindent \textbf{$\bf(iii)$} There exists $C_2>0$ such that
\begin{align*}
\mathrm{dist}(\mathcal{A}_{\varepsilon}, \mathcal{A}_0) + \mathrm{dist}(\mathcal{A}_0, \mathcal{A}_{\varepsilon}) \leq
 C_2\,\varepsilon^{\frac{\theta\rho}{\rho+\varrho}},
\end{align*}
where $\mathrm{dist}(\cdot,\cdot)$ denotes the Hausdorff semi-distance.
\end{theorem}
\demo
Items (i)--(iii) follow, respectively, from Corollary \ref{cor533} together with \cite[Theorem 2.13]{alexandre1}, from \cite[Theorem 3.9]{alexandre2} combined with Theorem~\ref{teo:unstable-manifold}, and from \cite[Theorem 8.2.2]{Babin1992}, using the semigroup estimate assumed in the statement.
\fimdemo

\medskip

\noindent\textbf{Acknowledgments:} \textit{The authors would like to thank J. M. Arrieta for helpful discussions and valuable suggestions that greatly improved this work. The first author EATL was supported by Grants FAPESP \#2024/14305-2 (Brazil). The third author MCP was supported by Grants CNPq \#303561/2024-6 and FAPESP \#20/14075-6 (Brazil) and gratefully acknowledges the medical team led by Prof. Dr. G. Lepski (HC/FMUSP) for their outstanding care, which was instrumental in saving his life.}


\begin{thebibliography}{99}

\bibitem{AAB} Aragão, G. S.; Arrieta, J. M.; Bruschi, S. M. {\it Continuity of attractors of parabolic equations with nonlinear boundary conditions and rapidly varying boundaries. The case of a Lipschitz deformation}, J. Differential Equations, 429 (2025), 460--502.

\bibitem{patricia} Araujo P. N., Nakasato J., Pereira M. C. {\it A semilinear elliptic equation with homogeneous Neumann boundary conditions posed in thin domains with outward peaks}, to appear Revista Mat. Complutense, DOI:10.1007/s13163-025-00548-2.

\bibitem{flank} Arrieta, J. M.; Bezerra, F. D. M.; Carvalho, A. N. {\it Rate of convergence of global attractors of some perturbed reaction-diffusion problems}, Topol. Methods Nolinear Anal., 41 2 (2013), 229--253.

\bibitem{arrieta} Arrieta, J. M.; Carvalho, A. N. {\it Spectral convergence and nonlinear dynamics of reaction--diffusion equations under perturbations of the domain}, J. Differential Equations, 199 (2004), 143--178.


\bibitem{german1} Arrieta, J. M.; Carvalho, A. N.; Lozada-Cruz, G. {\it Dynamics in dumbbell domains III. Continuity of attractors}, J. Differential Equations, 247 (2009), 225--259.

\bibitem{ACPS} Arrieta, J. M.; Carvalho, A. N.; Pereira, M. C.; Silva, R. P. {\it Semilinear parabolic problems in thin domains with a highly oscillatory boundary}, Nonlinear Analysis 74 (2011) 5111--5132.



\bibitem{esperanza} Arrieta, J. M.; Santamaría, E. {\it Distance of attractors of reaction-difusion equations in thin domains}, J. Differential Equations, 263 (2017), 5459--5506.

\bibitem{esperanza1} Arrieta, J. M.; Santamaría, E. {\it Estimates on the distance of inertial manifolds}, Discrete and Continuous Dynamical Systems, A 34 (10) (2014), 3921--3944.


\bibitem{am2020} Arrieta, J. M.; Villanueva-Pesquera, M. {\it Elliptic and parabolic problems in thin domains with doubly oscillatory boundary}, Communications in Pure and Applied Analysis 19 (4) (2020), 1891--1914. 

\bibitem{Babin1992}
Babin, A.; Vishik, M. 
\emph{Attractors of Evolution Equations},
Elsevier, 1992.

\bibitem{barros} Barros, S. R. M.; Pereira, M. C., {\it Semilinear elliptic equations in thin domains with reaction terms concentrating on boundary}, Journal of Mathematical Analysis and Applications 441.1 (2016), 375-392.

\bibitem{campiti} Campiti, M.; Metafune, G. and Pallara P. {\it Degenerate Self-adjoint Evolution Equations on the Unit Interval}, Semigroup Forum, 57 (1998), 1--36.


\bibitem{alexandre2} Carvalho, A. N.; Cholewa, J. W. {\it Exponential global attractors for semigroups in metric spaces with applications to differential equations}, Ergodic Theory and Dynamical Systems, Cambridge, v. 31, n. 6, (2011), 1641--1667.

\bibitem{Carvalho2010}
Carvalho, A. N.; Cholewa, J.; Dlotko, T. 
{\it Equi-exponential attraction and rate of convergence of attractors for singularly perturbed evolution equations},
\emph{Proc. Roy. Soc. Edinburgh Sect. A}, 144(1) (2014), 13–51.

\bibitem{carvalho} Carvalho, A. N., Cholewa, J. W., Lozada-Cruz, G.; Primo, M. R. {\it Reduction of infinite dimensional systems to finite dimensions: compact convergence approach}, SIAM J. Math. Anal. Vol. 45, No. 2, 600--638.

\bibitem{alexandre1} Carvalho, A. N.; Langa, J. A. {\it An extension of the concept of gradient semigroups which is stable under perturbation}, J. Differential Equations, 246 (2009), 2646--2668.

\bibitem{A.N.Carvalho2010}
Carvalho, A. N.; Langa, J.; Robinson, J. C. 
\emph{Attractors for Infinite-Dimensional Non-Autonomous Dynamical Systems},
Springer, 2010.

\bibitem{Carvalho}
Carvalho, A. N.; Pires, L.
{\it Parabolic equations with localized large diffusion: rate of convergence of attractors},
\emph{Topol. Methods Nonlinear Anal.}, 53(1) (2019), 1–23.

\bibitem{leonardo} Carvalho, A. N.; Pires, L. {\it Rate of convergence of attractors for singularly perturbed semilinear problems}, J. Math. Anal. Appl. 452 (2017) 258--296.



\bibitem{Carvalho-Piskarev:06} Carvalho, A. N., Piskarev, S. {\it A general approximations scheme for attractors of abstract parabolic problems}, Numerical Functional Analysis and Optimization, v. 27, n. 7/8 (2006), 785--829.

\bibitem{cho} Cholewa, J. W.; Dlotko, T. {\it Global attractors in abstract parabolic problems}, Cambridge: Cambridge University Press (2000).

\bibitem{Daners} Daners, D. \emph{Domain perturbation for linear and semi-linear boundary value problems}, Handbook of differential equations: stationary partial differential equations,  Vol. VI, 1-81, Elsevier/North-Holland, Amsterdam, 2008.

\bibitem{davies} Davies E.B., {\it Heat kernels and spectral theory}, Cambridge: Cambridge University Press, (1989).


\bibitem{Hale1988}
Hale, J. K. 
\emph{Asymptotic Behavior of Dissipative Systems},
Math. Surveys Monogr., Vol. 25, Amer. Math. Soc., 1988.

\bibitem{haleetall}
Hale, J. K.; Raugel, G. 
{\it Lower semi-continuity of attractors of gradient systems and applications},
\emph{Ann. Mat. Pura Appl.}, 154 (1989), 281–326.

\bibitem{raugel} Hale, J. K.; Raugel, G. {\it Reaction--difusion equation in thin domains}, J. Math. Pures Appl., (9) 71 (1) (1992) 33--95.


\bibitem{henry} Henry, D. {\it Geometric Theory of Semilinear Parabolic Equations}, Berlin: Springer--Verlag (1981).


\bibitem{Henry_domain}
Henry, D. 
\newblock {\em Perturbation of the Boundary in Boundary-Value Problems of
  Partial Differential Equations}.
\newblock Cambridge University Press, 2005.

\bibitem{KOVARIK20191600}
Kova{\v r}\`ik, H.; Pankrashkin, K. {\it Robin eigenvalues on domains with peaks}, J. of Differential Equations, 267(3) (2019) 1600--1630.

\bibitem{baddomains} Maz’ya, V. G.; Poborchi, S. V. {\it Imbedding theorems for Sobolev spaces on domains with peak and on Hölder domains}, St. Petersburg Mathematical Journal, 18(4):583–605, 2007.

\bibitem{Pereira2007}
Pereira, A. L.; Pereira, M. C.
\newblock {\it Continuity of attractors for a reaction-diffusion problem with
  nonlinear boundary conditions with respect to variations of the domain}.
\newblock {\em Journal of differential equations}, 239:343--370, 2007.

\bibitem{mcp} Pereira M. C., \newblock {\it Parabolic problems in highly oscillating thin domains}. \newblock {\em Annali di Matematica Pura ed Applicata} v. 194 (2015) 1203--1244. 

\bibitem{PPires2024}
Pereira, M. C.; Pires, L. 
{\it Rate of convergence for reaction-diffusion equations with nonlinear Neumann boundary conditions and $C^1$ variation of the domain},
\emph{J. Evol. Equ.}, 24(5) (2024), 1--41.















\bibitem{elaine} Tavares-Lima, E. A.; Lozada-Cruz, G. J., {\it Dynamics of parabolic equations in domains with a small hole II. Continuity of the attractors}, Journal of Differential Equations, v. 395 (2024), 308--332.

\bibitem{triebel} Triebel H., {\it Interpolation theory, function spaces, differential operators}, Berlin: VEB Deutscher, (1978).

\bibitem{vainikkko} Vainikko, G. {\it Approximative methods for nonlinear equations (two approaches to the convergence problem)}, Nonlinear Analysis, Theory, Methods and Applications,  v. 2, n. 6 (1978), 647--687.




\bibitem{NV} Varchon, N. \emph{Domain perturbation and invariant manifolds}, J. Evol. Equ. 12 3 (2012) 547-569. 




\end{thebibliography}
\end{document}